\documentclass[12pt]{amsart}
\usepackage{geometry}
\usepackage{amsmath}
\usepackage{amssymb}
\usepackage{amsthm}
\usepackage{amsfonts, dsfont}
\usepackage{paralist}
\usepackage{graphics} 
\usepackage{epsfig} 
\usepackage{graphicx}  
\usepackage{epstopdf}
\usepackage{epstopdf}
\usepackage{verbatim}
\usepackage{mathrsfs}
\usepackage{mathtools}
\usepackage{pstricks}
\usepackage{relsize}
\usepackage{tikz}
\usetikzlibrary{matrix}
\usepackage{subcaption}
\usepackage{pgfplots}
\usepackage{enumitem}
\usepackage{ upgreek }
\usepackage[colorlinks=true]{hyperref}
\hypersetup{urlcolor=blue, linkcolor=blue, citecolor=red}
\usepackage{float}
\usepackage{bm}
\usepackage{appendix}
\usepackage[ruled]{algorithm2e}

\usepackage{xcolor}

\newcommand{\D}[1]{\mbox{\rm #1}} 
\newcommand{\dd}{\D{d}}

\DeclareMathOperator*{\argmin}{argmin}

\usepackage{soul}
\usepackage{color}

\usepackage[abbrev]{amsrefs}
\usepackage{amssymb}

\usepackage[all,cmtip]{xy}

\usepackage{tikz}
\usepackage{tikz-cd}

\numberwithin{equation}{section}

\theoremstyle{plain} 
\newtheorem{theorem}{Theorem}[section]
\newtheorem{lemma}[theorem]{Lemma}

\newtheorem{corollary}[theorem]{Corollary}
\newtheorem{proposition}[theorem]{Proposition}
\newtheorem{remark}[theorem]{Remark}
\newtheorem{definition}[theorem]{Definition}

\definecolor{ForestGreen}{RGB}{34,139,34}
\definecolor{ao(english)}{rgb}{0.0, 0.5, 0.0}

\allowdisplaybreaks

\usepackage{lineno}
\begin{document}

\title[Non-Convex Global Optimization as Optimal Stabilization]{Non-Convex Global Optimization as an Optimal Stabilization Problem: Convergence Rates}
\thanks{Y.H. has been supported by a Roth Scholarship from Imperial College London.
}

\author{Yuyang Huang}
\address{Yuyang Huang \newline 
Department of Mathematics, Imperial College London, South Kensington Campus SW72AZ London, UK
}
\email{\texttt{yuyang.huang21@imperial.ac.uk}}
\thanks{ 
}

\author{Dante Kalise}
\address{Dante Kalise \newline 
Department of Mathematics, Imperial College London, South Kensington Campus SW72AZ London, UK
}
\email{\texttt{d.kalise-balza@imperial.ac.uk}}
\thanks{
}

\author{Hicham Kouhkouh}
\address{Hicham Kouhkouh \newline %\indent
Department of Mathematics and Scientific Computing, NAWI, University of Graz, %\newline \indent
Heinrichstra{\ss}e 36, 8010, Graz, Austria
}
\email{\texttt{hicham.kouhkouh@uni-graz.at}}
\thanks{}

\date{\today}

\begin{abstract}
We develop a rigorous framework for global non-convex optimization by reformulating the minimization problem as a discounted infinite-horizon optimal control problem. For non-convex, continuous, and  possibly non-smooth objective functions with multiple global minimizers, where classical gradient-based methods lack global convergence guarantees, we establish explicit exponential convergence rates with computable constants.  
Our analysis proves (i) variational convergence of the value function of the optimal control problem, (ii) convergence in the objective function for the original problem, as well as (iii) pathwise convergence of optimal trajectories to the minimizer set under minimal structural assumptions that require neither convexity, differentiability, nor \L{}ojasiewicz-type conditions on the objective.  These quantitative results significantly strengthen the asymptotic theory developed in our previous work  
(arXiv:2511.10815). %for arXiv
Numerical experiments demonstrate the practical effectiveness of the approach on challenging non-convex problems.
\end{abstract}

\subjclass[MSC]{37N35, 90C26, 49L12, 35Q93}
% 37N35
% Dynamical systems in control 
% 35Q93
% PDEs in connection with control and optimization
% 49L12
% Hamilton-Jacobi equations in optimal control and differential games
% 90C26
% Nonconvex programming, global optimization 

\keywords{Global optimization, non-convex optimization, optimal control, rate of convergence, Hamilton--Jacobi--Bellman equation}

\maketitle

\section{Introduction}
Given a non-convex, merely continuous objective function $f:\mathbb{R}^n\to\mathbb{R}$, we are concerned with the global, unconstrained minimization problem
\begin{equation}\label{opt intro}
	\min\limits_{z\in\mathbb{R}^{n}}\, f(z),
\end{equation}
where $f$ admits a non-trivial set of global minimizers. This is a fundamental problem in applied mathematics.
In machine learning, objectives arising from deep neural network training, hyperparameter tuning, and generative modeling are typically high-dimensional, non-convex, and often non-smooth due to piecewise-linear activations or non-differentiable regularizers such as $\ell_1$ penalties \cite{li2018visualizing, choromanska2015loss, goodfellow2014generative}. Over-parameterized models frequently possess multiple equivalent global minimizers induced by symmetries and redundancies in the parameter space \cite{simsek2021geometry,brea2019weight}. Similar mathematical structures appear in trajectory optimization for robotics and aerospace \cite{kelly2017introduction}, inverse problems in imaging and signal processing \cite{nikolova2010fast}, and energy minimization in computational chemistry and molecular dynamics \cite{klepeis2003hybrid}.

Despite the prevalence and practical importance of these problems, theoretical guarantees for finding global optima remain elusive. In practice, gradient-based methods dominate owing to their scalability and low iteration cost \cite{bottou2018optimization}, yet they offer weak guarantees in non-convex settings: trajectories may stagnate at local minima or saddle points, and verifiable conditions for reaching global optima remain restrictive. This theory-practice gap persists even in low- or moderate-dimensional regimes where global solutions are both computationally accessible and practically meaningful.
In this paper, we address this gap by reformulating global optimization as an optimal control problem, yielding rigorous convergence guarantees under minimal structural assumptions on the objective function.

\subsection*{A control-theoretic approach to global optimization.}
Instead of directly tackling the minimization problem \eqref{opt intro} in its original form, where non-convexity and lack of smoothness limit the use of classical methods, we recast it into a control-theoretic framework. Building on previous works \cite{bardi2023eikonal,huang2025control}, this takes the form of the following discounted infinite-horizon optimal control problem:
\begin{equation}\label{ocp intro}
	\begin{aligned}
		u_{\lambda}(x) := \inf_{\alpha(\cdot)} \; & \int_{0}^{\infty} \left(\frac{1}{2}|\alpha(s)|^{2} + f(y_{x}^{\alpha}(s))\right) e^{-\lambda s}\,\mathrm{d}s\,,\qquad \lambda>0\,, \\
		\text{subject to} \quad & \dot{y}_{x}^{\alpha}(s) = \alpha(s), \quad y_{x}^{\alpha}(0)=x\in\mathbb{R}^n\,,
	\end{aligned}
\end{equation}
where $\alpha(\cdot):[0,\infty) \to \mathbb{R}^{n}$ is a measurable control, and $y_{x}^{\alpha}(\cdot)$ denotes the trajectory starting from initial position $x$ with continuous-time dynamics $\dot{y}_{x}^{\alpha}(s) = \alpha(s)$. The discount factor $\lambda>0$ ensures well-posedness of the infinite-horizon problem. The value function $u_{\lambda}(x)$ represents the optimal cost-to-go, balancing the control effort $\frac{1}{2}|\alpha(s)|^{2}$ and the objective function value $f(y_{x}^{\alpha}(s))$ along the trajectory. Moreover, $u_{\lambda}(\cdot)$ is the unique viscosity solution \cite{bardi1997optimal} of the Hamilton--Jacobi--Bellman (HJB) equation:
\begin{equation*}
	\lambda\, u_{\lambda}(x) + \max_{\alpha \in \mathbb{R}^n}\left\{ -\alpha \cdot Du_{\lambda}(x) - \frac{1}{2}|\alpha|^{2} \right\} = f(x) \quad \text{ in } \mathbb{R}^{n},
\end{equation*}
The optimal trajectories associated with \eqref{ocp intro} can be characterized as the gradient flow of the value function $u_\lambda(\cdot)$. Indeed, whenever $u_\lambda$ is differentiable, the optimal feedback control takes the form $\alpha^*(x) = -Du_\lambda(x)$, and the resulting closed-loop dynamics
\begin{equation*}
	\dot{y}(t) = -Du_\lambda(y(t)), \qquad y(0) = x,
\end{equation*}
are precisely gradient descent applied to $u_\lambda(\cdot)$ rather than to the original objective $f(\cdot)$. This is the key distinction from classical gradient-based methods: while gradient descent on $f$ may stagnate at local minima or saddle points, the value function $u_\lambda(\cdot)$ acts as a globally-informed surrogate whose gradient field steers trajectories toward the global minimizer set of $f$.

This formulation allows us to capture the optimization landscape through the lens of the value function $u_{\lambda}(\cdot)$, which acts as a surrogate for the original objective $f(\cdot)$. Although $u_{\lambda}(\cdot)$ is not guaranteed to be convex, it nonetheless provides an effective means to navigate the objective landscape of $f(\cdot)$ and reach global minimizers despite non-convexity of $f(\cdot)$. Figure~\ref{fig: illustration} offers a glimpse into this mechanism, comparing classical gradient-based methods with our controlled dynamics on a prototypical non-convex landscape. The visualization reveals how the value function acts as a globally informative surrogate, with its associated gradient field steering trajectories directly toward the global minimizer even when standard methods become trapped. A detailed derivation of this control formulation and its properties is provided in Section~\ref{sec: control prob}.
\begin{figure}[!ht]
	\centering
	\begin{subfigure}[t]{\textwidth}
		\centering
		\includegraphics[width=0.8\linewidth]{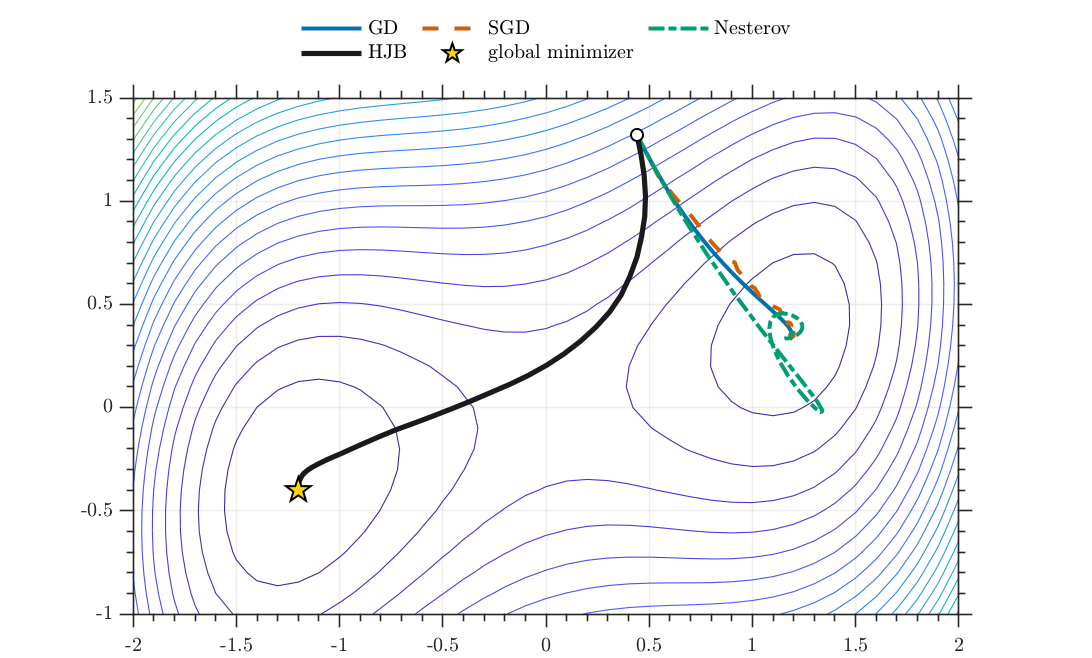}
		\caption{Contours of the objective function $f(z_1, z_2)=(z_1^2-1.44)^2+4(z_2-0.3 z_1)^2+0.15 z_1,$ with trajectories produced by gradient descent (GD), stochastic gradient descent (SGD), Nesterov's accelerated gradient method, and the proposed HJB flow $\dot{y} = -Du_\lambda(y)$, all initialized from the same starting point. The yellow star indicates the global minimizer.}
	\end{subfigure}
	\begin{subfigure}[t]{\textwidth}
		\centering
		\includegraphics[width=0.7\linewidth]{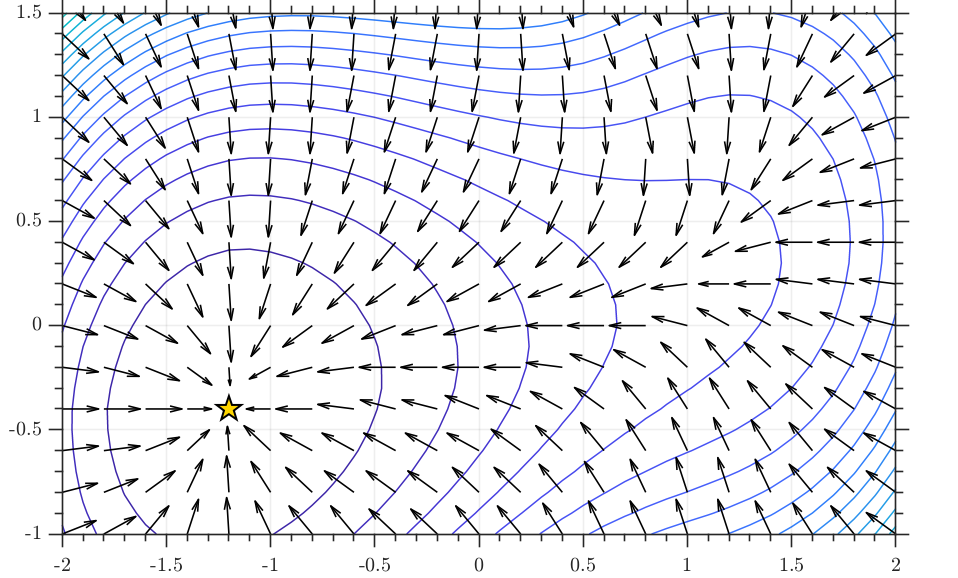}
		\caption{Contours of the value function $u_\lambda(\cdot)$ and associated optimal feedback field. The value function provides a smoother and unimodal landscape relative to the original objective, and its gradient field steers the dynamics toward the global minimizer, avoiding entrapment in local minima.}
	\end{subfigure}
	\caption{Comparison of classical gradient-based optimization dynamics and the HJB flow on a non-convex two-dimensional test objective with two wells and a unique global minimizer.}
	\label{fig: illustration}
\end{figure}

\subsection*{Main results: exponential convergence without convexity}
Our analysis establishes \emph{variational}, \emph{objective} and \emph{pathwise} convergence results with explicit exponential rates, significantly strengthening the asymptotic guarantees provided in \cite{huang2025control}. 
Let us define $y^{*}_{x}(\cdot)$ as an optimal trajectory for the problem \eqref{ocp intro}, where $x\in\mathbb{R}^{n}$ is arbitrarily fixed. Our main findings can be summarized as follows:
\begin{enumerate}
	\item Variational convergence: 
	\begin{equation*}
		\lambda\,u_{\lambda}(y_{x}^{*}(t)) \;\xrightarrow[t\to\infty]{}\; \min\limits_{\mathbb{R}^n} f\,,\; \text{ exponentially fast, }\quad \forall\,x\in\mathbb{R}^{n}, \; \forall\,t\geq 0.
	\end{equation*}
	\item Objective convergence: 
	\begin{equation*}
		f\!\left(y_x^*(t)\right)\xrightarrow[t\to\infty]{}\; \min\limits_{\mathbb{R}^n} f\,,\; \text{ exponentially fast, }\quad \forall\,x\in\mathbb{R}^{n}, \; \forall\,t\geq 0.
	\end{equation*}
	\item Pathwise convergence: There exist $\tau,\delta>0$ such that 
	\begin{equation*}
		\operatorname{dist}(y_{x}^{*}(t),\mathfrak{M})^{2} 
		\;  \lesssim \; e^{-\delta(t-\tau)}, \quad \quad \forall\,t\geq\tau\,,
	\end{equation*}
	and
	\begin{equation*}
		|\dot{y}_{x}^{*}(t)|^{2} 
		\;  \lesssim \; \,e^{-\delta(t-\tau)} \quad \text{ for a.a. } t\geq \tau\,,
	\end{equation*}
	where $\mathfrak{M}$ is the set of global minimizers of $f(\cdot)$. The notation ``\;$\lesssim$\;'' means ``\;$\leq$\;'' up to a multiplicative constant, and ``\textit{a.a.}'' stands for ``\textit{almost all}''.
\end{enumerate}
Crucially, these results also hold for \emph{quasi-optimal} trajectories, i.e. trajectories achieving optimal cost up to some small error. This robustness is essential for numerical approximation, where exact optimality is computationally prohibitive. The pathwise convergence result (3) moreover reveals a Turnpike-type property with deep connections to controllability theory that we explore in Section~\ref{sec: app opt}.

It is also worth emphasizing that the resulting trajectories and convergence guarantees are obtained through deterministic control principles, rather than through randomized search or sampling strategies commonly used in non-convex optimization. In this sense, the method remains fully analytical and its convergence properties can be rigorously established through the lens of optimal control theory and Hamilton--Jacobi--Bellman equations.

\subsection*{Contributions}
The present work advances a control-theoretic approach to non-convex optimization by establishing rigorous convergence guarantees that go substantially beyond existing results. Our main contributions can be summarized as follows:

\begin{itemize}
	\item \textbf{Explicit exponential convergence rates:} We provide quantitative exponential convergence rates for variational, objective and pathwise convergence in the non-convex, multi-minimizer setting. Our results establish precise decay rates with explicitly computable constants, complementing the asymptotic analysis in \cite{huang2025control}.
	
	\item \textbf{Minimal structural assumptions:} The proposed methodology requires neither convexity, nor differentiability, nor \L{}ojasiewicz-type conditions. This is achieved through the control-theoretic reformulation, which transforms the intractable non-convex landscape into a stabilization problem where global convergence can be rigorously established.
	
\item \textbf{Robustness to quasi-optimality:} We extend all convergence  guarantees to quasi-optimal trajectories, providing explicit error estimates  that quantify how suboptimality in the control propagates to the convergence  rates. This is essential for practical implementation, where the value function $u_\lambda(\cdot)$ cannot be approximated to arbitrary accuracy in high-dimensional settings due to the curse of dimensionality.
	
\item \textbf{Connections to classical optimization theory:}
	We establish precise relationships between the assumptions required by our control-theoretic framework and classical optimization conditions such as the Polyak--\L{}ojasiewicz (PL) inequality, local strong convexity, and metric regularity. Our key structural assumption for variational convergence is analogous to a PL condition on the \emph{value function} $u_{\lambda}(\cdot)$ rather than on the original objective $f(\cdot)$. We demonstrate that when $f(\cdot)$ exhibits locally linear or quadratic growth near its global minimizers, a mild geometric property satisfied by many non-convex objectives, then the value function $u_{\lambda}(\cdot)$ automatically satisfies the required condition. Notably, these local growth properties are substantially weaker than classical assumptions such as local strong convexity or smoothness on $f(\cdot)$, which are sufficient but not necessary for our framework.
\end{itemize}

\subsection*{Related Work}
To place our contributions in context, we survey related work connecting optimization to dynamical systems, optimal control, and Hamilton--Jacobi equations.
Over the last decades, there has been growing interest in the interplay between optimization and dynamical systems. Modeling classical iterative optimization algorithms as discretizations of suitable dynamical systems has led to valuable insights into convergence and stability, as demonstrated in \cite{alvarez2002second,attouch2019rate,su2016differential,siegel2024qualitative,attouch2000heavy,muehlebach2019dynamical,polyak2017lyapunov,franca2018admm,siegel2019accelerated,shi2022understanding}. However, these results are typically obtained under assumptions of smoothness and some form of convexity, including standard convexity, strong convexity or quasi-convexity, imposed on the objective function. Even in the non-convex setting, available results often rely on \L{}ojasiewicz-type conditions, as in \cite{apidopoulos2022nonconvex,aujol2023convergence}. Recent work has also explored variational formulations of optimization algorithms that connect to optimal control frameworks, such as \cite{tzen2023variational} for mirror descent, though still under convexity assumptions.

In general non-convex regimes, global convergence results are typically guaranteed only for critical points (not necessarily optimal), while methods specifically designed for global optimization are largely metaheuristic, such as Ant Colony Optimization \cite{dorigo2006ant}, Particle Swarm Optimization \cite{wang2018particle}, or Simulated Annealing \cite{bertsimas1993simulated}. A prominent example is Consensus-Based Optimization (CBO), a class of derivative-free, interacting particle methods for global minimization of non-convex functions. Assuming the existence of a unique global minimizer, a comprehensive global convergence result with an explicit convergence rate is established for the finite-particle, discrete-time dynamics \cite{fornasier2024consensus,fornasier2022convergence}, whereas the case of multiple global minimizers is treated in \cite{huang2025faithful}. Of particular relevance to the present work is the controlled-CBO dynamics introduced in \cite{huang2024fast}, where the standard CBO drift is modified by an optimal feedback term computed from a numerical approximation of an associated HJB value function.

Parallel to these stochastic approaches, a fully deterministic line of research formulates optimization as optimal control problems governed by Hamilton--Jacobi or Eikonal-type equations. Unlike heuristic methods, this perspective avoids randomness entirely and is grounded in analytical control theory. In \cite{bardi2023eikonal}, global asymptotic convergence to the set of global minimizers is obtained in the vanishing-discount limit $\lambda \to 0$ by analyzing a stationary Eikonal equation. A closely related and now active direction connects Moreau envelopes and proximal operators to viscous Hamilton--Jacobi equations via the Hopf--Lax formula and artificial viscosity \cite{osher2023hamilton,heaton2024global}. In this context, exponential convergence rates for convex, lower semi-continuous functions under generalized Polyak--\L{}ojasiewicz inequalities have been established \cite{apidopoulos2022nonconvex}. We refer to \cite{huang2025control} for a more complete literature review on PDE-based optimization methods.

\subsection*{Organization}
This manuscript is organized as follows. \textbf{Section \ref{sec: control prob}} introduces the optimal control formulation and establishes the foundational properties of the value function. \textbf{Section \ref{sec: conv rate}} establishes the variational and objective convergence. The results are formulated for both optimal and quasi-optimal trajectories, with explicit convergence rates and error estimates in the quasi-optimal case. \textbf{Section \ref{sec: app opt}} extends these results to pathwise convergence. In particular, we prove that (quasi-)optimal trajectories generated by the control problem converge exponentially fast to the set of global minimizers. As a byproduct, we obtain a Turnpike-type property and highlight connections with controllability. \textbf{Section \ref{sec: comparison}} relates our framework to classical optimization conditions such as the Polyak--\L{}ojasiewicz inequality, local strong convexity, and metric regularity, concluding with examples that demonstrate the broad applicability of our assumptions. \textbf{Section \ref{sec: numerics}} presents numerical experiments that support and illustrate the theoretical analysis, demonstrating the practical effectiveness of the control-based approach in computing global minimizers for challenging non-convex problems with multiple local minima. In \textbf{Section \ref{sec: summary}}, we summarize the main results and outline possible future directions. Finally, \textbf{Appendix \ref{app: nonsmooth}}, \textbf{\ref{appendix: control}}, and \textbf{\ref{app: ricc}} collect auxiliary definitions and provide additional proofs for completeness and the reader's convenience.

\section{From optimization to optimal stabilization}\label{sec: control prob}

In the following, we present a  global-optimization-as-optimal-control approach. Consider a control system where the control variable $\alpha(\cdot)$ governs the state trajectory  $y_{x}^{\alpha}(\cdot)\in \mathbb{R}^n$ through the  simple dynamics
\begin{equation*}
\dot{y}_{x}^{\alpha}(t)=\alpha(t), \quad y_{x}^{\alpha}(0)=x, \quad t\geq 0,
\end{equation*}
where $x\in\mathbb{R}^n$ is a given initial condition  and the control $\alpha(\cdot)$ belongs to an admissible set $\mathcal{A}$, which is defined by 
\begin{equation}\label{admissible control set}
    \mathcal{A} := \{\, \alpha(\cdot):[0,\infty)\to \mathbb{R}^{n}\, :\, \text{measurable}, \;|\alpha(s)| \leq M, \; \forall\,s\in [0,\infty)\,\},
\end{equation}
where $M>0$ is a large fixed constant (see Remark~\ref{rem: bounded controls}). 
The trajectory $ y_{x}^{\alpha}(\cdot)$ associated with an admissible control $\alpha(\cdot)\in\mathcal{A}$ is called an admissible trajectory. 
For a given objective $f(\cdot):\mathbb{R}^n\rightarrow\mathbb{R}$ and discount factor $\lambda>0$, to quantify the performance of the control, we consider the  running cost functional defined on $\mathbb{R}^{n}\times\mathcal{A}$ as
\begin{equation}\label{cost function}
    \mathscr{J}(x,\alpha(\cdot)) := \int_{0}^{\infty} \left( \frac{1}{2}|\alpha(s)|^{2} + f(y_{x}^{\alpha}(s)) \right)\,e^{-\lambda\,s}\,\dd s. 
\end{equation}
The quadratic term $\frac{1}{2}|\alpha(\cdot)|^2$ is a control penalization which regularizes the effort required to navigate the landscape.  The term $f(\cdot)$ evaluates the objective along the trajectory and rewards reaching and staying in regions with low objective  value. 

Finally,  the optimal control is obtained by solving the time-homogeneous, discounted infinite horizon optimal control problem (see \cite{carlson2012infinite} for background and general results)
\begin{equation}\label{OCP}
\begin{aligned}
     u_{\lambda}(x) & = 
    \inf\limits_{\alpha(\cdot)\in \mathcal{A}} 
    \quad \mathscr{J}(x,\alpha(\cdot))\\
     \text{subject  to  }\; \dot{y}_{x}^{\alpha}(s) & = \alpha(s),\quad y_{x}^{\alpha}(0)=x\in\mathbb{R}^n,\\
     \text{and the controls } &\alpha(\cdot):[0,\infty) \to B_{M} \text{ are measurable},
\end{aligned}
\end{equation}
where $B_{M}$ is the closed ball in $\mathbb{R}^{n}$ of radius $M$. 
The optimal value function is also a unique viscosity solution \cite{bardi1997optimal} of the Hamilton-Jacobi-Bellman (HJB) equation 
\begin{equation}\label{eq: HJB}
    \lambda\, u_{\lambda}(x) + \frac{1}{2}|Du_{\lambda}(x)|^{2} = f(x) \quad \quad \text{ in } \mathbb{R}^{n}.
\end{equation}
We now introduce our first main assumption.

\begin{enumerate}[label=\textbf{(A\arabic*)}]
\item\label{f: nice} $f : \mathbb{R}^n\to\mathbb{R}$ is  continuous and bounded that is
    \begin{equation*}
    \exists\;\underline{f},\,\overline{f}\; \text{ such that }\; \underline{f} \leq f(x) \leq \overline{f},\quad \forall\;x\in \mathbb{R}^n, \quad \text{where } \; \underline{f}:= \min\limits_{\mathbb{R}^{n}}f.
    \end{equation*}
\end{enumerate}

\begin{remark}[Bounded controls without loss of generality]\label{rem: bounded controls}
Fix $M>0$ and consider the discounted infinite-horizon problem \eqref{OCP}--\eqref{admissible control set} with admissible controls $\alpha(\cdot):[0,\infty)\to B_M$ measurable. Its value function $u_{\lambda,M}$ is characterized as the unique viscosity solution of the stationary HJB equation
\begin{equation}\label{hjb rem B}
    \lambda\,u_{\lambda,M}(x) + \max\limits_{\alpha\in B_{M}}\left\{-\alpha\cdot Du_{\lambda,M}(x) - \frac{1}{2}|\alpha|^{2}\right\} = f(x) \quad \text{ in } \mathbb{R}^{n}.
\end{equation}
We refer to \cite{bardi1997optimal}, in particular Chapter III, Section 2, Page 104 therein.
In order to compute the maximization in \eqref{hjb rem B}, we need to project over the Euclidean ball $B_M$, that is for any $p\in \mathbb{R}^n$:
\begin{equation}\label{ball_sup_formula}
\begin{aligned}
    \sup\limits_{\alpha\in B_{M}}\left\{ - \alpha \cdot p - \frac{1}{2}|\alpha|^{2}\right\}
    & =
    \left\{\;
    \begin{aligned}
        & \frac{1}{2}|p|^{2}, \quad \text{if } |p|\leq M,\\
        & M\,|p| - \frac{1}{2}M^2, \quad \text{if } |p|>M,
    \end{aligned}
    \right.
    \\
    & = \frac{1}{2}|p|^{2} - \frac{1}{2}(|p| - M)^{2}_{+}.
\end{aligned}
\end{equation}
In particular, whenever $|p|\leq M$, we have $(|p| - M)_{+} = 0$, and the maximizer is $\alpha^\ast=-p$.

Now consider the same control problem but with \emph{unbounded} measurable controls $\alpha(\cdot):[0,\infty)\to\mathbb{R}^n$; denote its value by $u_{\lambda,\infty}$. Then the resulting optimal control problem is characterized by the HJB equation
\begin{equation}\label{hjb rem R}
\begin{aligned}
&\lambda\,u_{\lambda,\infty}(x)+\sup_{\alpha\in\mathbb{R}^n}\Big\{-\alpha\cdot Du_{\lambda,\infty}(x)-\tfrac12|\alpha|^2\Big\}=f(x) \quad \text{ in } \mathbb{R}^{n}\\
&\Longleftrightarrow\quad
\lambda\,u_{\lambda,\infty}(x)+\tfrac12|Du_{\lambda,\infty}(x)|^2=f(x) \quad \text{ in } \mathbb{R}^{n}.
\end{aligned}
\end{equation}

Under Assumption \ref{f: nice}, \cite[Lemma 2.1(iii)]{huang2025control} gives the uniform bound
\begin{equation*}
|Du_{\lambda,\infty}(x)|\le R:=\sqrt{6\|f\|_\infty}\qquad\text{for all }x\in\mathbb{R}^n.
\end{equation*}
Hence, for any $M\ge R$, we are always in the regime $|p|\le M$ when $p=Du_{\lambda,\infty}(x)$, and by \eqref{ball_sup_formula}, the ball-supremum coincides with the full-space supremum:
\begin{equation*}
\sup_{\alpha\in B_M}\Big\{-\alpha\cdot Du_{\lambda,\infty}(x)-\tfrac12|\alpha|^2\Big\}
=\tfrac12|Du_{\lambda,\infty}(x)|^2
=\sup_{\alpha\in\mathbb{R}^n}\Big\{-\alpha\cdot Du_{\lambda,\infty}(x)-\tfrac12|\alpha|^2\Big\}.
\end{equation*}
Therefore $u_{\lambda,\infty}$ solves \eqref{hjb rem B} (with $M\ge R$), and by uniqueness of viscosity solutions to \eqref{hjb rem B} we obtain
\begin{equation*}
u_{\lambda,M}\equiv u_{\lambda,\infty}\qquad\text{for every }M\ge R.
\end{equation*}
In particular, once $M>\sqrt{6\|f\|_\infty}$ is fixed, restricting controls to $B_M$ does not change the value function, and the bounded- and unbounded-control formulations of \eqref{OCP} are equivalent. In what follows we therefore work with bounded controls for a fixed $M>R$, which simplifies the presentation without loss of generality.
\end{remark}

\subsection{Preliminaries}
We now establish basic properties of $u_{\lambda}(\cdot)$ that will underpin our convergence analysis. Our first result identifies a fundamental connection between the value function and global minimizers of $f(\cdot)$.

\begin{proposition}
\label{prop: min_equality}
Let Assumption \ref{f: nice} hold. The value function $u_{\lambda}(\cdot)$ satisfies $\underline{f}\leq \lambda\, u_{\lambda}(x)$ for all $x$. Moreover, $\lambda\, u_{\lambda}(x) = \underline{f}$ if and only if $x\in \mathfrak{M}$.
\end{proposition}

The proof can be found in Appendix \ref{appendix: control}. This result establishes that the minimizers of 
$f$ and $u_{\lambda}$ coincide. 
Having established this equivalence, we now investigate how optimality propagates along trajectories. To this end, we rely on the Dynamic Programming Principle (DPP), a cornerstone of optimal control theory that connects the value function to the optimal cost-to-go and hence to the evolution of optimal trajectories.

\begin{proposition}\label{prop: DPP}
The value function $u_{\lambda}(\cdot)$ satisfies the Dynamic Programming Principle (DPP)
\begin{equation}\label{DPP}
    u_{\lambda}(x) = \inf\limits_{\alpha(\cdot)\in \mathcal{A}} \left\{\, \int_{0}^{t} \left(\frac{1}{2}|\alpha(s)|^{2} + f(y_{x}^{\alpha}(s))\right)\, e^{-\lambda s}\;\dd s \, + \, u_{\lambda}(y_{x}^{\alpha}(t))\,e^{-\lambda t} \,\right\}.
\end{equation}
\end{proposition}
The proof can be found in Appendix \ref{appendix: control}.

Let $x\in \mathbb{R}^{n}$ be fixed. For any $\alpha(\cdot)\in\mathcal{A}$, define the function $h:[0,\infty)\to \mathbb{R}$ by
\begin{equation*}
    h(t) := \int_{0}^{t} \left(\frac{1}{2}|\alpha(s)|^{2} + f(y_{x}^{\alpha}(s))\right)\, e^{-\lambda s}\;\dd s \, + \, u_{\lambda}(y_{x}^{\alpha}(t))\,e^{-\lambda t}.
\end{equation*}
A direct consequence of the DPP \eqref{DPP} is the existence of a monotonic quantity along trajectories.

\begin{proposition}\label{Prop: h}
The function $t\mapsto h(t)$ is non-decreasing for all $\alpha(\cdot)\in \mathcal{A}$, and it is constant if and only if $\alpha(\cdot)$ is optimal.
\end{proposition}

The proof can be found in Appendix \ref{appendix: control}. The monotonicity of $h$ reflects the optimality principle along trajectories and will be an essential ingredient in the proofs of convergence and stability later.

With this analytical foundation in place, we can now synthesize the optimal control law in a closed-loop feedback form. The next theorem reveals that the optimal trajectories essentially perform a generalized gradient descent on the value function $u_\lambda$.

\begin{theorem}\label{thm: opt cont}
Let Assumption \ref{f: nice} be satisfied. 
The curve starting from $x$ at $t=0$, and governed by $\dot{y}^{*}_{x}(t) = -Du_{\lambda}(y^{*}_{x}(t))$ for almost every $t\geq 0$, is an optimal trajectory for the control problem \eqref{OCP}.
\end{theorem}

The proof can be found in Appendix \ref{appendix: control}. This result shows that the optimal trajectories associated with the control problem \eqref{OCP} are characterized by the closed-loop dynamics driven by the feedback field $-D u_\lambda(\cdot)$. 
Unlike classical gradient descent, which uses only local information from $f$, this feedback law incorporates the global optimality information encoded in the HJB equation.

While optimal trajectories can be characterized analytically as in the previous theorem, their practical computation is frequently hindered by the limitations of numerical approximation schemes, which may suffer from instability, inefficiency, or sensitivity to problem data. Therefore, it becomes necessary to consider quasi-optimal trajectories, which we now proceed to define.

\begin{definition}\label{def: quasi}
Let $\varepsilon>0$ %be fixed
and $\alpha_{\varepsilon}(\cdot)$ be an admissible control. The corresponding trajectory $y_{x}^{\alpha_{\varepsilon}}(\cdot)$ is called quasi-optimal if 
\begin{equation*}
     \int_{0}^{\infty}\; \left(\frac{1}{2}|\alpha_{\varepsilon}(s)|^{2} + f(y_{x}^{\alpha_{\varepsilon}}(s))\right)\, e^{-\lambda s}\;\text{d}s = \mathscr{J}(x,\alpha_{\varepsilon}(\cdot)) \leq u_{\lambda}(x)+ \varepsilon.
\end{equation*}
When $\varepsilon$ is fixed, then we shall call it an $\varepsilon$-optimal trajectory. In particular, 
optimal trajectories are those corresponding to $\varepsilon=0$. A slightly more general definition will be later introduced in Assumption \ref{quasi opt cost}.
\end{definition}

Lastly, we define the shifted problem which corresponds to $\tilde{f}(\cdot) : = f(\cdot) - \underline{f}$. Thus the control problem is
\begin{equation}\label{OCP shift}
\begin{aligned}
    \tilde{u}_{\lambda}(x) = \; 
    \inf\limits_{\alpha(\cdot)} \; & \int_{0}^{\infty}\; \left(\frac{1}{2}|\alpha(s)|^{2} + \tilde{f}(y_{x}^{\alpha}(s))\right)\, e^{-\lambda s}\;\text{d}s,\\
    & \text{subject  to  }\; \dot{y}_{x}^{\alpha}(s) = \alpha(s),\quad y_{x}^{\alpha}(0)=x\in\mathbb{R}^n,\\
    & \text{and the controls } \alpha(\cdot):[0,\infty) \to B_{M} \text{ are measurable,} 
\end{aligned}
\end{equation}
where the value function satisfies $\tilde{u}_{\lambda}(\cdot) = u_{\lambda}(\cdot) - \underline{f}/\lambda$, and is the viscosity solution of 
\begin{equation}\label{eq: HJB lambda shift}
    \lambda\, \tilde{u}_{\lambda}(x) + \frac{1}{2}|D\tilde{u}_{\lambda}(x)|^{2} = \tilde{f}(x), \quad \quad x\in\mathbb{R}^{n}.
\end{equation}
It is clear that \eqref{OCP shift} and \eqref{OCP} share the same solutions. 
In particular, the result in Theorem \ref{thm: opt cont} holds for \eqref{OCP shift} mutatis mutandis. 

\section{Variational and objective convergence with exponential rates} \label{sec: conv rate}

We have seen from Proposition \ref{prop: min_equality} that the minimal value of the value function $u_{\lambda}(\cdot)$ is $\underline{f}/\lambda$, and it is achieved on the set $\mathfrak{M}$ of global minimizers of $f(\cdot)$. In the notation of \eqref{OCP shift}-\eqref{eq: HJB lambda shift}, this means that $\tilde{u}_{\lambda}(\cdot)$ is always non-negative, and $\tilde{u}_{\lambda}(x)=0$ if and only if $x\in \mathfrak{M}$. Our aim in this section is to provide qualitative and quantitative insights on the convergence of $s\mapsto u_{\lambda}(y_{x}(s))$ to its minimal value, along optimal or quasi-optimal trajectories $y_{x}(\cdot)$.
Our analysis relies on the following key assumption.
\begin{enumerate}[label=\textbf{(B)}]
    \item\label{main assumption} There exists $K>\lambda>0$ such that $\; K u_\lambda(x) \leq f(x)$ for all $x\in\mathbb{R}^n$.
\end{enumerate}
It implies in particular $K \,\tilde{u}_{\lambda}(x) \leq \tilde{f}(x)$  for all $x\in\mathbb{R}^n$.  This assumption is inspired by \cite[Assumption 2.3]{gaitsgory2015stabilization}, where it is used to ensure stability properties of  control problems. 
Later in the next sections, we will provide further conditions which guarantee the validity of Assumption \ref{main assumption}; see in particular the comments after Proposition \ref{prop: for main assumption}, and Proposition \ref{prop: equiv to B}, as well as Section \S\ref{sec: PL}.

\subsection{Variational convergence}

Thanks to Theorem \ref{thm: opt cont}, an optimal trajectory for the control problem \eqref{OCP} or \eqref{OCP shift} is the following 
\begin{equation}\label{eq: opt traj}
    \dot{y}^{*}_{x}(t) = -Du_{\lambda}(y^{*}_{x}(t)) \quad \text{for a.a. } t\geq 0, \quad y_{x}(0)=x\in\mathbb{R}^{n}.
\end{equation}

Let us recall some classical notions in dynamical system theory. We call $\mathcal{V}(\cdot):\mathbb{R}^{d}\to [0,\infty)$ a \textit{Lyapunov function} for a dynamical system $\dot{y}(t) = G(y(t))$ with respect to the set $\mathfrak{M}$ if 
\begin{enumerate}[label = (\roman*)]
    \item $\mathcal{V}(x) = 0$ for all $x\in \mathfrak{M}$, and $\mathcal{V}(x)>0$ otherwise,
    \item $\frac{\dd}{\dd t} \mathcal{V}(y(t)) = D \mathcal{V}(y(t))\cdot G(y(t)) \leq 0$.
\end{enumerate}
Such a function acts as an energy measuring how far the system is from the equilibrium set $\mathfrak{M}$. Its existence guarantees the stability of $\mathfrak{M}$, that is, trajectories starting sufficiently close to $\mathfrak{M}$ remain close for all time. If, in addition, $\mathcal{V}(\cdot)$ satisfies
\begin{enumerate}[label = (\roman*)]
\setcounter{enumi}{2}
    \item $\dot{\mathcal{V}}(x) < 0$ for all $x\notin \mathfrak{M}$,
\end{enumerate}
then $\mathfrak{M}$ is \textit{asymptotically stable}, meaning that trajectories starting close to $\mathfrak{M}$, not only stay close but converge to it.

\begin{proposition}
Let $x\in\mathbb{R}^{n}$ and $y^{*}_{x}(\cdot)$ be an optimal trajectory governed by \eqref{eq: opt traj}. 
Suppose Assumption \ref{f: nice}  holds. 
Then the function $\tilde{u}_{\lambda}(\cdot)$ is a Lyapunov function for this dynamics with respect to the set $\mathfrak{M}$. Moreover, $\mathfrak{M}$ is asymptotically stable. 
\end{proposition}

\begin{proof}
Recall $\tilde{u}_{\lambda}(\cdot) = u_{\lambda}(\cdot) - \underline{f}/\lambda$. 
It suffices to compute
\begin{equation}\label{Lyapunov ineq}
    \frac{\dd}{\dd t} \tilde{u}_{\lambda}(y^{*}_{x}(t)) = Du_{\lambda}(y^{*}_{x}(t))\cdot \dot{y}^{*}_{x}(t) = - |Du_{\lambda}(y^{*}_{x}(t))|^{2} \leq 0.
\end{equation}
Moreover, Proposition \ref{prop: min_equality} ensures that $\tilde{u}_{\lambda}(x) =0$ if and only if $x\in \mathfrak{M}$, which proves that $\tilde{u}_{\lambda}(\cdot)$ is indeed a Lyapunov function and that $\mathfrak{M}$ is Lyapunov stable. 

To prove that $\mathfrak{M}$ is asymptotically stable, we show in addition that \eqref{Lyapunov ineq} is strict when outside of $\mathfrak{M}$. Let $z\in\mathbb{R}^{n}$ be such that $-|Du_{\lambda}(z)|^{2} <0$. Using the HJB equation \eqref{eq: HJB lambda shift}, it holds $\; \lambda\, \tilde{u}_{\lambda}(z) - \tilde{f}(z) = - \frac{1}{2}|D\tilde{u}_{\lambda}(z)|^{2}  <0, \;$ hence $0\leq \lambda\, \tilde{u}_{\lambda}(z) < \tilde{f}(z) $ and $z\notin \mathfrak{M}$ because otherwise $0=\tilde{f}(z)$. Therefore, the inequality \eqref{Lyapunov ineq} is strict whenever outside of $\mathfrak{M}$, which ultimately guarantees its asymptotic stability.
\end{proof}

\begin{theorem}\label{thm: conv value opt}
Let $x\in \mathbb{R}^{n}$ and  $y_{x}^{*}(\cdot)$ be an optimal trajectory for \eqref{OCP} (or \eqref{OCP shift}) governed by \eqref{eq: opt traj}. 
Suppose Assumptions \ref{f: nice} and \ref{main assumption} hold. Then
\begin{equation*}
    0\leq u_{\lambda}(y_{x}^{*}(t)) - \underline{f}/\lambda \leq e^{-(K-\lambda)t} \big(u_{\lambda}(x)-\underline{f}/\lambda\big),\quad \forall\,x\in\mathbb{R}^{n}, \; \forall\,t\geq 0.
\end{equation*}
In particular, $\lambda\,u_{\lambda}(y_{x}^{*}(t))\to \underline{f}$ exponentially fast when $t\to \infty$, regardless of the initial position $x\in \mathbb{R}^{n}$.
\end{theorem}

As we show below, Theorem \ref{thm: conv value opt} is in fact a special case of Theorem  \ref{thm: conv value quasiopt} in the next subsection. However, since the proof of Theorem \ref{thm: conv value opt} is simpler and more instructive, we present it first.

\begin{proof}
Recall $\tilde{u}_{\lambda}(\cdot) = u_{\lambda}(\cdot) - \underline{f}/\lambda$ and $\tilde{f}(\cdot) = f(\cdot) - \underline{f}$. 
Using the DPP as in Proposition \ref{prop: DPP} and  Proposition \ref{Prop: h}, we know that the function
\begin{equation*}
    \tilde{h}(t) := \int_{0}^{t} \left(\frac{1}{2}|\alpha^{*}(s)|^{2} + \tilde{f}(y_{x}^{*}(s))\right)\, e^{-\lambda s}\;\dd s \, + \, \tilde{u}_{\lambda}(y_{x}^{*}(t))\,e^{-\lambda t}
\end{equation*}
is constant for all $t\geq 0$, where we recall $y_{x}^{*}(\cdot)$ is an optimal trajectory given by \eqref{eq: opt traj} and Theorem \ref{thm: opt cont}, and $\alpha^{*}(\cdot)=\dot{y}_{x}^{*}(\cdot)$. In particular, noting that $\tilde{h}(0)= \tilde{u}_{\lambda}(x)$, we have
\begin{equation*}
    \tilde{u}_{\lambda}(y_{x}^{*}(t)) = e^{\lambda\, t}\tilde{u}_{\lambda}(x) - \int_{0}^{t} \left(\frac{1}{2}|\alpha^{*}(s)|^{2} + \tilde{f}(y_{x}^{*}(s))\right)\, e^{-\lambda (s-t)}\;\dd s ,
\end{equation*}
and for almost every $t\geq 0$,
\begin{equation*}
\begin{aligned}
    \frac{\dd}{\dd t} \tilde{u}_{\lambda}(y_{x}^{*}(t))
    & = \lambda\,\tilde{u}_{\lambda}(y_{x}^{*}(t)) - \left(\frac{1}{2}|\alpha^{*}(t)|^{2} + \tilde{f}(y_{x}^{*}(t))\right)\\
    & \leq -(K-\lambda)\,\tilde{u}_{\lambda}(y_{x}^{*}(t)) \quad \text{ using } \ref{main assumption}.
\end{aligned}
\end{equation*}
Applying Gr\"onwall's inequality yields
\begin{equation*}
    0\leq \tilde{u}_{\lambda}(y_{x}^{*}(t)) \leq e^{-(K-\lambda)t} \tilde{u}_{\lambda}(x),\quad\quad \forall\,x\in\mathbb{R}^{n}, \quad \forall\,t\geq 0,
\end{equation*}
recalling the absolute continuity of $t\mapsto \tilde{u}_{\lambda}(y_{x}^{*}(t))$, and the non-negativity bound from Proposition \ref{prop: min_equality}.
\end{proof}

\subsection{Objective convergence}
As a consequence of Theorem~\ref{thm: conv value opt}, we establish  convergence of the objective values $f(y_x^*(t))$ toward the global minimum $\underline{f}$ along optimal trajectories.

Our first result is a convergence along a subsequence of time. 

\begin{corollary}\label{cor: vertical 1}
Let $x\in \mathbb{R}^{n}$ and  $y_{x}^{*}(\cdot)$ be an optimal trajectory for \eqref{OCP} governed by \eqref{eq: opt traj}. 
Suppose Assumptions \ref{f: nice} and \ref{main assumption} hold. Then there exist a constant $C>0$ and a sequence $\{s_t\}_{t\geq 0}$ such that $s_t\to \infty$ when $t\to \infty$, and along which  \begin{equation*}
    0 \; \leq \; f(y_{x}^{*}(s_t)) - \underline{f} \; \leq \; C\, e^{-(K-\lambda)\,t}, \quad \forall\,t\geq 0.  
\end{equation*}
 In particular  $\liminf\limits_{t\to \infty} f(y_{x}^{*}(t)) = \underline{f}$. 
\end{corollary}

\begin{proof}
Let $\alpha^*(\cdot)$ be an optimal control, we have 
\begin{equation*}
    u_{\lambda}(x) = \int_{0}^{\infty} \left(\frac{1}{2}|\alpha^{*}(s)|^{2} + f(y_{x}^{*}(s))\right) e^{-\lambda \, s }\dd s.
\end{equation*}
Moreover, we have shown in Theorem \ref{thm: conv value opt} that there exists a constant $C>0$, independent of time, such that 
\begin{equation*}
    0\leq u_{\lambda}(y_{x}^{*}(t)) - \underline{f}/\lambda \leq C\, e^{-(K-\lambda)t}, \quad \forall\, t\geq 0.
\end{equation*}
In fact, $C=u_{\lambda}(x)-\underline{f}/\lambda$.
Starting from $y_{x}^{*}(t)$, the shifted control $\tilde{\alpha}(s) = \alpha^{*}(t+s)$, $\forall s\geq 0$ is optimal and generates the trajectory
\begin{equation*}
    y^{*}_{y^{*}_{x}(t)}(s) = y^{*}_{x}(t+s), \quad \forall\,s\geq 0.
\end{equation*}
Therefore
\begin{equation*}
\begin{aligned}
    u_{\lambda}(y^{*}_{x}(t))
    & = \int_{0}^{\infty} \left(\frac{1}{2}|\tilde{\alpha}^{*}(s)|^{2} + f(y_{y_{x}^{*}(t)}^{*}(s))\right) e^{-\lambda \, s }\dd s\\
    & = \int_{0}^{\infty} \left(\frac{1}{2}|\alpha^{*}(t+s)|^{2} + f(y_{x}^{*}(t+s))\right) e^{-\lambda \, s }\dd s\\
    & = e^{\lambda\,t} \int_{t}^{\infty} \left(\frac{1}{2}|\alpha^{*}(r)|^{2} + f(y_{x}^{*}(r))\right) e^{-\lambda \, r }\dd r \quad \text{ setting } r=t+s.
\end{aligned}
\end{equation*}
Observe that $\frac{1}{\lambda} = e^{\lambda \, t} \int_{t}^{\infty} e^{-\lambda\, r}\dd r$, we have
\begin{equation*}
\begin{aligned}
    0 & \leq u_{\lambda}(y_{x}^{*}(t)) - \underline{f}/\lambda\\
    & = e^{\lambda\,t} \int_{t}^{\infty} \left(\frac{1}{2}|\alpha^{*}(r)|^{2} + f(y_{x}^{*}(r)) - \underline{f}\right) e^{-\lambda \, r }\dd r \; \leq C\, e^{-(K-\lambda)t},
\end{aligned}
\end{equation*}
which leads to 
\begin{equation}\label{L1}
    0\leq \int_{t}^{\infty} \left(\frac{1}{2}|\alpha^{*}(r)|^{2} + f(y_{x}^{*}(r)) - \underline{f}\right) e^{-\lambda \, r }\dd r \; \leq C\, e^{-K\,t}, \quad \forall\,t\geq 0,
\end{equation}
and finally 
\begin{equation*}
    0\leq \int_{t}^{\infty}\left( f(y_{x}^{*}(r)) - \underline{f}\right) e^{-\lambda \, r }\dd r \; \leq C\, e^{-K\,t}, \quad \forall\,t\geq 0.
\end{equation*}
Recall $\dot{y}_{x}^{*}(s) = \alpha^{*}(s)$ and $|\alpha^{*}(\cdot)|\leq M$, so, $y_{x}^{*}(\cdot)$ is Lipschitz continuous. Since $f(\cdot)$ is continuous by Assumption \ref{f: nice}, we have $s\mapsto f(y_{x}^{*}(s))$ is continuous as well. Define
\begin{equation*}
    g(\cdot):= f(y_{x}^{*}(\cdot)) - \underline{f}.
\end{equation*}
We have $g(\cdot)\geq 0$ and is continuous, satisfying 
\begin{equation*}
    0\leq \int_{t}^{\infty} g(r)\,e^{-\lambda \,r}\,\dd r \leq C\, e^{-K\,t}, \quad \forall\,t\geq 0.
\end{equation*}
In particular, for any $h>0$
\begin{equation*}
    0\leq \int_{t}^{t+h} g(r)\,e^{-\lambda \,r}\,\dd r \leq \int_{t}^{\infty} g(r)\,e^{-\lambda \,r}\,\dd r \leq C\, e^{-K\,t}, \quad \forall\,t\geq 0.
\end{equation*}
Since $r\in [t,t+h]$, we have $e^{-\lambda (t+h)} \leq e^{-\lambda\,r}$, and then
\begin{equation*}
    0\leq \int_{t}^{t+h} g(r)\,\dd r \leq \int_{t}^{\infty} g(r)\,e^{-\lambda \,r}\,\dd r \leq C\,e^{\lambda\,h}\, e^{-(K-\lambda)\,t}, \quad \forall\,t\geq 0.
\end{equation*}
Because $g(\cdot)$ is continuous, the Mean Value Theorem for integrals ensures the existence of $s_t \in [t,t+h]$ such that
\begin{equation*}
    g(s_t) = \frac{1}{h}\int_{t}^{t+h} g(r)\,\dd r
\end{equation*}
and thus
\begin{equation*}
    0\leq g(s_t)  \leq \frac{C\,e^{\lambda\,h}}{h}\, e^{-(K-\lambda)\,t}, \quad \forall\,t\geq 0,
\end{equation*}
where $s_t\to \infty$ when $t\to \infty$ since $s_t\in [t,t+h]$. 
Finally, $f(y_{x}^{*}(s_t))\to \underline{f}$ exponentially fast  as $t\to\infty$, in particular $\liminf\limits_{t\to \infty} f(y_{x}^{*}(t)) = \underline{f}$. 
\end{proof}

The next corollary is a stronger result, as it provides convergence for all times $t\geq0$, not merely along a subsequence. 
This, however, requires an additional assumption of H\"older continuity. 

\begin{enumerate}[label=\textbf{(A\arabic*')}]
\item\label{f: Hol} $f$ satisfies assumption \ref{f: nice}, moreover there exists $L>0$ and $\theta\in (0,1]$ such that
\begin{equation*}
    |f(x) - f(y)| \leq L\,|x-y|^{\theta}, \quad \forall\,x,y\in \mathbb{R}^{n}.
\end{equation*}
\end{enumerate}

Theorem \ref{thm: conv value opt} provides exponential convergence of the value function along optimal trajectories, or equivalently, exponential decay of the cost-to-go. By itself, this variational information does not exclude spikes in the pointwise quantity $f\left(y_x^*(t)\right)-\underline{f}$. The H\"older continuity of  $f$, together with the Lipschitz continuity of $y_x^*(\cdot)$, prevents such spikes from being arbitrarily narrow, and thus upgrades the variational decay into exponential pointwise convergence for all times.

\begin{corollary}
Let $x\in \mathbb{R}^{n}$ and  $y_{x}^{*}(\cdot)$ be an optimal trajectory for \eqref{OCP} governed by \eqref{eq: opt traj}. 
Suppose Assumptions \ref{f: Hol} and \ref{main assumption} hold. Then there exists a constant $C>0$ such that
\begin{equation*}
    0 \; \leq \; f(y_{x}^{*}(t)) - \underline{f} \; \leq \; C\,e^{-\frac{\theta}{1+\theta}(K-\lambda)\,t}, \quad \forall\,t\geq0.
\end{equation*}
\end{corollary}

\begin{proof}
From the proof of Corollary \ref{cor: vertical 1}, we recall 
\begin{equation*}
    g(\cdot):= f(y_{x}^{*}(\cdot)) - \underline{f}.
\end{equation*}
We have $g(\cdot) \geq 0$ and is continuous, satisfying 
\begin{equation*}
    0\leq \int_{t}^{\infty} g(r)\,e^{-\lambda \,r}\,\dd r \leq C\, e^{-K\,t}, \quad \forall\,t\geq 0.
\end{equation*}
Let us fix $\delta>0$ to be later made precise. For any $t>0$, and for all $r\in [t,t+\delta]$, thanks to \ref{f: Hol} and the Lipschitz continuity of the trajectories, we have 
\begin{equation*}
\begin{aligned}
    |g(r) - g(t)| 
    \leq L \, |y_{x}^{*}(r) - y_{x}^{*}(t)|^{\theta} \leq L \, M^{\theta}\,\delta^{\theta}.
\end{aligned}
\end{equation*}
Hence
\begin{equation*}
    g(r) \geq g(t) - L\,M^{\theta}\,\delta^{\theta}.
\end{equation*}
Integrating over $r \in [t,t+\delta]$ yields
\begin{equation}
\label{eq: vertical_g}
    \int_{t}^{t+\delta} g(r)\,e^{-\lambda\,r}\,\dd r \geq (g(t) - L\,M^{\theta}\,\delta^{\theta}) \int_{t}^{t+\delta} e^{-\lambda\,r}\,\dd r = (g(t) - L\,M^{\theta}\,\delta^{\theta}) \frac{e^{-\lambda\,t} (1 - e^{-\lambda\,\delta})}{\lambda}.
\end{equation}
Since we have
\begin{equation*}
    0\leq \int_{t}^{t+\delta} g(r)\,e^{-\lambda\,r}\,\dd r
    \leq \int_{t}^{\infty} g(r)\,e^{-\lambda\,r}\,\dd r \leq C\, e^{-K\,t},
\end{equation*}
it holds
\begin{equation*}
    (g(t) - L\,M^{\theta}\,\delta^{\theta}) \frac{e^{-\lambda\,t} (1 - e^{-\lambda\,\delta})}{\lambda} 
    \leq C\, e^{-K\,t},
\end{equation*}
that is,
\begin{equation*}
    g(t) \leq L\,M^{\theta}\,\delta^{\theta} + \frac{\lambda\,C}{1 - e^{-\lambda\,\delta}}\, e^{-(K-\lambda)t}.
\end{equation*}
Using the inequality $\frac{1}{1-e^{-z}} \leq 1 + \frac{1}{z}$ for all $z>0$, we have
\begin{equation*}
    g(t) \leq L\,M^{\theta}\,\delta^{\theta} + C\left(\lambda + \frac{1}{\delta}\right)\, e^{-(K-\lambda)t}.
\end{equation*}
Let us choose 
\begin{equation*}
    \delta = e^{-p\,t}
\end{equation*}
and $p$ will be specified below. Then we have
\begin{equation*}
    g(t)\leq L\,M^{\theta}\,e^{-p\,\theta\,t}+ C\, e^{-(K-\lambda)\,t +p \,t}+C\,\lambda\, e^{-(K-\lambda)t}.
\end{equation*}
We now choose $p$ such that first and second terms admit the same decay rate, namely $p\,\theta = (K-\lambda) - p$, which yields $p = \frac{K-\lambda}{1+\theta}$. Therefore, we obtain
\begin{equation*}
    g(t)\leq (L\,M^{\theta}+ C)\,e^{-\frac{\theta(K-\lambda)}{1+\theta}\,t}+C\,\lambda\, e^{-(K-\lambda)t}.
\end{equation*}
Since $\frac{\theta}{1+\theta}<1$ and $K>\lambda$, we also have $e^{-(K-\lambda)t} \leq e^{-\frac{\theta(K-\lambda)}{1+\theta}t}$, hence
\begin{equation*}
    g(t) \leq (L\,M{\theta}\, + C + \lambda\,C)\, e^{-\frac{\theta(K-\lambda)}{1+\theta}\,t}, \quad \forall\,t\geq0,
\end{equation*}
which concludes the proof.
\end{proof}

\subsection{Robustness to approximation errors}

In practical implementations, computing exact optimal solution to the control problem is often computationally intractable or achievable only up to some numerical error. It is therefore essential to establish convergence guarantees for quasi-optimal trajectories as defined in Definition \ref{def: quasi}.

Recalling the cost functional \eqref{cost function}, we write
\begin{equation*}
    \mathscr{J}(y_{x}^{\alpha}(t),\alpha(t+\,\cdot\,)) = \int_{t}^{\infty} \left( \frac{1}{2}|\alpha(s)|^{2} + f(y_{x}^{\alpha}(s)) \right)\,e^{-\lambda\,(s-t)}\,\dd s
\end{equation*}
and define the shifted one as
\begin{equation*}
    \widetilde{\mathscr{J}}(y_{x}^{\alpha}(t),\alpha(t+\,\cdot\,)) := \int_{t}^{\infty} \left( \frac{1}{2}|\alpha(s)|^{2} + \tilde{f}(y_{x}^{\alpha}(s)) \right)\,e^{-\lambda\,(s-t)}\,\dd s = \mathscr{J}(y_{x}^{\alpha}(t),\alpha(t+\,\cdot\,)) - \underline{f}/\lambda,
\end{equation*}
where $\tilde{f}(x) := f(x) - \underline{f}$. The corresponding value function is then $\tilde{u}_{\lambda}(x) = u_{\lambda}(x) - \underline{f}/\lambda$.

We need the following assumption on quasi-optimal trajectories. 

\begin{enumerate}[label=\textbf{(C)}]
    \item\label{quasi opt cost} There exist $\varepsilon_{\circ}\geq 0$ and $\eta(\cdot)\in L^{\infty}([0,\infty);\mathbb{R}_{+})$ 
    such that
    \begin{equation*}
       \widetilde{\mathscr{J}}(y_{x}^{\alpha_{\varepsilon}}(t),\alpha_{\varepsilon}(t+\,\cdot\,)) \leq \tilde{u}_{\lambda}(y_{x}^{\alpha_{\varepsilon}}(t)) + \varepsilon(t) \quad \quad \text{ for a.a. } t\geq 0,
    \end{equation*} 
    where $\varepsilon(t) = \eta(t)\,\tilde{u}_{\lambda}(y_{x}^{\alpha_{\varepsilon}}(t)) + \varepsilon_{\circ}$.
\end{enumerate}

\begin{remark} \label{rem: (C)}
This more general framework for quasi-optimal trajectories allows for better flexibility in a wide range of applications. The following two observations are worth noting.
\begin{enumerate}
    \item Assumption \ref{quasi opt cost} is simply
    \begin{equation*}
       \widetilde{\mathscr{J}}(y_{x}^{\alpha_{\varepsilon}}(t),\alpha_{\varepsilon}(t+\,\cdot\,)) \leq \big(1+\eta(t)\big)\tilde{u}_{\lambda}(y_{x}^{\alpha_{\varepsilon}}(t)) + \varepsilon_{\circ},
    \end{equation*}
    which coincides with Definition \ref{def: quasi} in the particular case when $t=0$, $\eta(\cdot)\equiv 0$, and then $\varepsilon(\cdot)\equiv \varepsilon_{\circ}$. If moreover $\varepsilon_{\circ}=0$, then this is an optimal trajectory. 
    \item When $t=0$, it also implies 
    \begin{equation*}
       {\mathscr{J}}(x,\alpha_{\varepsilon}(\,\cdot\,)) \leq {u}_{\lambda}(x) + \tilde{\varepsilon} ,
    \end{equation*} 
    where $\tilde{\varepsilon} := \|\eta\|_{\infty}(\overline{f}-\underline{f})/\lambda + \varepsilon_{\circ} \geq \eta(0)\,\tilde{u}_{\lambda}(x) + \varepsilon_{\circ} = \varepsilon(0)$. Hence, if Assumption \ref{quasi opt cost} is satisfied, then the quasi-optimal trajectory is also quasi-optimal in the sense of Definition \ref{def: quasi}, although with a different $\varepsilon$.
\end{enumerate}

\end{remark}

Before we state our next result, let us recall from Proposition \ref{prop: min_equality} that $0\leq \tilde{u}_{\lambda}(\cdot)$. The following is an analogue to Theorem \ref{thm: conv value opt} for quasi-optimal trajectories (instead of optimal ones).

\begin{theorem}\label{thm: conv value quasiopt}
Let $x\in \mathbb{R}^{n}$ and $y_{x}^{\alpha_{\varepsilon}}(\cdot)$ be a quasi-optimal trajectory. 
Suppose Assumptions \ref{f: nice}, \ref{main assumption}, and \ref{quasi opt cost} hold, and that $\eta(\cdot)$ in \ref{quasi opt cost} satisfies $\;\|\eta\|_{\infty} < 1- \lambda/K.$ 
Then we have
\begin{equation*}
    0\leq\;\tilde{u}_{\lambda}(y_{x}^{\alpha_{\varepsilon}}(t)) \;\leq\; C\, e^{-\delta\, t}\,\tilde{u}_{\lambda}(x) + \varepsilon_{\circ} \left(1\,+K\delta^{-1}\right),\quad\quad \forall\,t\geq0,
\end{equation*}
where  
\begin{equation*}
    \delta = (1-\|\eta\|_{\infty})K - \lambda\,>0 \quad \text{ and } \quad C=(1+\eta(0)).
\end{equation*}
In particular, we have
\begin{equation*}
    0\leq \limsup\limits_{t\to \infty} \; \tilde{u}_{\lambda}(y_{x}^{\alpha_{\varepsilon}}(t)) \leq \left(1\,+K\delta^{-1}\right)\,\varepsilon_{\circ}.
\end{equation*}
\end{theorem}

Recalling  $\tilde{u}_{\lambda}(x) = u_{\lambda}(x) - \underline{f}/\lambda$, this means that $\lambda u_{\lambda}(\cdot)$ converges exponentially fast to a small neighborhood of $\underline{f}$. 
As we have previously mentioned, this result generalizes Theorem \ref{thm: conv value opt} which is indeed obtained when $\varepsilon_{\circ}=0$, and $\eta(\cdot)\equiv 0$, then $\delta = K-\lambda$ and $C=1$.

\begin{proof}
Let us use the notation $\tilde{\ell}(\alpha,x) := \frac{1}{2}|\alpha|^{2} + \tilde{f}(x)$. By definition, we have
\begin{equation*}
\begin{aligned}
    \widetilde{\mathscr{J}}(y_{x}^{\alpha_{\varepsilon}}(t),\alpha_{\varepsilon}(t+\,\cdot\,))
    & = \int_{t}^{\infty} \tilde{\ell}(\alpha_{\varepsilon}(s),y_{x}^{\alpha_{\varepsilon}}(s))\,e^{-\lambda\,(s-t)}\,\dd s\\
    & = \int_{0}^{\infty} \tilde{\ell}(\alpha_{\varepsilon}(s),y_{x}^{\alpha_{\varepsilon}}(s))\,e^{-\lambda\,(s-t)}\,\dd s - \int_{0}^{t} \tilde{\ell}(\alpha_{\varepsilon}(s),y_{x}^{\alpha_{\varepsilon}}(s))\,e^{-\lambda\,(s-t)}\,\dd s\\
    & = e^{\lambda\,t}\widetilde{\mathscr{J}}(x,\alpha_{\varepsilon}(\,\cdot\,))- e^{\lambda\,t}\int_{0}^{t} \tilde{\ell}(\alpha_{\varepsilon}(s),y_{x}^{\alpha_{\varepsilon}}(s))\,e^{-\lambda\,s}\,\dd s.
\end{aligned}
\end{equation*}
This is absolutely continuous, hence for almost every $t\geq 0$, we have
\begin{equation}\label{eq: ineq J in proof 1}
\begin{aligned}
    \frac{\dd }{\dd t} \widetilde{\mathscr{J}}(y_{x}^{\alpha_{\varepsilon}}(t),\alpha_{\varepsilon}(t+\,\cdot\,))
    & = \lambda \,\widetilde{\mathscr{J}}(y_{x}^{\alpha_{\varepsilon}}(t),\alpha_{\varepsilon}(t+\,\cdot\,))  - \tilde{\ell}(\alpha_{\varepsilon}(t),y_{x}^{\alpha_{\varepsilon}}(t))\\
    & \leq \lambda \,\widetilde{\mathscr{J}}(y_{x}^{\alpha_{\varepsilon}}(t),\alpha_{\varepsilon}(t+\,\cdot\,)) - K\,\tilde{u}_{\lambda}(y_{x}^{\alpha_{\varepsilon}}(t)) \quad \text{ using } \ref{main assumption}\\
    & \leq -(K-\lambda) \,\widetilde{\mathscr{J}}(y_{x}^{\alpha_{\varepsilon}}(t),\alpha_{\varepsilon}(t+\,\cdot\,)) + K\,\varepsilon(t) \quad \text{ using } \ref{quasi opt cost}.
\end{aligned}
\end{equation}
Recalling the definition of $\varepsilon(t)$ in Assumption \ref{quasi opt cost}, we have
\begin{equation}\label{epsilon in proof}
\begin{aligned}
    \varepsilon(t) 
    = \eta(t)\,\tilde{u}_{\lambda}(y_{x}^{\alpha_{\varepsilon}}(t)) + \varepsilon_{\circ}
    & \leq \eta(t)\,\widetilde{\mathscr{J}}(y_{x}^{\alpha_{\varepsilon}}(t),\alpha_{\varepsilon}(t+\,\cdot\,)) + \varepsilon_{\circ}.
\end{aligned}
\end{equation}
Hence, we have
\begin{equation}\label{eq: ineq J in proof 2}
\begin{aligned}
    \frac{\dd }{\dd t} \widetilde{\mathscr{J}}(y_{x}^{\alpha_{\varepsilon}}(t),\alpha_{\varepsilon}(t+\,\cdot\,))
    & \leq -(K-\lambda - K\,\eta(t)) \,\widetilde{\mathscr{J}}(y_{x}^{\alpha_{\varepsilon}}(t),\alpha_{\varepsilon}(t+\,\cdot\,)) + K\varepsilon_{\circ}.
\end{aligned}
\end{equation}
Let us define $c(\cdot)\in L^{1}_{\text{loc}}([0,\infty);\mathbb{R})$ such that 
\begin{equation*}
    c(t):=(1-\eta(t))K-\lambda \quad \text{ and } \quad C(s,t) := \int_{s}^{t}c(r)\dd r = (K-\lambda)(t-s) - K\int_{s}^{t}\eta(r)\dd r.
\end{equation*}
Applying Gr\"onwall's inequality yields
\begin{equation}\label{ineq fron gronw in proof}
    \widetilde{\mathscr{J}}(y_{x}^{\alpha_{\varepsilon}}(t),\alpha_{\varepsilon}(t+\,\cdot\,)) \leq e^{-C(0,t)}\,\widetilde{\mathscr{J}}(x,\alpha_{\varepsilon}(\,\cdot\,)) + K\varepsilon_{\circ} \int_{0}^{t}e^{-C(s,t)}\dd s, \quad \forall\,t\geq0.
\end{equation}
We have $\;\tilde{u}_{\lambda}(y_{x}^{\alpha_{\varepsilon}}(t)) \leq \widetilde{\mathscr{J}}(y_{x}^{\alpha_{\varepsilon}}(t),\alpha_{\varepsilon}(t+\,\cdot\,))$,\; and from Assumption \ref{quasi opt cost} we also have 
\begin{equation*}
    \widetilde{\mathscr{J}}(x,\alpha_{\varepsilon}(\,\cdot\,)) \leq \tilde{u}_{\lambda}(x) + \varepsilon(0) = (1+\eta(0))\tilde{u}_{\lambda}(x) + \varepsilon_{\circ}.
\end{equation*}
The last two inequalities, when together with \eqref{ineq fron gronw in proof}, yield
\begin{equation}\label{ineq quasi opt value in proof}
    \tilde{u}_{\lambda}(y_{x}^{\alpha_{\varepsilon}}(t)) \leq e^{-C(0,t)}\,(1+\eta(0))\tilde{u}_{\lambda}(x) + \varepsilon_{\circ} \left(e^{-C(0,t)}\,+K\int_{0}^{t}e^{-C(s,t)}\dd s\right).
\end{equation}
The assumption $\;\eta(t)\leq \|\eta\|_{\infty}< 1- \lambda/K\;$ for a.a. $t\geq 0$ guarantees that 
\begin{equation*}
    0< \delta \leq c(t) \quad\quad \text{ and } \quad\quad 0< \delta(t-s) \leq C(s,t).
\end{equation*}
where $\delta := (1-\|\eta\|_{\infty})K - \lambda$. \;  
Moreover, it holds
\begin{equation*}
\begin{aligned}
    \int_{0}^{t}e^{-C(s,t)}\dd s
    & \leq \int_{0}^{t}e^{-\delta(t-s)} \dd s = \frac{1 - e^{-\delta \,t}}{\delta} \leq \frac{1}{\delta}.
\end{aligned}
\end{equation*}
The inequality \eqref{ineq quasi opt value in proof} finally becomes
\begin{equation*}
    \tilde{u}_{\lambda}(y_{x}^{\alpha_{\varepsilon}}(t)) \leq (1+\eta(0))\, e^{-\delta\, t}\,\tilde{u}_{\lambda}(x) + \varepsilon_{\circ} \left(e^{-\delta\,t}\,+K\delta^{-1}\right), \quad \forall\,t\geq0.
\end{equation*}
\end{proof}

As a consequence of Theorem \ref{thm: conv value quasiopt}, one can also derive objective convergence estimates for quasi-optimal trajectories. To this end, we consider the special case of Assumption \ref{quasi opt cost} corresponding to $\eta\equiv0$ and $\varepsilon_\circ=\varepsilon$.

\begin{corollary}
Let $x\in\mathbb R^n$, and let $y_x^{\alpha_\varepsilon}(\cdot)$ be a trajectory associated with an admissible control $\alpha_\varepsilon(\cdot)$ satisfying Assumption \ref{quasi opt cost} with $\eta(\cdot)\equiv0$ and $\varepsilon_\circ=\varepsilon\geq 0$ fixed. 

\begin{enumerate}[label=(\roman*)]
    \item\label{subseq conv eps} \textit{[Subsequence Convergence]} Suppose Assumptions \ref{f: nice} and \ref{main assumption} hold. Then there exists a constant $C > 0$ and a sequence $\{s_t\}_{t \ge 0}$ with $s_t \to \infty$ as $t \to \infty$, along which:
    \begin{equation*}
    0 \le f(y_x^{\alpha_{\varepsilon}}(s_t)) - \underline{f} \le C\, e^{-(K-\lambda)t} + \mathcal{O}(\varepsilon), 
    \quad \forall t \ge 0.
    \end{equation*}

    \item\label{pt conv eps} \textit{[Pointwise Convergence]} Suppose Assumptions \ref{f: Hol} and \ref{main assumption} hold. Then there exists a constant $C > 0$ such that:
    \begin{equation*}
    0 \le f(y_x^{\alpha_{\varepsilon}}(t)) - \underline{f} \le C\, e^{-\frac{\theta}{1+\theta}(K-\lambda)t} + \mathcal{O}(\varepsilon),  
    \quad \forall t \ge 0.
    \end{equation*}
\end{enumerate}
\end{corollary}

\begin{proof}
Define $g(r) := f(y_x^{\alpha_{\varepsilon}}(r)) - \underline{f}  \ge 0$. From the  analysis in  the Theorem  \ref{thm: conv value quasiopt} (setting $\eta \equiv 0$ and $\varepsilon_0 = \varepsilon$), the shifted cost functional along the trajectory satisfies
\begin{equation*}
\widetilde{\mathscr{J}}(y_x^{\alpha_{\varepsilon}}(t), \alpha_{\varepsilon}(t+\cdot)) \le e^{-(K-\lambda)t} (\tilde{u}_\lambda(x) + \varepsilon) + \frac{\varepsilon}{K-\lambda}(1 - e^{-(K-\lambda)t}).
\end{equation*}
Then, we have
\begin{equation}\label{eq: quasi_integral}
\begin{aligned}
    \int_t^\infty g(r)e^{-\lambda r} \,\dd r 
    & \le e^{-\lambda t} \int_t^\infty (\frac{1}{2}|\alpha_\varepsilon(r)|^2 + g(r))e^{-\lambda(r-t)} \,\dd r \\
    &\le C_0 e^{-Kt} + \varepsilon e^{-Kt}+ \frac{\varepsilon}{K-\lambda}e^{-\lambda t}\\
    & \le C_0 e^{-Kt} +(1+ \frac{1}{K-\lambda}) \varepsilon\, e^{-\lambda t},
\end{aligned}
\end{equation}
where $C_0 = \tilde{u}_\lambda(x)$.

For $h > 0$ and  any $t \ge 0$, we have $\int_t^{t+h} g(r) e^{-\lambda r} \,\dd r \le \int_t^\infty g(r) e^{-\lambda r} \,\dd r$. Since $e^{-\lambda(t+h)} \le e^{-\lambda r}$ for $r \in [t, t+h]$, it follows:
\begin{equation*}
\int_t^{t+h} g(r) \,\dd r \le e^{\lambda(t+h)} \left[ C_0 e^{-Kt} + (1+ \frac{1}{K-\lambda}) \varepsilon\ e^{-\lambda t} \right] = C_0 e^{\lambda h} e^{-(K-\lambda)t} + (1+ \frac{1}{K-\lambda}) \varepsilon\ e^{\lambda h}.
\end{equation*}
By the Mean Value Theorem for integrals, there exists $s_t \in [t, t+h]$ such that
\begin{equation*}
    g(s_t) = \frac{1}{h} \int_t^{t+h} g(r) \,\dd r, \quad \text{ and thus } \quad
    0 \le g(s_t) \le \frac{e^{\lambda h}}{h} \left[ C_0 e^{-(K-\lambda)t} + (1+ \frac{1}{K-\lambda}) \varepsilon\ \right],
\end{equation*}
which is the result in \ref{subseq conv eps}. 

To prove the second statement \ref{pt conv eps}, fix $\delta > 0$ to be specified. As in \eqref{eq: vertical_g}, for all $r \in [t, t+\delta]$, Assumption \ref{f: Hol} and the Lipschitz continuity of admissible  trajectories give
\begin{equation*}
\int_t^{t+\delta} g(r) e^{-\lambda r} \,\dd r \ge (g(t) - L M^\theta \delta^\theta) \frac{e^{-\lambda t}(1 - e^{-\lambda \delta})}{\lambda}.
\end{equation*}
Using the upper bound on the integral from \eqref{eq: quasi_integral} and the inequality $\frac{1}{1-e^{-z}} \le 1 + \frac{1}{z}$:
\begin{equation*}
g(t) \le L M^\theta \delta^\theta + \left( \lambda + \frac{1}{\delta} \right) \left[ \lambda C_0 e^{-(K-\lambda)t} + (1+ \frac{1}{K-\lambda})\lambda \varepsilon \right].
\end{equation*}
Choose
\begin{equation*}
\delta:=\max\Bigl(e^{-pt},\,\varepsilon^{\frac{1}{1+\theta}}\Bigr),
\qquad
p:=\frac{K-\lambda}{1+\theta}.
\end{equation*}
Then
\begin{equation*}
\delta^\theta \le e^{-p\theta t}+\varepsilon^{\frac{\theta}{1+\theta}},
\qquad
\frac{e^{-(K-\lambda)t}}{\delta}\le e^{-((K-\lambda)-p)t},
\qquad
\frac{\varepsilon}{\delta}\le \varepsilon^{\frac{\theta}{1+\theta}}.
\end{equation*}
Since $p\theta=(K-\lambda)-p=\frac{\theta}{1+\theta}(K-\lambda)$, it follows that
\begin{equation*}
g(t)\le C_3 e^{-\frac{\theta}{1+\theta}(K-\lambda)t}
   + C_4\max\{\varepsilon,\varepsilon^{\frac{\theta}{1+\theta}}\}.
\end{equation*}
with $C_3=L M^\theta+\lambda(1+\lambda) C_0\;$ and $\;C_4=\left(L M^\theta+\lambda(1+\lambda)\right)(1+ \frac{1}{K-\lambda}) .$
\end{proof}

When dealing with numerical approximations, one could obtain a refinement of Assumption \ref{quasi opt cost} for quasi-optimal trajectories as in \cite[Assumption 3.4]{gaitsgory2015stabilization}, mainly:

\begin{enumerate}[label=\textbf{($\text{C}_{\Delta}$)}]
    \item\label{quasi opt cost delta} 
    There exists a sequence of updating times $0=\tau_{0}<\tau_{1}<\tau_{2}<\dots\;$ with $0< \Delta_{\min}\leq \tau_{i+1} - \tau_{i}\leq \Delta_{\max}<\infty$
    such that
    \begin{equation*}
       \widetilde{\mathscr{J}}(y_{x}^{\alpha_{\varepsilon}}(\tau_{i}),\alpha_{\varepsilon}(\tau_{i}+\,\cdot\,)) \leq \tilde{u}_{\lambda}(y_{x}^{\alpha_{\varepsilon}}(\tau_{i})) + \varepsilon(\tau_{i}), \quad \quad \forall\,i,
    \end{equation*} 
    where for some $\sigma>0$ we have
    \begin{equation*}
        \varepsilon(\cdot) := \frac{\sigma}{2}e^{-K\,\Delta_{\max}}\,\tilde{u}_{\lambda}(y_{x}^{\alpha_{\varepsilon}}(\cdot)).
    \end{equation*}
\end{enumerate}

The difference with Assumption \ref{quasi opt cost} is that now quasi-optimality is assumed to hold only along a discrete sequence of sampling times $\{\tau_{i}\}_{i\in \mathbb{N}}$, and not for a.a. $t\geq 0$. But it also requires a special form for $\varepsilon(\cdot)$, which is not the case in Assumption \ref{quasi opt cost}. 

\begin{theorem}\label{thm: conv value quasiopt delta}
Let $x\in \mathbb{R}^{n}$ and $y_{x}^{\alpha_{\varepsilon}}(\cdot)$ be a quasi-optimal trajectory. 
Suppose Assumptions \ref{f: nice}, \ref{main assumption}, and \ref{quasi opt cost delta} hold. Then we have
\begin{equation}\label{conv theta}
    0\leq\;\tilde{u}_{\lambda}(y_{x}^{\alpha_{\varepsilon}}(t)) \;\leq\; (1+c)\, e^{-\theta\, t}\,\tilde{u}_{\lambda}(x), \quad\quad \forall\, t\geq0
\end{equation}
where $\;\theta = K-\lambda - Kc-\frac{\ln(1+c)}{\Delta_{\min}}$, \; and \; $c:=\tfrac{\sigma}{2}e^{-K\Delta_{\max}}.$ 
\end{theorem}

Recalling  $\tilde{u}_{\lambda}(x) = u_{\lambda}(x) - \underline{f}/\lambda$, Theorem \ref{thm: conv value quasiopt delta} means that 
$\lambda\,u_{\lambda}(y_{x}^{\alpha_{\varepsilon}}(t))\to \underline{f}$ exponentially fast when $t\to \infty$, regardless of the initial position $x\in \mathbb{R}^{n}$, even along a quasi-optimal trajectory. This underscores the fact that even a quasi-optimal trajectory, arising from an approximate solution of the control problem, can effectively lead to the optimal value $\underline{f}$ of the optimization problem.

\begin{proof}
We start as in the proof of Theorem \ref{thm: conv value quasiopt}. We denote $\widetilde{\mathscr{J}}(t):=\widetilde{\mathscr{J}}(y_{x}^{\alpha_{\varepsilon}}(t),\alpha_{\varepsilon}(t+\,\cdot\,))$ for simplicity. By modifying \eqref{epsilon in proof} according to \ref{quasi opt cost delta}, one has for $t \in[\tau_i, \tau_{i+1})$,
\begin{equation}\label{ineq diff in proof}
\begin{aligned}
    \frac{\dd }{\dd t} \widetilde{\mathscr{J}}(t)
    & \leq -(K-\lambda) \,\widetilde{\mathscr{J}}(t) + K\,\varepsilon(t)\\
    & \leq -\left[K-\lambda - K \, \frac{\sigma}{2}e^{-K\,\Delta_{\max}}\right] \,\widetilde{\mathscr{J}}(t).
\end{aligned}
\end{equation}
Define $\Gamma:= K-\lambda - K \, \frac{\sigma}{2}e^{-K\,\Delta_{\max}}.$ 
Integrating over $[\tau_i,\tau_{i+1})$ yields
\begin{equation}\label{eq:J-step}
\widetilde{\mathscr{J}}(\tau_{i+1})
   \le e^{-\Gamma(\tau_{i+1}-\tau_i)}\,\widetilde{\mathscr{J}}(\tau_{i}).
\end{equation}
At each sampling time, by Assumption~\ref{quasi opt cost delta}, one has
\begin{equation*}
\widetilde{\mathscr{J}}(\tau_{i})
      \le \tilde u_\lambda(y_{x}^{\alpha_{\varepsilon}}(\tau_i))
         + \frac{\sigma}{2}e^{-K\Delta_{\max}}\tilde u_\lambda(y_{x}^{\alpha_{\varepsilon}}(\tau_i))
      = (1+c)\,\tilde u_\lambda(y_{x}^{\alpha_{\varepsilon}}(\tau_i)),
      \text{ where } c:=\tfrac{\sigma}{2}e^{-K\Delta_{\max}}.
\end{equation*}
Since $\tilde u_\lambda(y_{x}^{\alpha_{\varepsilon}}(\tau_{i+1}))\le\widetilde{\mathscr{J}}(\tau_{i+1})$,
combining this with \eqref{eq:J-step} gives, for all $i\ge0$,
\begin{equation*}
\tilde u_\lambda(y_{x}^{\alpha_{\varepsilon}}(\tau_{i+1}))
   \le (1+c)e^{-\Gamma(\tau_{i+1}-\tau_i)}\tilde u_\lambda(y_{x}^{\alpha_{\varepsilon}}(\tau_i)).
\end{equation*}
By induction starting from $\tau_0=0$, we obtain
\begin{equation}\label{eq:disc-bound}
\tilde u_\lambda(y_{x}^{\alpha_{\varepsilon}}(\tau_i))
   \le (1+c)^i\,e^{-\Gamma\tau_i}\,\tilde u_\lambda(x).
\end{equation}

Now fix any $t\ge0$ and let $i(t):=\max\{i:\tau_i\le t\}$. 
For $t\in[\tau_{i(t)},\tau_{i(t)+1})$, 
we integrate \eqref{ineq diff in proof} between $[\tau_{i(t)},\,t]$ and obtain
\begin{equation*}
\widetilde{\mathscr{J}}(t)
   \le e^{-\Gamma(t-\tau_{i(t)})}\,\widetilde{\mathscr{J}}(\tau_{i(t)})
   \le e^{-\Gamma(t-\tau_{i(t)})}(1+c)\,\tilde u_\lambda(y_{x}^{\alpha_{\varepsilon}}(\tau_{i(t)})).
\end{equation*}
Noting that $\tilde u_\lambda(y_{x}^{\alpha_{\varepsilon}}(t))\le\widetilde{\mathscr{J}}(t)$, when together with
\eqref{eq:disc-bound}, yields
\begin{equation*}
\tilde u_\lambda(y_{x}^{\alpha_{\varepsilon}}(t))
   \le (1+c)^{i(t)+1} e^{-\Gamma t}\,\tilde u_\lambda(x).
\end{equation*}
Since $\tau_{i+1}-\tau_i \geq \Delta_{\min }$, we have  $
i(t)\leq \frac{t}{\Delta_{\min }} $
and
\begin{equation*}
(1+c)^{i(t)+1} \leq(1+c)^{\frac{t}{\Delta_{\min }}+1}=(1+c) e^{\frac{\log (1+c)}{\Delta_{\min }} t} .
\end{equation*}
Consequently,
\begin{equation*}
\tilde u_\lambda(y_{x}^{\alpha_{\varepsilon}}(t))
   \le (1+c)\,e^{-\bigl(\Gamma-\frac{\ln(1+c)}{\Delta_{\min}}\bigr)t}
         \tilde u_\lambda(x).
\end{equation*}
Finally, choosing
\begin{equation*}
\theta :=  \Gamma-\frac{\ln(1+c)}{\Delta_{\min}}=K-\lambda-Kc-\frac{\ln(1+c)}{\Delta_{\min}},
\end{equation*}
one gets the desired result
\begin{equation*}
0\le \tilde u_\lambda(y_{x}^{\alpha_{\varepsilon}}(t))
   \le (1+c)e^{-\theta t}\tilde u_\lambda(x),
   \qquad \forall t\ge0.
\end{equation*}
\end{proof}
%\end{thisnote}

\section{Global convergence and stabilization}\label{sec: app opt}

This section is devoted to formalizing our findings in the context of optimization. 
Our analysis relies on the following additional assumptions\footnote{
The assumptions \ref{f: has min} and \ref{f: stab}  were introduced in our companion work  \cite{huang2025control}. To maintain consistency, we shall refer to them by the same names.} regarding the objective function $f(\cdot)$.

\begin{enumerate}[label=\textbf{(A\arabic*)}]
\setcounter{enumi}{1}
\item\label{f: has min} $f$ attains the minimum, i.e.,  
  \begin{equation*}
    \mathfrak{M}:=\{x\in\mathbb{R}^{n}\,:\, f(x) = \underline{f}:= \min\limits_{z\in\mathbb{R}^{n}}f(z)\} \ne \emptyset .
\end{equation*}
\item\label{f: stab} For all $\delta>0$, there exists  $\gamma=\gamma(\delta)>0$ such that
    \begin{equation*}
         \inf\{f(x) - \underline{f}\,:\,\text{dist}(x,\mathfrak{M})\,> \delta\} \,>\, \gamma(\delta).
    \end{equation*}
\end{enumerate}

\begin{enumerate}[label=\textbf{(D)}]
    \item\label{QG} There exist $\mathbf{r}>0$  and $0<\mathfrak{c}_{1}\leq \mathfrak{c}_{2}<\infty$  such that
    \begin{equation*}
    \frac{\mathfrak{c}_{1}}{2} \, \operatorname{dist}(x,\mathfrak M)^2 \leq \; f(x) - \underline{f} \; \leq  \frac{\mathfrak{c}_{2}}{2} \, \operatorname{dist}(x,\mathfrak M)^2, \quad\forall\,x \in \left\{x\in\mathbb{R}^{n}: \operatorname{dist}(x,\mathfrak{M})\leq \mathbf{r}\right\}.
    \end{equation*}
\end{enumerate}

Later, we shall assume $\mathfrak{M}$ to be compact. In that case, Assumption \ref{f: stab} is equivalent to 
\begin{equation}\label{equiv to stab}
    \liminf\limits_{|x|\to \infty} f(x)> \underline{f}.
\end{equation}
In general, when $\mathfrak{M}$ is unbounded, \eqref{equiv to stab} is not true as it has been noted in \cite{bardi2023eikonal}. From an optimization viewpoint, Assumption \ref{f: stab} imposes a uniform energy gap between the global minimizers and all other points, in particular non-global minimizers, thereby quantifying how much higher the energy must be outside the optimal set. Such a separation guarantees the robustness of minimizers, supports asymptotic principles such as large deviations and Laplace approximations, and, in the context of Gibbs measures, leads to exponential concentration of probability mass near the set of global minimizers.

The left-hand inequality in Assumption \ref{QG} is a \emph{local quadratic growth} condition relative to the global minimizer set $\mathfrak{M}$. This is a purely metric (geometric) requirement that imposes neither convexity nor differentiability. In particular, it does not exclude the presence of other local minima away from $\mathfrak{M}$, rather, it postulates that the objective becomes well-conditioned once we enter a neighborhood of $\mathfrak{M}$. For context, in smooth convex problems, quadratic growth is closely related to \L{}ojasiewicz-type conditions, which are standard tools for establishing linear convergence of first-order and proximal methods; see, e.g., \cite{aujol2023convergence,drusvyatskiy2018error,polyak1963gradient}. Our analysis, however, relies only on the geometric condition \ref{QG} and therefore applies beyond the smooth convex setting. We explore these connections to classical optimization conditions in greater detail in Section \ref{sec: comparison}.

\subsection{Stability of the set of global minimizers}\label{sec: stab}

The following theorem is borrowed from \cite[Theorem 5.1]{huang2025control}. 

\begin{theorem}\label{thm: practical}
Let Assumptions \ref{f: nice}, \ref{f: has min}, and \ref{f: stab} be satisfied. Let $x\in\mathbb{R}^{n}$ be given, and $y_{x}^{\alpha_{\varepsilon}}(\cdot)$ be an $\varepsilon$-optimal trajectory for problem \eqref{OCP} with $\varepsilon\geq 0$ small. 
Then for all $\mathbf{r}>0$, there exist $\lambda>0$ small and $\tau = \tau(\mathbf{r},\varepsilon, \lambda)>0$ such that 
\begin{equation*}
    \operatorname{dist}\left(y_{x}^{\alpha_{\varepsilon}}(s), \mathfrak{M}\right) \leq \mathbf{r}, \quad \quad \forall\, s\in [\tau, \infty).
\end{equation*}
In other words, $y_{x}^{\alpha_{\varepsilon}}(s) \in \left\{z\in \mathbb{R}^{n}: \operatorname{dist}(z,\mathfrak{M})\leq \mathbf{r}\right\}$ for all $s\geq \tau$. 
\end{theorem}

Let us comment on the time interval $[0,\tau]$. By definition of $\tau$ in Theorem \ref{thm: practical}, this interval corresponds to the initial phase during which a (quasi-)optimal trajectory remains outside a neighborhood of the set of global minimizers $\mathfrak{M}$. As shown in \cite[Theorem 3.4]{huang2025control}, the duration of this phase decreases as $\lambda \to 0$; more precisely, it is of order $\mathcal{O}(\lambda)$. Hence, by tuning the parameter $\lambda$, the period $[0,\tau]$ can be made arbitrarily short. Once the trajectory enters a neighborhood of the global minimizers, the results that we are about to establish below will ensure exponential convergence to the set $\mathfrak{M}$.

\subsection{A Riccati-type structure in the value function} 

Let us first consider a closed and non-empty set $\mathscr{O}\subset \mathbb{R}^{n}$, and show that in the situation where we choose $x\mapsto \operatorname{dist}(x,\mathscr{O})$ instead of $f(\cdot)$ in the control problem \eqref{OCP}, one obtains a closed-form expression for the corresponding value function. This is comparable with Riccati solution in the linear-quadratic case. To this aim, we define for some $\mathfrak{c}>0$ constant 
\begin{equation}\label{OCP riccati}
\begin{aligned}
    R_{\mathfrak{c}}(x,\mathscr{O}) := \; 
    \inf\limits_{\alpha(\cdot)} \; & \int_{0}^{\infty}\; \left(\frac{1}{2}|\alpha(s)|^{2} + \frac{\mathfrak{c}}{2}\,\operatorname{dist}(y_{x}^{\alpha}(s),\mathscr{O})^{2}\right)\, e^{-\lambda s}\;\dd s,\\
    & \text{subject  to  }\; \dot{y}_{x}^{\alpha}(s) = \alpha(s),\quad y_{x}^{\alpha}(0)=x\in\mathbb{R}^n,\\
    & \text{and the controls } \alpha(\cdot):[0,\infty) \to B_{M} \text{ are measurable}. 
\end{aligned}
\end{equation}

The following result provides an explicit definition of the value function $R_{\mathfrak{c}}(x,\mathscr{O})$. Its proof is postponed to Appendix \ref{app: ricc}.

\begin{lemma}\label{lem: ricc}
Let $\mathscr{O}\subset\mathbb{R}^{n}$ be a non-empty and compact set, and let $\mathfrak{c}>0$. Then the value function corresponding to \eqref{OCP riccati} admits the closed form expression
\begin{equation*}
    R_{\mathfrak{c}}(x,\mathscr{O}) = C\,\operatorname{dist}(x,\mathscr{O})^{2} ,\quad \forall\,x\in\mathbb{R}^{n},
\end{equation*}
where $C = (-\lambda + \sqrt{\lambda^{2} + 4\mathfrak{c}}\,)/4 \,>\,0$.
\end{lemma}

\begin{remark}
To compare Lemma \ref{lem: ricc} with Riccati theory, let us consider the case $\mathfrak{M}=\{0\}$, that is
\begin{equation*}
\begin{aligned}
    \inf\limits_{\alpha(\cdot)} \; & \int_{0}^{\infty}\; \left(\frac{1}{2}|\alpha(s)|^{2} + \frac{\mathfrak{c}}{2}\,|y_{x}^{\alpha}(s)|^{2}\right)\, e^{-\lambda s}\;\dd s,\\
    & \text{subject  to  }\; \dot{y}_{x}^{\alpha}(s) = \alpha(s),\quad y_{x}^{\alpha}(0)=x\in\mathbb{R}^n.
\end{aligned}
\end{equation*}
This is a simple example of a linear–quadratic regulator (LQR). Both the optimal control and the value function are given by the discounted algebraic Riccati equation
\begin{equation*}
    \lambda E = \mathfrak{c}\mathbb{I}_{n} - E^{2},
\end{equation*}
where the unknown is $E$, a symmetric matrix in $\mathbb{R}^{n}\times\mathbb{R}^{n}$, and $\mathbb{I}_{n}$ is the identity matrix. To solve this matrix equation, one observes that: if $E$ is a solution, then $U^{\top}EU$ is also a solution for all $U$ that is an orthogonal matrix. But the only symmetric matrices that are invariant under orthogonal transformations are those of the form $\varrho I_{n}$ where $\varrho\in \mathbb{R}$ is a scalar. Thus, we search for a solution of the form $E=\varrho\mathbb{I}_{n}$. This leads to a second-order equation $\varrho^{2} + \lambda \varrho - \mathfrak{c} = 0$ with two algebraic roots, only the positive one ($\varrho = (-\lambda + \sqrt{\lambda^{2}+4\mathfrak{c}})/2$) yields a stabilizing feedback controller. Therefore $E = \varrho \mathbb{I}_{n}$, the optimal control is $\alpha^{*}(s) = -E y(s) = -\varrho \,y(s)$, and the optimal trajectory is $y(s) = x\,e^{-\varrho s}$. Inserting this control and this trajectory back in the cost functional, we recover the value function $R_{\mathfrak{c}}(x,\{0\}) = \frac{\varrho}{2} |x|^{2}$. This is the same solution as in Lemma \ref{lem: ricc}, where $C$ therein is exactly $\varrho/2$ here, and $\operatorname{dist}(x,\{0\}) = |x|$. The proof of this lemma when $\mathfrak{M}$ is not a trivial set requires a different strategy; see Appendix \ref{app: ricc}.
\end{remark}

Let us recall  $\tilde{u}_{\lambda}(\cdot) = u_{\lambda}(\cdot) - \underline{f}/\lambda$ where $u_{\lambda}(\cdot)$ is the value function in \eqref{OCP}. The next result shows that, along a (quasi-)optimal trajectory, the value function is proportional to the distance between the trajectory and the set of global minimizers $\mathfrak{M}$. In particular, small values of the value function correspond to small distances to $\mathfrak{M}$, up to an error term reflecting the possible lack of optimality of trajectories.

\begin{proposition}\label{prop: sandwich u}
For a given $x\in\mathbb{R}^{n}$, let $y_{x}^{\alpha_{\varepsilon}}(\cdot)$ be a quasi-optimal trajectory for problem \eqref{OCP}. % with $\varepsilon\geq 0$ small. 
Suppose the following: 
\begin{enumerate}[label = -]
    \item Assumptions \ref{f: nice}, \ref{f: has min}, and \ref{f: stab} are satisfied, and $\;\mathfrak{M}$ is compact.
    \item Assumption \ref{quasi opt cost} holds for some $\eta(\cdot)\in L^{\infty}([0,\infty);\mathbb{R}_{+})$ and $\varepsilon_{\circ}\geq 0$.
    \item Assumption \ref{QG} holds for some $\mathbf{r}>0$ and $0<\mathfrak{c}_{1}\leq \mathfrak{c}_{2}<\infty$.
\end{enumerate}
Then there exist $\lambda>0$ small and $\tau = \tau(\mathbf{r},\varepsilon_{\circ}, \lambda, \|\eta\|_{\infty})>0$ such that $\;\forall\,t\geq\tau$
\begin{equation*}
    C_{1}\,\operatorname{dist}(y_{x}^{\alpha_{\varepsilon}}(t),\mathfrak{M})^{2} \; 
    \leq \; \tilde{u}_{\lambda}(y_{x}^{\alpha_{\varepsilon}}(t)) + \varepsilon(t) \; 
    \leq \; C_{2}\,\operatorname{dist}(y_{x}^{\alpha_{\varepsilon}}(t),\mathfrak{M})^{2} + \varepsilon(t)+\varepsilon_{\circ}
\end{equation*}
or equivalently, recalling the definition of $\varepsilon(\cdot)$ from Assumption \ref{quasi opt cost},
\begin{equation*}
    C_{1}\,\operatorname{dist}(y_{x}^{\alpha_{\varepsilon}}(t),\mathfrak{M})^{2} \;
    \leq \; (1+\eta(t))\tilde{u}_{\lambda}(y_{x}^{\alpha_{\varepsilon}}(t)) + \varepsilon_{\circ} \; 
    \leq \; (1+\eta(t))\,C_{2}\,\operatorname{dist}(y_{x}^{\alpha_{\varepsilon}}(t),\mathfrak{M})^{2} + (2+ \eta(t))\varepsilon_{\circ},
\end{equation*}
where 
\begin{equation*}
    C_{i} = (-\lambda + \sqrt{\lambda^{2} + 4\mathfrak{c}_{i}}\,)/4 \quad \quad \text{ for } i=1,2.
\end{equation*}
\end{proposition}

A special case  is when we choose in Assumption \ref{quasi opt cost} the function $\eta(\cdot)\equiv 0$ and $\varepsilon_{\circ}=\varepsilon$, hence recovering the case of quasi-optimal trajectory as in Definition \ref{def: quasi}. 
\begin{proposition}
For a given $x\in\mathbb{R}^{n}$, let $y_{x}^{\alpha_{\varepsilon}}(\cdot)$ be an $\varepsilon$-optimal trajectory for problem \eqref{OCP} with $\varepsilon\geq 0$ small. 
Suppose the following: 
\begin{enumerate}[label = -]
    \item Assumptions \ref{f: nice}, \ref{f: has min}, and \ref{f: stab} are satisfied, and $\;\mathfrak{M}$ is compact. 
    \item Assumption \ref{QG} holds for some $\mathbf{r}>0$ and $0<\mathfrak{c}_{1}\leq \mathfrak{c}_{2}<\infty$.
\end{enumerate}
Then there exist $\lambda>0$ small and $\tau = \tau(\mathbf{r},\varepsilon, \lambda)>0$ such that %$\;\forall\,t\geq\tau$
\begin{equation*}
    C_{1}\,\operatorname{dist}(y_{x}^{\alpha_{\varepsilon}}(t),\mathfrak{M})^{2} \; -\varepsilon\; 
    \leq \; \tilde{u}_{\lambda}(y_{x}^{\alpha_{\varepsilon}}(t))  \; 
    \leq \; C_{2}\,\operatorname{dist}(y_{x}^{\alpha_{\varepsilon}}(t),\mathfrak{M})^{2} + \varepsilon ,\quad\quad \forall\,t\geq\tau.
\end{equation*}
The same result holds for an optimal trajectory driven by $\alpha^{*}(\cdot)$, after setting $\varepsilon=0$ , that is
\begin{equation*}
    C_{1}\,\operatorname{dist}(y_{x}^{*}(t),\mathfrak{M})^{2} \; \leq \; \tilde{u}_{\lambda}(y_{x}^{*}(t))\; \leq \; C_{2}\,\operatorname{dist}(y_{x}^{*}(t),\mathfrak{M})^{2}, \quad\quad \forall\,t\geq\tau.
\end{equation*}
\end{proposition}

\begin{proof}[Proof of Proposition \ref{prop: sandwich u}]
Let $\mathbf{r}>0$ be fixed, and $0<\mathfrak{c}_{1}\leq \mathfrak{c}_{2}<\infty$ be given as in Assumption \ref{QG}. 
The existence of $\lambda>0$ small and $\tau = \tau(\mathbf{r},\varepsilon_{\circ},\lambda, \|\eta\|_{\infty})>0$  such that the quasi-optimal trajectory satisfies for all $s\geq \tau$,  $\;y_{x}^{\alpha_{\varepsilon}}(s) \in \left\{z\in \mathbb{R}^{n}: \operatorname{dist}(z,\mathfrak{M})\leq \mathbf{r}\right\}$, is guaranteed by Theorem \ref{thm: practical} together with Remark \ref{rem: (C)}. Therefore, recalling $\tilde{f}(\cdot)=f-\underline{f}$ and using Assumption \ref{QG}, one obtains
\begin{equation}\label{sandwich f in proof}
    \frac{\mathfrak{c}_{1}}{2} \, \operatorname{dist}(y_{x}^{\alpha_{\varepsilon}}(s),\mathfrak M)^2 \le \; \tilde{f}(y_{x}^{\alpha_{\varepsilon}}(s))  \; \le  \frac{\mathfrak{c}_{2}}{2} \, \operatorname{dist}(y_{x}^{\alpha_{\varepsilon}}(s),\mathfrak M)^2, \quad\forall\,s\geq \tau.
\end{equation}

Let $t\geq \tau$ be arbitrarily fixed, and define $\xi:= y_{x}^{\alpha_{\varepsilon}}(t)$.  The  first inequality in \eqref{sandwich f in proof} yields 
\begin{equation*}
    \int_{t}^{\infty}\left(\frac{1}{2}|\alpha_{\varepsilon}(s)|^{2} + \frac{\mathfrak{c}_{1}}{2} \, \operatorname{dist}(y_{\xi}^{\alpha_{\varepsilon}}(s),\mathfrak M)^2 \right)e^{-\lambda\,(s-t)}\,\dd s
    \leq \; \int_{t}^{\infty}\left(\frac{1}{2}|\alpha_{\varepsilon}(s)|^{2} + \tilde{f}(y_{\xi}^{\alpha_{\varepsilon}}(s))\right)e^{-\lambda\,(s-t)}\,\dd s.
\end{equation*}

The right-hand side is exactly $\widetilde{\mathscr{J}}(\xi,\alpha_{\varepsilon}(t+\,\cdot\,))$. Using the definition of quasi-optimality as in Assumption \ref{quasi opt cost}, the right-hand side term in the above inequality is upper-bounded by $\tilde{u}_{\lambda}(\xi) + \varepsilon(t)$. 

For the left-hand side, it suffices to note that $\alpha_{\varepsilon}(\cdot)$ is admissible (without necessarily being optimal) in the problem \eqref{OCP riccati} where we chose $\mathfrak{c}=\mathfrak{c}_{1}$ and $\mathscr{O}=\mathfrak{M}$. Therefore it is lower bounded by $R_{\mathfrak{c}_{1}}(\xi, \mathfrak{M})$. Ultimately, one gets
\begin{equation}\label{eq: low R}
    R_{\mathfrak{c}_{1}}(\xi, \mathfrak{M}) \leq \tilde{u}_{\lambda}(\xi) + \varepsilon(t).
\end{equation}

On the other hand, let us choose $\beta_{\varepsilon}(\cdot)$ an  $\varepsilon_{\circ}$-optimal control (as in Definition \ref{def: quasi}) in the problem \eqref{OCP riccati} where we chose $\mathfrak{c}=\mathfrak{c}_{2}$ and $\mathscr{O}=\mathfrak{M}$, that is
\begin{equation*}
    \int_{t}^{\infty}\; \left(\frac{1}{2}|\beta_{\varepsilon}(s)|^{2} + \frac{\mathfrak{c}_{2}}{2}\,\operatorname{dist}(y_{\xi}^{\beta_{\varepsilon}}(s),\mathfrak{M})\right)\, e^{-\lambda (s-t)}\;\dd s \leq R_{\mathfrak{c}_{2}}(\xi, \mathfrak{M}) + \varepsilon_{\circ}.
\end{equation*}
Using now the second inequality in \eqref{sandwich f in proof}, one has
\begin{equation*}
    \int_{t}^{\infty}\; \left(\frac{1}{2}|\beta_{\varepsilon}(s)|^{2} + \tilde{f}(y_{\xi}^{\beta_{\varepsilon}}(s))\right)\, e^{-\lambda (s-t)}\;\dd s
    \leq 
    \int_{t}^{\infty}\; \left(\frac{1}{2}|\beta_{\varepsilon}(s)|^{2} + \frac{\mathfrak{c}_{2}}{2}\,\operatorname{dist}(y_{\xi}^{\beta_{\varepsilon}}(s),\mathfrak{M})\right)\, e^{-\lambda (s-t)}\;\dd s.
\end{equation*}
The control $\beta_{\varepsilon}(\cdot)$ being admissible (without necessarily begin optimal) for the problem \eqref{OCP}, one gets
\begin{equation*}
    \tilde{u}_{\lambda}(\xi)
    \leq 
    \int_{t}^{\infty}\; \left(\frac{1}{2}|\beta_{\varepsilon}(s)|^{2} + \tilde{f}(y_{\xi}^{\beta_{\varepsilon}}(s))\right)\, e^{-\lambda (s-t)}\;\dd s
\end{equation*}
and then
\begin{equation}\label{eq: up R}
    \tilde{u}_{\lambda}(\xi)
    \leq 
    R_{\mathfrak{c}_{2}}(\xi, \mathfrak{M}) + \varepsilon_{\circ}.
\end{equation}

Finally, from \eqref{eq: low R} and \eqref{eq: up R} one obtains
\begin{equation*}
\begin{aligned}
    R_{\mathfrak{c}_{1}}(\xi, \mathfrak{M}) 
    & \leq \tilde{u}_{\lambda}(\xi) + \varepsilon(t) 
    = (1+\eta(t))\tilde{u}_{\lambda}(\xi) + \varepsilon_{\circ}\\
    & \leq R_{\mathfrak{c}_{2}}(\xi, \mathfrak{M}) + \varepsilon(t)+\varepsilon_{\circ} \\
    & = R_{\mathfrak{c}_{2}}(\xi, \mathfrak{M})+ \eta(t)\tilde{u}_{\lambda}(\xi) +2\varepsilon_{\circ} \\
    & \leq  (1+\eta(t))\,R_{\mathfrak{c}_{2}}(\xi, \mathfrak{M}) + (2+ \eta(t))\varepsilon_{\circ},
\end{aligned}
\end{equation*}
that is
\begin{equation*}
    R_{\mathfrak{c}_{1}}(\xi, \mathfrak{M}) 
    \leq \tilde{u}_{\lambda}(\xi) + \varepsilon(t) 
    \leq R_{\mathfrak{c}_{2}}(\xi, \mathfrak{M}) + \varepsilon(t)+\varepsilon_{\circ},
\end{equation*}
or, recalling the definition of $\varepsilon(t)$ in Assumption \ref{quasi opt cost}, 
\begin{equation*}
    R_{\mathfrak{c}_{1}}(\xi, \mathfrak{M}) 
    \leq (1+\eta(t))\tilde{u}_{\lambda}(\xi) + \varepsilon_{\circ} 
    \leq (1+\eta(t))\,R_{\mathfrak{c}_{2}}(\xi, \mathfrak{M}) + (2+ \eta(t))\varepsilon_{\circ}.
\end{equation*}
It suffices now to observe, using Lemma \ref{lem: ricc}, that
\begin{equation*}
    R_{\mathfrak{c}_{i}}(\xi,\mathfrak{M}) = C_{i}\,\operatorname{dist}(y_{x}^{\alpha_{\varepsilon}}(t),\mathfrak{M})^{2} \quad \text{ where }  C_{i} = (-\lambda + \sqrt{\lambda^{2} + 4\mathfrak{c}_{i}}\,)/4, \text{ and } i=1,2,
\end{equation*}
which leads to
\begin{equation*}
    C_{1}\,\operatorname{dist}(y_{x}^{\alpha_{\varepsilon}}(t),\mathfrak{M})^{2} 
    \leq \tilde{u}_{\lambda}(\xi) + \varepsilon(t) 
    \leq C_{2}\,\operatorname{dist}(y_{x}^{\alpha_{\varepsilon}}(t),\mathfrak{M})^{2} + \varepsilon(t)+\varepsilon_{\circ},
\end{equation*}
or equivalently
\begin{equation*}
    C_{1}\,\operatorname{dist}(y_{x}^{\alpha_{\varepsilon}}(t),\mathfrak{M})^{2}
    \leq (1+\eta(t))\tilde{u}_{\lambda}(\xi) + \varepsilon_{\circ} 
    \leq (1+\eta(t))\,C_{2}\,\operatorname{dist}(y_{x}^{\alpha_{\varepsilon}}(t),\mathfrak{M})^{2} + (2+ \eta(t))\varepsilon_{\circ}.
\end{equation*}
\end{proof}

Our next result provides a guarantee to a similar condition as in Assumption \ref{main assumption} which will be useful for our main result in the upcoming section. 

\begin{proposition}\label{prop: for main assumption}
For a given $x\in\mathbb{R}^{n}$, let $y_{x}^{\alpha_{\varepsilon}}(\cdot)$ be a quasi-optimal trajectory for problem \eqref{OCP}. 
Suppose the following: 
\begin{enumerate}[label = -]
    \item Assumptions \ref{f: nice}, \ref{f: has min}, and \ref{f: stab} are satisfied, and $\;\mathfrak{M}$ is compact.
    \item Assumption \ref{quasi opt cost} holds for some $\eta(\cdot)\in L^{\infty}([0,\infty);\mathbb{R}_{+})$ and $\varepsilon_{\circ}\geq 0$.
    \item Assumption \ref{QG} holds for some $\mathbf{r}>0$ and $0<\mathfrak{c}_{1}\leq \mathfrak{c}_{2}<\infty$.
\end{enumerate}
Then there exist $\lambda>0$ small and $\tau = \tau(\mathbf{r},\varepsilon_{\circ}, \lambda,\|\eta\|_{\infty})>0$ such that $\;\forall\,t\geq\tau$,
\begin{equation*}
    \frac{\mathfrak{c}_{1}}{2\,C_{2}} \;\tilde{u}_{\lambda}(y_{x}^{\alpha_{\varepsilon}}(t)) \; \leq \; \tilde{f}(y_{x}^{\alpha_{\varepsilon}}(t))  + \,\frac{\mathfrak{c}_{1}}{2\,C_{2}}\,\varepsilon_{\circ},
\end{equation*}
where $C_{2}$ is the same constant as defined in Proposition \ref{prop: sandwich u}.  
In particular, when having an optimal trajectory, this result  holds with $\varepsilon_{\circ}=0$ and $\eta(\cdot)\equiv 0$, hence $\tau = \tau(\mathbf{r}, \lambda)$.
\end{proposition}

\begin{proof}
As in the beginning of the proof of Proposition \ref{prop: sandwich u}, the quasi-optimal trajectory satisfies for all $s\geq \tau$,  $\;y_{x}^{\alpha_{\varepsilon}}(s) \in \left\{z\in \mathbb{R}^{n}: \operatorname{dist}(z,\mathfrak{M})\leq \mathbf{r}\right\}$. Therefore on one hand, Assumption \ref{QG} guarantees
\begin{equation*}
    \frac{\mathfrak{c}_{1}}{2} \, \operatorname{dist}(y_{x}^{\alpha_{\varepsilon}}(s),\mathfrak M)^2 \le \; \tilde{f}(y_{x}^{\alpha_{\varepsilon}}(s)),  \quad\quad \forall\,s\geq \tau.
\end{equation*}
On the other hand, the right-hand side inequality in the conclusion of Proposition \ref{prop: sandwich u}  yields
\begin{equation*}
    \tilde{u}_{\lambda}(y_{x}^{\alpha_{\varepsilon}}(s)) \; \leq \; C_{2}\,\operatorname{dist}(y_{x}^{\alpha_{\varepsilon}}(s),\mathfrak{M})^{2} + \,\varepsilon_{\circ},  \quad\quad \forall\,s\geq \tau.
\end{equation*}
All together, one gets
\begin{equation*}
    \frac{\mathfrak{c}_{1}}{2\,C_{2}} \; \tilde{u}_{\lambda}(y_{x}^{\alpha_{\varepsilon}}(s)) \; \leq \; \tilde{f}(y_{x}^{\alpha_{\varepsilon}}(s))  + \,\frac{\mathfrak{c}_{1}}{2\,C_{2}}\,\varepsilon_{\circ} , \quad\quad \forall\,s\geq \tau.
\end{equation*}
In particular, when having an optimal control, this inequality holds with $\varepsilon_{\circ}=0$.
\end{proof}

Under Assumption \ref{QG}, once the optimal trajectory enters the neighborhood of the minimizer set where quadratic growth holds, Proposition \ref{prop: for main assumption} recovers Assumption \ref{main assumption} along the optimal trajectory with constant $K(\lambda, \mathfrak{c}_{1},\mathfrak{c}_{2})=\frac{2 \mathfrak{c}_{1}}{-\lambda+\sqrt{\lambda^2+4 \mathfrak{c}_{2}}}$. Therefore, the same variational argument applies from time $\tau$ onward.

\subsection{Exponential convergence to global minimizers}

We are now ready to state and prove our main convergence result. It shows that a (quasi-)optimal trajectory of the control problem  \eqref{OCP} converges exponentially fast to the set of global minimizers $\mathfrak{M}$, and that the corresponding objective values converge exponentially fast to the global minimum $\underline{f}$.

\begin{theorem}\label{thm: conv of traj}
For a given $x\in\mathbb{R}^{n}$, let $y_{x}^{\alpha_{\varepsilon}}(\cdot)$ be a quasi-optimal trajectory for problem \eqref{OCP}. Suppose the following: 
\begin{enumerate}[label = -]
    \item Assumptions \ref{f: nice}, \ref{f: has min}, and \ref{f: stab} are satisfied, and $\;\mathfrak{M}$ is compact.
    \item Assumption \ref{quasi opt cost} holds, with $\eta(\cdot)$ s.t. $\;\|\eta\|_{\infty} < 1- \lambda/K$, $\;K=\mathfrak{c}_{1}/(2C_{2})$, and $\varepsilon_{\circ}\geq0$. 
    \item Assumption \ref{QG} holds  with some constants $\mathbf{r}>0$ and $0<\mathfrak{c}_{1}\leq \mathfrak{c}_{2}<\infty$.
\end{enumerate}
Then, there exist $\lambda>0$ small and $\tau = \tau(\mathbf{r},\varepsilon_{\circ}, \lambda,\|\eta\|_{\infty})>0$, such that $\; \forall\,t\geq\tau$,
\begin{equation*}
    \operatorname{dist}(y_{x}^{\alpha_{\varepsilon}}(t),\mathfrak{M})^{2} 
    \; \leq \;  
    \mathfrak{a}\,e^{-\delta(t-\tau)}\;\operatorname{dist}(y_{x}^{\alpha_{\varepsilon}}(\tau),\mathfrak{M})^{2} + \mathcal{O}(\varepsilon_{\circ})
\end{equation*}
and
\begin{equation*}
    f(y_{x}^{\alpha_{\varepsilon}}(t))-\underline{f} 
    \; \leq \;  
    \frac{\mathfrak{c}_2\,\mathfrak{a}}{\mathfrak{c}_1}\,e^{-\delta(t-\tau)}\;(f(y_{x}^{\alpha_{\varepsilon}}(\tau))-\underline{f}) + \frac{\mathfrak{c}_2}{2}\;\mathcal{O}(\varepsilon_{\circ}),
\end{equation*}
where
\begin{equation*}
\begin{aligned}
    & \mathfrak{a}=\frac{C_{2}(1+\|\eta\|_{\infty})}{C_{1}}, \quad \delta = (1-\|\eta\|_{\infty})K - \lambda,\\
    \text{ and } \quad& \mathcal{O}(\varepsilon_{\circ}) = \frac{\varepsilon_{\circ}}{C_{1}}\left(1+\delta^{-1} + (2+\|\eta\|_{\infty})\, e^{-\delta(t-\tau)}\right).
\end{aligned}
\end{equation*}
The constants $C_{1}$ and $C_{2}$ are as defined in Proposition \ref{prop: sandwich u}. 
\end{theorem}

\begin{proof}
We proceed as in the proof of Theorem \ref{thm: conv value quasiopt} using the conclusion of Proposition \ref{prop: for main assumption} with $K:= \mathfrak{c}_{1}/(2C_{2})$.

Let $t\geq \tau$ be arbitrarily fixed. 
As in the proof of Theorem \ref{thm: conv value quasiopt}, recalling the notation of Assumption \ref{quasi opt cost}, let us define $c(\cdot)\in L^{1}_{\text{loc}}([0,\infty);\mathbb{R})$ such that 
\begin{equation*}
    c(t):=(1-\eta(t))K-\lambda \quad \text{ and } \quad C(s,t) := \int_{s}^{t}c(r)\dd r = (K-\lambda)(t-s) - K\int_{s}^{t}\eta(r)\dd r.
\end{equation*} 
Then the same computation leading to \eqref{ineq fron gronw in proof} still hold, the only difference being the range of integration when applying Gr\"onwall's inequality which is now between $\tau$ and $t\geq\tau$ (and not between $0$ and $t$). This yields 
\begin{equation}\label{ineq fron gronw in proof 2}
    \widetilde{\mathscr{J}}(y_{x}^{\alpha_{\varepsilon}}(t),\alpha_{\varepsilon}(t+\,\cdot\,)) \leq e^{-C(\tau,t)}\,\widetilde{\mathscr{J}}(y_{x}^{\alpha_{\varepsilon}}(\tau),\alpha_{\varepsilon}(\tau+\,\cdot\,)) + \varepsilon_{\circ} \int_{\tau}^{t}e^{-C(s,t)}\dd s, \quad \forall\,t\geq0.
\end{equation}
We have $\;\tilde{u}_{\lambda}(y_{x}^{\alpha_{\varepsilon}}(t)) \leq \widetilde{\mathscr{J}}(y_{x}^{\alpha_{\varepsilon}}(t),\alpha_{\varepsilon}(t+\,\cdot\,))$,\; and from Assumption \ref{quasi opt cost} we also have 
\begin{equation*}
   \widetilde{\mathscr{J}}(y_{x}^{\alpha_{\varepsilon}}(\tau),\alpha_{\varepsilon}(\tau+\,\cdot\,)) \leq \tilde{u}_{\lambda}(y_{x}^{\alpha_{\varepsilon}}(\tau)) + \varepsilon(\tau) = (1+\eta(\tau))\tilde{u}_{\lambda}(y_{x}^{\alpha_{\varepsilon}}(\tau)) + \varepsilon_{\circ}.
\end{equation*}
These last two inequalities, when together with \eqref{ineq fron gronw in proof 2}, yield
\begin{equation}\label{ineq quasi opt value in proof 2}
    \tilde{u}_{\lambda}(y_{x}^{\alpha_{\varepsilon}}(t)) \leq e^{-C(\tau,t)}\,(1+\eta(\tau))\tilde{u}_{\lambda}(y_{x}^{\alpha_{\varepsilon}}(\tau)) + \varepsilon_{\circ} \left(e^{-C(\tau,t)}\,+\int_{\tau}^{t}e^{-C(s,t)}\dd s\right).
\end{equation}
The assumption $\;\eta(t)\leq \|\eta\|_{\infty}< 1- \lambda/K\;$ for a.a. $t\geq 0$ guarantees that 
\begin{equation*}
    0< \delta \leq c(t) \quad\quad \text{ and } \quad\quad 0< \delta(t-s) \leq C(s,t).
\end{equation*}
where $\delta := (1-\|\eta\|_{\infty})K - \lambda$. \;  
Moreover, it holds
\begin{equation*}
\begin{aligned}
    \int_{\tau}^{t}e^{-C(s,t)}\dd s
    & = \int_{\tau}^{t}e^{-\delta(t-s)} \dd s = \frac{1 - e^{-\delta \,(t-\tau)}}{\delta} \leq \frac{1}{\delta}.
\end{aligned}
\end{equation*}
The inequality \eqref{ineq quasi opt value in proof 2} finally becomes
\begin{equation*}
    \tilde{u}_{\lambda}(y_{x}^{\alpha_{\varepsilon}}(t)) \leq (1+\eta(\tau))\, e^{-\delta\, (t-\tau)}\,\tilde{u}_{\lambda}(y_{x}^{\alpha_{\varepsilon}}(\tau)) + \varepsilon_{\circ} \left(e^{-\delta\,(t-\tau)}\,+\delta^{-1}\right), \quad \forall\,t\geq0.
\end{equation*}
We now apply  Proposition \ref{prop: sandwich u} to obtain, $\; \forall\,t\geq\tau$,
\begin{equation*}
\begin{aligned}
    C_{1}\,\operatorname{dist}(y_{x}^{\alpha_{\varepsilon}}(t),\mathfrak{M})^{2}  
    & \leq \tilde{u}_{\lambda}(y_{x}^{\alpha_{\varepsilon}}(t)) + \varepsilon_{\circ} \\
    & \leq \, (1+\eta(\tau))\, e^{-\delta\, (t-\tau)}\,\tilde{u}_{\lambda}(y_{x}^{\alpha_{\varepsilon}}(\tau)) + \varepsilon_{\circ} \left(1+e^{-\delta\,(t-\tau)}\,+\delta^{-1}\right) \\
    & \leq \; (1+\eta(\tau))\, e^{-\delta(t-\tau)}\,\left[C_{2}\,\operatorname{dist}(y_{x}^{\alpha_{\varepsilon}}(\tau),\mathfrak{M})^{2} + \,\varepsilon_{\circ}\right] 
    + \varepsilon_{\circ} \left(1+e^{-\delta(t-\tau)}\,+\delta^{-1}\right) \\
    & = C_{2}\,(1+\eta(\tau))\, e^{-\delta(t-\tau)}\,\operatorname{dist}(y_{x}^{\alpha_{\varepsilon}}(\tau),\mathfrak{M})^{2} + 
    \varepsilon_{\circ}\left(1+\delta^{-1} + (2+\eta(\tau))\, e^{-\delta(t-\tau)}\right).
    \end{aligned}
\end{equation*}
Finally, by Assumption \ref{QG}, for every $z\in \left\{z\in\mathbb{R}^{n}: \operatorname{dist}(z,\mathfrak{M})\leq \mathbf{r}\right\}$, it holds that 
\begin{equation*}
    f(z)-\underline{f}\le \frac{\mathfrak c_2}{2}\,\operatorname{dist}(z,\mathfrak M)^2\leq\frac{\mathfrak{c}_2}{\mathfrak{c}_1}(f(z)-\underline{f}).
\end{equation*}
Combining  with the previously established bound on $\operatorname{dist}(y_x^{\alpha_\varepsilon}(t),\mathfrak M)^2$, the desired result on $f(y_x^{\alpha_\varepsilon}(t))-\underline{f}$ follows.
\end{proof}

It should be noted that one has $\;\;\operatorname{dist}(y_{x}^{\alpha_{\varepsilon}}(\tau),\mathfrak{M})^{2}\,\leq \mathbf{r}^2 \, \leq \, \operatorname{dist}(x,\mathfrak{M})^{2},$ \; 
hence the inequality in Theorem \ref{thm: conv of traj} can be equivalently expressed as 
\begin{equation*}
    \operatorname{dist}(y_{x}^{\alpha_{\varepsilon}}(t),\mathfrak{M})^{2} 
    \; \leq \;  
    \mathfrak{a}\,e^{-\delta(t-\tau)}\;\operatorname{dist}(x,\mathfrak{M})^{2} + \mathcal{O}(\varepsilon_{\circ}) \quad \text{ for all } t\geq \tau,
\end{equation*}
and so is the case in the subsequent results. 

A direct consequence of Theorem \ref{thm: conv of traj} is next.

\begin{corollary}\label{cor: conv of traj}
For a given $x\in\mathbb{R}^{n}$, let $y_{x}^{*}(\cdot)$ be an optimal trajectory for problem \eqref{OCP} as given by Theorem \ref{thm: opt cont}. Suppose the following: 
\begin{enumerate}[label = -]
    \item Assumptions \ref{f: nice}, \ref{f: has min}, and \ref{f: stab} are satisfied, and $\;\mathfrak{M}$ is compact.
    \item Assumption \ref{QG} holds  with some constants $\mathbf{r}>0$ and $0<\mathfrak{c}_{1}\leq \mathfrak{c}_{2}<\infty$.
\end{enumerate}
Then there exist $\lambda>0$ small and $\tau = \tau(\mathbf{r}, \lambda)>0$ such that 
\begin{equation*}
    \operatorname{dist}(y_{x}^{*}(t),\mathfrak{M})^{2} 
    \;  \leq \;  
    \mathfrak{a}\,e^{-\delta(t-\tau)}\;\operatorname{dist}(y_{x}^{*}(\tau),\mathfrak{M})^{2}, \quad \quad \forall\,t\geq\tau,
\end{equation*}
\begin{equation*}
    f(y_{x}^{*}(t))-\underline{f} 
    \; \leq \;  
    \frac{\mathfrak{c}_2\,\mathfrak{a}}{\mathfrak{c}_1}\,e^{-\delta(t-\tau)}\;(f(y_{x}^{*}(\tau))-\underline{f}),\quad \quad \forall\,t\geq\tau,
\end{equation*}
and
\begin{equation*}
    |\dot{y}_{x}^{*}(t)|^{2} 
    \;  \leq \;  
    \mathfrak{c}_{2}\,\mathfrak{a}\,e^{-\delta(t-\tau)}\;\operatorname{dist}(y_{x}^{*}(\tau),\mathfrak{M})^{2} \quad \text{ for a.a. } t\geq \tau,
\end{equation*}
where
\begin{equation*}
\begin{aligned}
    & \mathfrak{a}=\frac{C_{2}}{C_{1}}, \quad \text{ and }  \quad \delta = K - \lambda.
\end{aligned}
\end{equation*}
The constants $C_{1}$ and $C_{2}$ are as defined in Proposition \ref{prop: sandwich u}.
\end{corollary}

\begin{remark}\label{rem: ae}
About the ``\textit{almost all} $t\geq \tau$'' requirement in Corollary \ref{cor: conv of traj}: If one assumes $f(\cdot)$ to be in addition semi-concave, then the value function $u_{\lambda}(\cdot)$ is also semi-concave \cite[Lemma 2]{bardi2023eikonal}. In this case, the value function is differentiable \textit{everywhere} along an optimal trajectory \cite[Theorem 6.4.7, page 165]{cannarsa2004semiconcave}. Consequently, the second inequality in Corollary \ref{cor: conv of traj} will be true ``\textit{for all} $t\geq \tau$''. This is also the case in  Theorem \ref{thm: turnpike}. We should note that the semi-concavity is a curvature upper bound, which does not conflict with any type of convexity.
\end{remark}

\begin{proof}[Proof of Corollary \ref{cor: conv of traj}]
The convergence of both the trajectory and the function values is an immediate consequence of Theorem \ref{thm: conv of traj} where we set $\eta(\cdot)\equiv 0$ and $\varepsilon_{\circ}=0$. Only the convergence of the velocity $\dot{y}_{x}^{*}(\cdot)$ is needed. To this aim, we shall use the HJB equation \eqref{eq: HJB} together with the characterization of an optimal control in Theorem \ref{thm: opt cont} that is $\dot{y}_{x}^{*}(t) = \alpha^{*}(t) = -Du_{\lambda}(y_{x}^{*}(t))$ for almost every $t\geq 0$. Then using the HJB equation, one has
\begin{equation*}
    \lambda(u_{\lambda}(y_{x}^{*}(t)) - \underline{f}/\lambda\,) + \frac{1}{2}|Du_{\lambda}(y_{x}^{*}(t))|^{2} = f(y_{x}^{*}(t)) - \underline{f}\quad \text{ for a.a. } t\geq 0.
\end{equation*}

Let $\mathbf{r}>0$ be fixed, and $0<\mathfrak{c}_{1}\leq \mathfrak{c}_{2}<\infty$ be given as in Assumption \ref{QG}. 
The existence of $\lambda>0$ small and $\tau = \tau(\mathbf{r},\lambda)>0$  such that the optimal trajectory satisfies for all $t\geq \tau$,  $\;y_{x}^{*}(t) \in \left\{z\in \mathbb{R}^{n}: \operatorname{dist}(z,\mathfrak{M})\leq \mathbf{r}\right\}$, is guaranteed by Theorem \ref{thm: practical} together with Remark \ref{rem: (C)}. 

On the other hand, from Proposition \ref{prop: min_equality}, we know that $ \lambda(u_{\lambda}(\cdot) - \underline{f}/\lambda\,)\geq 0$, and using the second inequality in Assumption \ref{QG}, one gets
\begin{equation*}
     |\alpha^{*}(t)|^{2} \leq \mathfrak{c}_{2} \, \operatorname{dist}(y_{x}^{*}(t),\mathfrak M)^2 \quad \text{ for a.a. } t\geq \tau.
\end{equation*}
Therefore, using the first statement of this corollary, one obtains
\begin{equation*}
    \operatorname{dist}(y_{x}^{*}(t),\mathfrak{M})^{2} 
    \;  \leq \;  
    \mathfrak{a}\,e^{-\delta(t-\tau)}\;\operatorname{dist}(y_{x}^{*}(\tau),\mathfrak{M})^{2}.
\end{equation*}
Therefore, one obtains
\begin{equation*}
     |\alpha^{*}(t)|^{2} \leq \mathfrak{c}_{2}\,\mathfrak{a} \, e^{-\delta(t-\tau)}\;\operatorname{dist}(y_{x}^{*}(\tau),\mathfrak M)^2 \quad \text{ for a.a. } t\geq \tau.
\end{equation*}
where  $\mathfrak{a}=C_{2}/C_{1}$ as in Theorem \ref{thm: conv of traj}.
\end{proof}

\subsection{A Turnpike-type result}

An interpretation of Corollary \ref{cor: conv of traj} that deserves particular attention is its connection with the Turnpike phenomenon. Indeed, Corollary \ref{cor: conv of traj} shows that any optimal control–trajectory pair (as in Theorem \ref{thm: opt cont}) converges exponentially fast toward its stationary counterpart. This interpretation is formalized in the following theorem.

\begin{theorem}\label{thm: turnpike}
For a given $x\in\mathbb{R}^{n}$, let $(\alpha^{*}(\cdot), y_{x}^{*}(\cdot))$ be an optimal control-trajectory for problem \eqref{OCP} as in Theorem \ref{thm: opt cont}. Suppose the following: 
\begin{enumerate}[label = -]
    \item Assumptions \ref{f: nice}, \ref{f: has min}, and \ref{f: stab} are satisfied, and $\;\mathfrak{M}$ is compact.
    \item Assumption \ref{QG} holds  with some constants $\mathbf{r}>0$ and $0<\mathfrak{c}_{1}\leq \mathfrak{c}_{2}<\infty$.
\end{enumerate}
Then there exist $\lambda>0$ small and $\tau = \tau(\mathbf{r}, \lambda)>0$ such that 
\begin{equation*}
    |\alpha^{*}(t)|^{2} \,+\, \operatorname{dist}(y_{x}^{*}(t),\mathfrak{M})^{2} 
    \; \leq \;  
    \tilde{\mathfrak{a}}\,e^{-\delta(t-\tau)}\;\operatorname{dist}(y_{x}^{*}(\tau),\mathfrak{M})^{2}  \quad \quad \text{ for a.a. } t\geq\tau,
\end{equation*}
where
\begin{equation*}
\begin{aligned}
    & \tilde{\mathfrak{a}} = (1+\mathfrak{c}_{2})\frac{C_{2}}{C_{1}}, \quad  \quad \delta = K - \lambda, \quad \text{ and } \quad K=\mathfrak{c}_{1}/(2C_{2}).
\end{aligned}
\end{equation*}
The constants $C_{1}$ and $C_{2}$ are as defined in Proposition \ref{prop: sandwich u}.
\end{theorem}

We refer to Remark \ref{rem: ae} for a discussion about the ``\textit{almost all} $t\geq \tau$'' requirement.

It is easy to note that the same proof also provides
\begin{equation*}
    |\alpha^{*}(t)| + \operatorname{dist}(y_{x}^{*}(t),\mathfrak{M}) 
    \; \leq \;  
    \hat{\mathfrak{a}}\,e^{-\frac{\delta}{2}(t-\tau)}\;\operatorname{dist}(y_{x}^{*}(\tau),\mathfrak{M})  \quad \quad \text{ for a.a. } t\geq\tau,
\end{equation*}
where $\hat{\mathfrak{a}} = (1+\sqrt{\mathfrak{c}_{2}})\sqrt{C_{2}/C_{1}}$.

The connection with the Turnpike property is straightforward (e.g., see \cite{porretta2013long}, the book \cite{zaslavski2015turnpike}, a link with HJB in  \cite{esteve2022turnpike}, and the recent survey \cite{trelat2025turnpike} with the many references therein). Indeed, to the optimal control problem \eqref{OCP}, is formally associated a stationary (finite dimensional) problem
\begin{equation*}
    \min\limits_{(a,\;x)\in \mathbb{R}^{n}\times\mathbb{R}^{n}} \quad 
    \frac{1}{2}|a|^{2} + f(x), \quad \text{s.t.: } 0 = a,
\end{equation*}
which corresponds exactly to the optimization problem that we are seeking to solve
\begin{equation*}
    \min\limits_{x\in\mathbb{R}^{n}} \, f(x).
\end{equation*}
Therefore, the Turnpike theory tells us that both the optimal control and trajectory should converge (exponentially) towards the steady optimal control-trajectory pair, which are in the present case the null control $0$ and the set $\mathfrak{M}$ of global minimizers of $f(\cdot)$. This is the statement of Theorem \ref{thm: turnpike}.

Finally, the behavior of the trajectory during the time interval $[0,\tau]$ is discussed earlier in \S \ref{sec: stab}. This interval is also commonly referred to as the ``first arc'' in the terminology of the Turnpike theory.

\subsection{Controllability and linear growth}

We next provide two assumptions which, as we will show, guarantee that Assumption \ref{main assumption} holds.

\begin{enumerate}[label=\textbf{(E)}]
    \item\label{Cont} For any $x \in \mathbb{R}^{n}$, there exists a control $\hat{\alpha}(\cdot)$ (not necessarily optimal) which drives $x$ to some point $x^{*} \in \mathfrak{M}$ in time $t(x)\leq \beta(f(x)-\underline{f})$, where $\beta>0$ is a constant.
\end{enumerate}
\begin{enumerate}[label=\textbf{(F)}]
    \item\label{LG} There exist $\mathbf{r}>0$  and $C>0$  such that
\begin{equation*}%\label{lin gr C}
    \operatorname{dist}(x,\mathfrak{M}) \leq C\, (f(x)-\underline{f}), \quad\forall\,x \in \left\{x\in\mathbb{R}^{n}: \operatorname{dist}(x,\mathfrak{M})\leq \mathbf{r}\right\}.
\end{equation*}
\end{enumerate}

\begin{remark}\label{rem: stc}
When the set $\mathfrak{M}$ is closed (or, when $f(\cdot)$ is continuous) and non-empty, 
Assumption \ref{Cont} could be seen as a strong controllability assumption, as it implies \textit{small-time controllability} defined in \cite[Definition IV.1.1]{bardi1997optimal}. Indeed, let us choose the constant control  
\begin{equation}\label{alpha hat}
    \hat{\alpha}(s) = \frac{x^{*} - x}{|x^{*} - x|}M, \quad \quad \forall\, s\in [0, \, t(x)],
\end{equation}
where $x^{*}$ is the projection (not necessarily unique) of $x$ onto $\mathfrak{M}$ which exists thanks to closedness. Then, our dynamics is $\dot{y}_{x}^{\hat{\alpha}}(s) = \hat{\alpha}(s)$ with $y_{x}^{\hat{\alpha}}(0)=x$. Integrating over $[0,t(x)]$ yields $y_{x}^{\hat{\alpha}}(t(x)) - y_{x}^{\hat{\alpha}}(0) = \frac{x^{*} - x}{|x^{*} - x|}M \, (t(x)-0)$ and by definition of $t(x)$, we have $y_{x}^{\hat{\alpha}}(t(x)) = x^{*}$. Thus it holds $x^{*} -x = \frac{x^{*} -x}{|x^{*} - x|}M\, t(x)$. Taking the norm on both sides finally yields $t(x) = \frac{1}{M}|x-x^{*}|$. 
Now recall from Remark \ref{rem: bounded controls} that the constant $M$ can be chosen as large as we want, which then makes $t(x)$ as small as we want, and thus the \textit{small-time controllability}. 
\end{remark}

\begin{proposition}\label{prop: cont gaits}
If Assumptions \ref{f: has min} and \ref{Cont} hold, then Assumption \ref{main assumption} is satisfied. 
\end{proposition}

\begin{proof}
Consider the constant control \eqref{alpha hat} introduced in Remark~\ref{rem: stc}, with $M$ fixed, and  $x^{*}\in \mathfrak{M}$ exists thanks to Assumption~\ref{f: has min}, moreover, $\hat{u}(s) = 0$ whenever $s> t(x)$.  Using the fact that it is admissible (without necessarily begin optimal), we have
\begingroup
\allowdisplaybreaks
%\begin{equation*}
\begin{align*}
    u_{\lambda}(x) 
    & \leq \int_{0}^{\infty} \left(\frac{1}{2}|\hat{\alpha}(s)|^{2} + f(y_{x}^{\hat{\alpha}}(s))\right)e^{-\lambda s}\, \dd s\\
    & \leq \int_{0}^{t(x)} \left(\frac{1}{2}|\hat{\alpha}(s)|^{2} + f(y_{x}^{\hat{\alpha}}(s))\right)e^{-\lambda s}\, \dd s + \underline{f}\,\frac{e^{-\lambda \,t(x)}}{\lambda}\\
    & \leq  \left(\frac{1}{2}M^{2} + \overline{f}\right)\,t(x) + \underline{f}\,\frac{1}{\lambda}\\
    & \leq \beta\left(\frac{1}{2}M^{2} + \overline{f}\right)\,\big(f(x) - \underline{f}\big) + \underline{f}\,\frac{1}{\lambda} \quad \text{ thanks to \ref{Cont}.}
\end{align*}
\endgroup
Therefore, it holds
\begin{equation*}
    \widetilde{K}\left(u_{\lambda}(x) - \underline{f}\,\frac{1}{\lambda}\right)
    \leq f(x) - \underline{f}, \quad \forall\, x,
\end{equation*}
 where $\widetilde{K} = \frac{1}{\beta}\left(\frac{1}{2}M^{2} + \overline{f}\right)^{-1}$. It suffices now to choose $\lambda$ small such that $\widetilde{K}> \lambda>0$. 
\end{proof}

\begin{proposition}\label{prop: contr implies lin}
Suppose Assumptions \ref{f: nice} and \ref{f: has min} hold. 
Then 
Assumption \ref{LG} is equivalent to Assumption \ref{Cont}.
\end{proposition}

\begin{proof}
Observe first that continuity of $f(\cdot)$ from Assumption \ref{f: nice} guarantees that the set of its global minimizers $\mathfrak{M}$ is closed, and Assumption \ref{f: has min} ensures that $\mathfrak{M}$ is non-empty. Hence, for any $x\in\mathbb{R}^{n}$, we can define its projection (possibly not unique) onto $\mathfrak{M}$. 

Now let us define 
$t(x) = \frac{1}{M}|x-x^{*}|$ as in Remark \ref{rem: stc} where $x^{*}$ is a projection of $x$ onto $\mathfrak{M}$. We have $t(x) = \frac{1}{M}\operatorname{dist}(x,\mathfrak{M})$. When put together with Assumption \ref{Cont}, it implies  
\begin{equation*}%\label{lin gr}
    \operatorname{dist}(x,\mathfrak{M}) \leq \beta \,M\, (f(x)-\underline{f}).
\end{equation*}

Conversely, if Assumption \ref{LG} is satisfied for some constant $C>0$, then one could choose as a control the projection of $x$ onto the set $\mathfrak{M}$ with any constant speed $M>0$, that is $\hat{\alpha}(s) = \frac{x^{*} - x}{|x^{*} - x|}M$ which corresponds to straight lines. In this case, we have again $M\,t(x) = |x-x^{*}|$, and since $\operatorname{dist}(x,\mathfrak{M}) = |x-x^{*}|$, it holds
\begin{equation*}
    M\,t(x) = \operatorname{dist}(x,\mathfrak{M}) \leq C\, (f(x)-\underline{f}),
\end{equation*}
hence Assumption \ref{Cont} is satisfied with $\beta = C/M$.
\end{proof}

\begin{proposition}
Suppose Assumptions \ref{f: nice}, \ref{f: has min}, and \ref{LG} hold. Then Assumption \ref{main assumption} is satisfied. 
\end{proposition}

\begin{proof}
This is a consequence of  Proposition \ref{prop: contr implies lin} and Proposition \ref{prop: cont gaits}.
\end{proof}

A direct consequence is the following.

\begin{proposition}\label{prop: equiv to B}
Suppose Assumptions \ref{f: nice} and \ref{f: has min} hold. 
Then the results of Theorem \ref{thm: conv value opt}, Theorem \ref{thm: conv value quasiopt}, and Theorem \ref{thm: conv value quasiopt delta} remain valid if we assume either  \ref{Cont} or \ref{LG}, instead of \ref{main assumption}. 
\end{proposition}

\section{Connections to classical optimization theory}\label{sec: comparison}

\subsection{Polyak--{\L}ojasiewicz condition}\label{sec: PL}

Recalling the HJB equation \eqref{eq: HJB}, 
Assumption \ref{main assumption} implies 
\begin{equation}\label{PL ineq}
\frac{1}{2}\left|D u_\lambda(x)\right|^2 \geq(K-\lambda) u_\lambda(x) \quad \text{for a.a. } x\in\mathbb{R}^{n},
\end{equation}
and 
\begin{equation*}
\frac{1}{2}\left|D \tilde{u}_\lambda(x)\right|^2 \geq(K-\lambda) \tilde{u}_\lambda(x) \quad \text{for a.a. } x\in\mathbb{R}^{n},
\end{equation*}
which are  Polyak–\L ojasiewicz (PL)–type inequalities for $u_\lambda(\cdot)$ and $\tilde{u}_{\lambda}(\cdot)$ respectively. Conversely, if \eqref{PL ineq} holds, then using the HJB equation \eqref{eq: HJB} one gets $K \,\tilde{u}_{\lambda}(x) \leq \tilde{f}(x)$ for a.a. $x\in\mathbb{R}^{n}$. When $u_{\lambda}(\cdot)$ is a classical solution to \eqref{eq: HJB}, the PL-inequality \eqref{PL ineq} holds everywhere in $\mathbb{R}^{n}$ and Assumption \ref{main assumption} becomes equivalent to \eqref{PL ineq}.

\subsection{Smooth functions}

Stronger conditions than Assumption \ref{QG} are the following.  

\begin{enumerate}[label=\textbf{(G\arabic*)}]
    \item\label{L smooth} There exist $\mathbf{r}>0$  and $0<\mathfrak{c}_{1}\leq \mathfrak{c}_{2}<\infty$  such that the function $f(\cdot)$ is $L$-smooth $\forall\,x \in \left\{x\in\mathbb{R}^{n}: \operatorname{dist}(x,\mathfrak{M})\leq \mathbf{r}\right\}$. 
    That is, $f(\cdot)$ is differentiable, and its gradient is $L$-Lipschitz in a bounded region around the minimizers. 
    \item\label{mu conv} There exist $\mathbf{r}>0$  and $0<\mathfrak{c}_{1}\leq \mathfrak{c}_{2}<\infty$  such that the function $f(\cdot)$ is $\mu$-strongly convex in $\left\{x\in\mathbb{R}^{n}: \operatorname{dist}(x,\mathfrak{M})\leq \mathbf{r}\right\}$. 
    That is, 
    \begin{equation*}
        f(\theta x + (1-\theta)y) \leq \theta\,f(x) + (1-\theta)\,f(y) - \frac{\mu}{2}\theta(1-\theta)|x-y|^{2}, \quad \forall\,\theta\in[0,1]
    \end{equation*}
    for all $x,y$ in $\left\{x\in\mathbb{R}^{n}: \operatorname{dist}(x,\mathfrak{M})\leq \mathbf{r}\right\}$. 
\end{enumerate}

Assumption \ref{L smooth} implies in particular the following inequality
\begin{equation*}
    f(y) \leq f(x) + \langle \nabla f(x), y-x\rangle + \frac{L}{2}\,|y-x|^{2}.
\end{equation*}

If one assumes $f(\cdot)$ to be differentiable, then Assumption \ref{mu conv} becomes
\begin{equation*}
    f(y) \geq f(x) + \langle \nabla f(x), y-x\rangle + \frac{\mu}{2}\,|y-x|^{2}.
\end{equation*}

If both Assumptions \ref{L smooth} and \ref{mu conv} hold, then $f(\cdot)$ has a unique minimizer, i.e. $\mathfrak{M}=\{x^{*}\}$ where $f(x^{*})=\underline{f}$, moreover 
\begin{equation*}
   \frac{\mu}{2}\,|x-x^{*}|^{2} \;\leq\; f(x) - \underline{f} \;\leq\; \frac{L}{2}\ |x-x^{*}|^{2}, \quad\forall\,x \in \left\{x\in\mathbb{R}^{n}: \operatorname{dist}(x,\mathfrak{M})\leq \mathbf{r}\right\}.
\end{equation*}
Indeed, the uniqueness of the minimizer is a consequence of the inequality in \ref{mu conv}.

Assumption \ref{QG} is weaker as it does not require differentiability of the function 
$f(\cdot)$, and it also accommodates the presence of multiple minimizers, possibly forming a non-trivial set.

\subsection{Metric regularity}

We show the connection between \ref{QG} and the classical notion of metric regularity \cite[Definition 2.80, p. 58]{bonnans2013perturbation}, which is related to openess \cite[Theorem 2.81, p. 58]{bonnans2013perturbation}. See also \cite[Chapter 9, \S G]{rockafellar1998variational}, in particular Theorem 9.43 therein where other equivalent properties are given. 

More precisely, we shall verify that the lower-bound in Assumption \ref{QG} implies metric regularity of the subdifferential of $f(\cdot)$, meaning that subgradients cannot vanish faster than linearly with respect to the distance to minimizers. This measures the steepness of the function near $\mathfrak{M}$. 

Recall that a set-valued map $\Psi: X \to 2^{Y}$ is \textit{metric regular} at a point $(x_{0}, y_{0}) \in \operatorname{Graph}(\Psi)$ at a rate $\kappa$, if for all $(x,y)$ in a neighborhood of  $(x_{0}, y_{0})$ it holds
\begin{equation*}
    \operatorname{dist}(x,\Psi^{-1}(y)) \leq \kappa\,\operatorname{dist}(y, \Psi(x)).
\end{equation*}
Recall the graph of a set-valued function is $\operatorname{Graph}(\Psi) := \{(x,y) \in X\times Y\,:\, x\in \Psi(y)\}$. 

Suppose $f(\cdot)$ is subdifferentiable in a neighborhood of $\mathfrak{M}$, and denote its subdifferential $\partial f: \mathbb{R}^{n}\to 2^{\mathbb{R}^{n}}$ defined at some $x$ by
\begin{equation*}
    \partial f(x) :=\{\, v\in \mathbb{R}^{n}\, :\, f(x) - f(y) \leq \langle v, x-y \rangle,\; \forall\, y\in \mathbb{R}^{n} \,\}.
\end{equation*}
Let $(x^{*},0) \in \operatorname{Graph}(\partial f)$ such that $x^{*}\in \mathfrak{M}$, and $0\in \partial f(x^{*})$. Then $f(x^{*}) = \underline{f}$ and it holds
\begin{equation*}
    f(x) - \underline{f} \leq \langle v, x-x^{*} \rangle \leq |v| \,|x-x^{*}|, \quad \forall\, v\in \partial f(x).
\end{equation*}
Taking the infimum over $v\in \partial f(x)$ and over $x^{*}\in \mathfrak{M}$ yields
\begin{equation*}
    f(x) - \underline{f} \leq \operatorname{dist}(0,\partial f(x))\, \operatorname{dist}(x,\mathfrak{M}),
\end{equation*}
where $\operatorname{dist}(0,\partial f(x)) = \inf \{|v| \,:\, v\in \partial f(x)\}$ is the so-called subgradient residual which precisely measures the steepness of the function near $\mathfrak{M}$. Using the lower-bound in Assumption \ref{QG}, and noting that $\partial^{-1}f(0) = \mathfrak{M}$, one gets
\begin{equation*}
    \operatorname{dist}(x,\partial^{-1}f(0)) \leq \; \frac{2}{\mathfrak{c}_{1}} \; \operatorname{dist}(0,\partial f(x)),
\end{equation*}
which holds for all $(x,0)$ in a neighborhood of $(x^{*},0)$. This is the metric regularity of the set-valued map $\partial f(\cdot)$ at all points $(x^{*},0)$ where $x^{*}\in \mathfrak{M}$, with rate $2/\mathfrak{c}_{1}$.

\subsection{Applications and examples}
\label{sec: ex}

We now consider several well-known examples from the literature and verify that each one adheres to our framework, with particular attention to Assumption \ref{QG}.

Although the following examples involve functions $f(\cdot)$ that are unbounded, in practical optimization it is customary to restrict the search to a sufficiently large but bounded region. Accordingly, one may truncate $f(\cdot)$ outside this region without loss of generality.

Moreover, some examples admit a non-compact set of global minimizers $\mathfrak{M}$. This issue can likewise be handled by truncating the domain to the bounded search region used in practice. Under such truncation, Assumption \ref{f: stab} is automatically satisfied, since we define $\;\liminf\limits_{|x|\to \infty} f(x) =: \overline{f}\;$ the truncation level of the loss function, chosen so that $\overline{f}> \underline{f}$, see \eqref{equiv to stab}.

\textbullet\; \textit{\underline{Example 1}. Strongly convex quadratic functions} 

Consider $f(x) = \frac{1}{2}x^{\top} Q x + b^{\top}x + c$ where $Q$ is symmetric positive definite with smallest eigenvalue $\mu>0$. Then $\mathfrak{M} = \{-Q^{-1}b =: x^{*}\}$, and
\begin{equation*}
	f(x) -f(x^{*}) = \frac{1}{2}(x - x^{*})^{\top} Q (x-x^{*}).
\end{equation*}
So
\begin{equation*}
	\frac{\mu}{2} |x-x^{*}|^{2} \leq f(x) -f(x^{*}) \leq \frac{L}{2} |x-x^{*}|,
\end{equation*}
where $L$ is the largest eigenvalue of $Q$. 

Representative examples: least-squares regression, ridge regression, LQR cost functions in control, and the Kalman filter cost function.

\textbullet\; \textit{\underline{Example 2}. Non-convex quadratic functions with a flat minimizer set}

Consider $f(x_{1}, x_{2}) = \frac{1}{2}(x_{1} - x_{2})^{2}$. Here the set of minimizers is the line $\mathfrak{M} = \{(x_{1},x_{2})\in \mathbb{R}^{2} \,:\, x_{1} = x_{2}\}$. Then 
\begin{equation*}
	f(x_{1},x_{2}) = \frac{1}{2} \operatorname{dist}(x,\mathfrak{M})^{2}
\end{equation*}
and the inequality holds with $\mathfrak{c}_{1}=\mathfrak{c}_{2} = 1$.

Representative examples: Consensus optimization, distributed averaging, synchronisation in multi-agent systems, and over-parametrized models in machine learning.

\textbullet\; \textit{\underline{Example 3}. Regularized losses and penalized problems}

Consider  $f(x) = |Ax - b|^{2} + \lambda|x|^{2}$. Here $f(\cdot)$ is strictly convex, and thus the quadratic growth holds globally with constants related to the smallest/largest eigenvalues of $A^{\top}A + \lambda \mathbb{I}_{n}$. 

Representative examples: machine learning regularization, statistical estimation, and robust control.

\textbullet\; \textit{\underline{Example 4}. Double-well potentials}

Consider  $f(x) = (x^{2} - 1)^{2} = x^{4} - 2x^{2} + 1$. The set of global minimizers is $\mathfrak{M} = \{-1, 1\}$ with $\underline{f} = 0$. 

Around a minimizer, for example $1$, it holds $f(1 + \delta) = 4\delta^{2} + 4\delta^{3} + \delta^{4}$. Hence for sufficiently small $|\delta|$, it holds $\; 2 \operatorname{dist}(x,\mathfrak{M})^{2} \leq f(x) \leq 5\operatorname{dist}(x,\mathfrak{M})^{2} \;$ for all $x$ such that $ |x-1| \leq \delta$ and $\delta =0.4$.

Representative examples: problems with phase transitions, bistable control systems, symmetry-breaking potentials in physics and machine learning (scalar neural activations, binary classification energy models).

\textbullet\; \textit{\underline{Example 5}. Matrix factorization and low-rank learning}

Consider  $f(X,Y) = \frac{1}{2} \|XY^{\top} - M^{*}\|^{2}_{F}$, where $X,Y\in \mathbb{R}^{n\times r}$ and $\|\cdot\|_{F}$ is the Frobenius norm. This is a non-convex problem but satisfies local quadratic growth around the set of global minimizers $\mathfrak{M} =\{(X,Y)\,:\, XY^{\top} = M^{*}\}$. The constants $\mathfrak{c}_{1},\mathfrak{c}_{2}$ depend on the singular values of $M^{*}$. 
This is also known as the \textit{Restricted Isometry Property}, see \cite[Definition 3.1]{tu2016low} and references therein.

Representative examples: matrix completion, dictionary learning, phase retrieval, and Principal Component Analysis.

\textbullet\; \textit{\underline{Example 6}. Neural Network Losses with symmetries} 

Consider  $f(x_{1}, x_{2}) = \frac{1}{2}(x_{1}x_{2} - 1)^{2}$. This is a non-convex function with minimizer set a hyperbola $\mathfrak{M} = \{(x_{1},x_{2}) \,:\, x_{1}x_{2} = 1\}$. One could perform a Taylor expansion around any such points $(\omega_{1},\omega_{2})\in \mathfrak{M}$ and obtain $f(x_{1},x_{2}) \approx \frac{1}{2}(\omega_{1} \delta_{2} + \omega_{2}\delta_{1})^{2}$, where $(\delta_{1}, \delta_{2}) = (x_{1} - \omega_{1}, x_{2} - \omega_{2})$. Hence $f(\cdot)$ grows quadratically with the distance to $\mathfrak{M}$. 

Representative examples: over-parameterized networks, deep linear networks, and matrix sensing with factorized parametrizations.

\textbullet\; \textit{\underline{Example 7}. Trigonometric/periodic landscapes}

Consider $f(x) = 1- \cos(x)$. The set of global minimizers is $\mathfrak{M}=2\pi\mathbb{Z}$, and $\underline{f} = 0$. Near any minimizer $x^{*} = 2\pi k$ with $k\in \mathbb{Z}$, on has $f(x) = \frac{1}{2}(x-x^{*}) + \mathcal{O}(|x-x^{*}|^{4})$, and the quadratic growth thus holds locally. 

Representative examples: synchronization problems, phase-locked loops, and periodic optimal control.

\textbullet\; \textit{\underline{Example 8}. $k$-means objectives}

Given data points $a_{1}, \dots, a_{N} \in \mathbb{R}^{n}$, the $k$-means objective is
\begin{equation*}
    f(x_{1},\dots, x_{k}) = \frac{1}{N} \sum\limits_{i=1}^{N} \min\limits_{1\leq j \leq k} |a_{i} - x_{j}|^{2}.
\end{equation*}
Here the variables $x_{1}, \dots, x_{k}$ are the cluster centers. 

We suppose that clusters are well-separated, that is: if $(x_{1}^{*},\dots, x_{k}^{*})$ is a global minimizer, then there exists a positive margin $\gamma>0$ such that every point $a_{i}$ is strictly closer to its assigned center $x_{p(i)}^{*}$ than to any other center
\begin{equation*}
	|a_{i} - x_{p(i)}^{*}|^{2} + \gamma \leq |a_{i} - x^{*}_{j}|, \quad \forall\, i,\, \forall\, j\neq p(i).
\end{equation*}
Then, in a small neighborhood of $\mathfrak{M}$ and around each optimal cluster, the function becomes quadratic. The set of minimizers is $\mathfrak{M} = \{ (x_{\sigma(1)}^{*},\dots, x_{\sigma(k)}^{*}) \,:\, \sigma \in \mathfrak{S}_{k}\}$, and $\mathfrak{S}_{k}$ is the set of permutations of $\{1,\dots, k\}$.

This property (of separated clusters) is one of the reasons why Lloyd’s algorithm \cite{lloyd1982least}  exhibits linear convergence when initialized near a correct solution (as shown in e.g. \cite{lu2016statistical}).

\textbullet\; \textit{\underline{Example 9}. The Exponential/Soft $k$-Means objective}

Consider the smooth ``soft'' $k$-means objective
\begin{equation*}
    f(x_{1}, \dots, x_{k})
    = \frac{1}{N}\sum_{i=1}^{N} \left[-\frac{1}{\beta} \log \left(\sum_{j=1}^{k} e^{-\beta \,|a_{i} - x_{j}|^2}\right)\right],
    \quad \beta>0.
\end{equation*}
For well-separated data and sufficiently large $\beta$, the global minimizers form a finite set  $\mathfrak{M} = \{(x_{\sigma(1)}^{*},\dots, x_{\sigma(k)}^{*}): \sigma \in \mathfrak{S}_{k}\}$,
where $(x_{1}^{*}, \dots, x_{k}^{*})$ are the true cluster centers. 
In a neighborhood of $\mathfrak{M}$, $f(\cdot)$ satisfies the two-sided quadratic growth condition
\begin{equation*}
    \frac{\mathfrak{c}_{1}}{2}\operatorname{dist}(x,\mathfrak{M})^{2}
    \leq f(x) - \underline{f}
    \leq \frac{\mathfrak{c}_{2}}{2}\operatorname{dist}(x,\mathfrak{M})^{2},
\end{equation*}
with constants $\mathfrak{c}_{1},\mathfrak{c}_{2}>0$ depending on $\beta$ and the cluster covariance structure.

\textbullet\; \textit{\underline{Example 10}. Gaussian mixtures and the EM objective}

Let $f(\theta)$ denote the negative log-likelihood of a Gaussian mixture model
\begin{equation*}
    f(\theta)
    = -\frac{1}{N}\sum_{i=1}^{N} \log\left(\sum_{j=1}^{k} w_{j} \, 
    \phi(a_{i};\mu_{j},\Sigma_{j})\right),
\end{equation*}
where $\phi(\,\cdot\,;\mu_{j},\Sigma_{j})$ is the Gaussian density and 
$\theta=(w_{1},\mu_{1},\Sigma_{1},\dots,w_{k},\mu_{k},\Sigma_{k})$. 
If the mixture components are well-separated and the true parameters
$\theta^{*}$ are identifiable up to permutation, then the set of global
minimizers is 
$\mathfrak{M}=\{\theta_{\pi}^{*}\,:\, \pi\in\mathfrak{S}_{k}\}$.
In a neighborhood of $\mathfrak{M}$, the function $f(\cdot)$ satisfies the
two-sided quadratic growth condition
\begin{equation*}
\frac{\mathfrak{c}_{1}}{2}\operatorname{dist}(\theta,\mathfrak{M})^{2}
\leq f(\theta)-\underline{f}
\leq \frac{\mathfrak{c}_{2}}{2}\operatorname{dist}(\theta,\mathfrak{M})^{2},
\end{equation*}
where the constants $\mathfrak{c}_{1},\mathfrak{c}_{2}>0$ depend on the
Fisher information matrix at $\theta^{*}$.

\section{Navigating non-convex landscapes: numerical experiments}\label{sec: numerics}

In this section, we present numerical experiments illustrating the theoretical findings of the paper in low-dimensional settings. We consider test problems with and without closed-form value functions, covering a range of non-convex landscapes including multiple global minimizers, continua of minimizers, and highly oscillatory objectives. The extension to high-dimensional settings is discussed in Section~\ref{sec: summary}.

\subsection{Semi-Lagrangian approximation of the HJB equation}

In general, the proposed approach for solving \eqref{OCP} relies on solving the HJB PDE
\begin{equation}\label{HJB with Ham}
	\begin{aligned}
		& \lambda\,u_{\lambda}(x) + \max\limits_{\alpha\in B_{M}}\left\{-\alpha\cdot Du_{\lambda}(x) - g(x,\alpha)\right\} = 0, \quad \text{for } x\in \mathbb{R}^{n}, \\
		& g(x,\alpha) = \frac{1}{2}|\alpha|^{2} + f(x).
	\end{aligned}
\end{equation}
In the low-dimensional tests presented here, this is efficiently approximated via grid-based schemes. We implement a first-order semi-Lagrangian scheme with accelerated policy iteration as in \cite{alla2015efficient}, whose main ingredients we recall for completeness, denoting $u_\lambda$ by $v$ to match the notation therein.

Given a computational domain $\Omega$ discretized by a uniform grid $G = \{x_i\}_{i=1}^{N_G}$ with mesh parameter $\Delta x$, the scheme is based on a time discretization of the DPP \eqref{DPP},
\begin{equation}\label{DPP discrete}
	v_{\Delta t}(x) = \min_{\alpha \in B_M} \left\{ e^{-\lambda \Delta t} \, v_{\Delta t}\left(x + \alpha\,\Delta t \right) + \Delta t \,g(x,\alpha) \right\},
\end{equation}
where $v_{\Delta t} \to v$ as $\Delta t \to 0$. On the grid, \eqref{DPP discrete} becomes the fixed-point problem
\begin{equation}\label{DPP discrete FP}
	V_i = \min_{\alpha\in B_M} \left\{ e^{-\lambda \Delta t} \,\mathds{I}[V]\!\left(x_i + \alpha\,\Delta t \right) + \Delta t\, g(x_i,\alpha) \right\}, \quad i=1,\ldots,N_G,
\end{equation}
where $V_i \approx v(x_i)$ and $\mathds{I}[V]:\mathbb{R}^n\to\mathbb{R}$ is a polynomial interpolation operator reconstructing $v$ at off-grid points from the nodal values (we refer to \cite[Appendix A]{bardi1997optimal} for details). This is solved by the Accelerated Policy Iteration algorithm in \cite{alla2015efficient}, which uses Value Iteration on a coarse mesh to initialize Policy Iteration on a fine mesh; all three algorithms are recalled in Appendix~\ref{app:algo}. The minimization in \eqref{DPP discrete FP} is performed by enumeration over a radial-angular discretization of the ball $|\alpha| \leq M$, and the semi-Lagrangian time step is set to $\Delta t = \Delta x / M$. The domain $\Omega$ is chosen sufficiently large to contain all global minimizers, and closed with a large Dirichlet boundary value for $v$, which is equivalent to imposing a state constraint keeping trajectories inside $\Omega$.

In all tests below, we set $\Delta x = 0.025$ and $\lambda = 10^{-3}$. The control bound $M$ and domain $\Omega$ are specified for each test case.

\subsection{Numerical tests}

\subsubsection{Validation with analytical solutions}

We first consider the quadratic distance-to-set objective
\begin{equation*}
	f(x) = \frac{\mathfrak{c}}{2}\operatorname{dist}(x,\mathscr{O})^2,
\end{equation*}
for which the value function is available in closed form by Lemma~\ref{lem: ricc},
\begin{equation*}
	u_\lambda(x) = C(\lambda)\operatorname{dist}(x,\mathscr{O})^2, \quad C(\lambda) = \frac{-\lambda+\sqrt{\lambda^2+4\mathfrak{c}}}{4}.
\end{equation*}
We test four choices of target set $\mathscr{O}$: (A) the interval $[-1,1]$, (B) the two-point set $\{-1,1\}$, (C) the unit circle $\{x\in\mathbb{R}^2:\|x\|=1\}$, and (D) the diagonal $\{(x_1,x_2)\in\mathbb{R}^2: x_1=x_2\}$ (Example~2 in Section~\ref{sec: ex}). Figure~\ref{fig:closedform} shows trajectories from multiple starting points together with the decay of $\operatorname{dist}(y(t),\mathscr{O})^2$. All trajectories converge to $\mathscr{O}$, and the semilog plots confirm exponential decay in time. In case~(B), the value function is not differentiable at the origin, so the feedback $Du_\lambda$ must be interpreted in a generalized subgradient sense, yielding a set-valued control; trajectories initialized from $0$ may converge to either minimizer. A similar non-uniqueness occurs at the origin in case~(C).

\begin{figure}[h]
	\centering
	\begin{subfigure}[t]{0.49\textwidth}
		\centering
		\includegraphics[width=\linewidth]{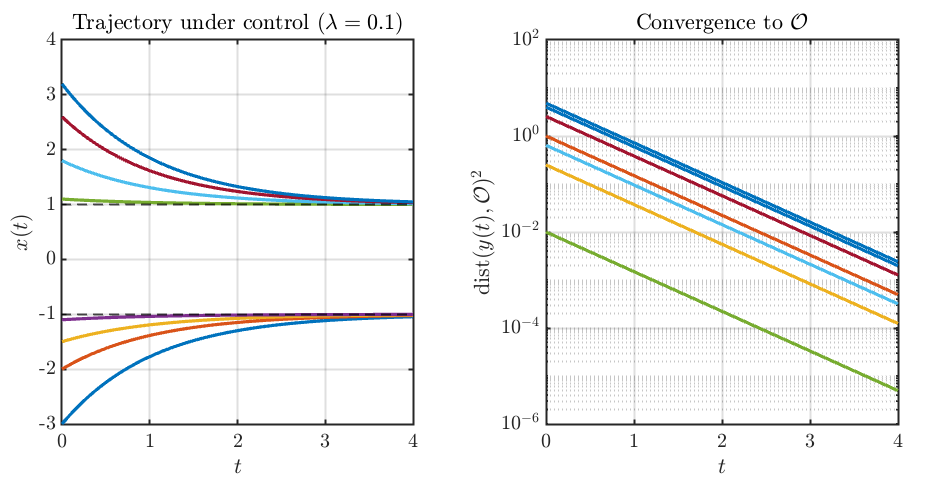}
		\caption{$\mathscr{O}=[-1,1]$}
		\label{fig:closedform:interval}
	\end{subfigure}\hfill
	\begin{subfigure}[t]{0.49\textwidth}
		\centering
		\includegraphics[width=\linewidth]{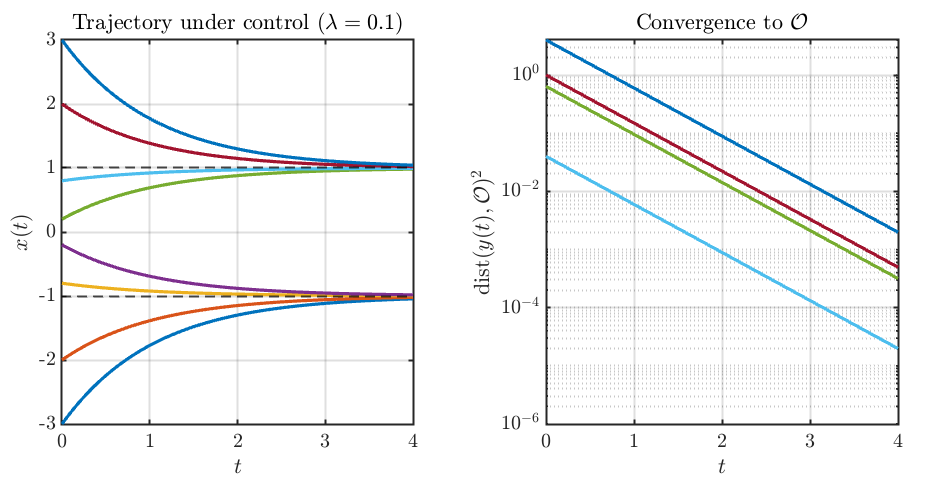}
		\caption{$\mathscr{O}=\{-1,1\}$}
		\label{fig:closedform:twopoint}
	\end{subfigure}
	
	\vspace{0.6em}
	
	\begin{subfigure}[t]{0.49\textwidth}
		\centering
		\includegraphics[width=\linewidth]{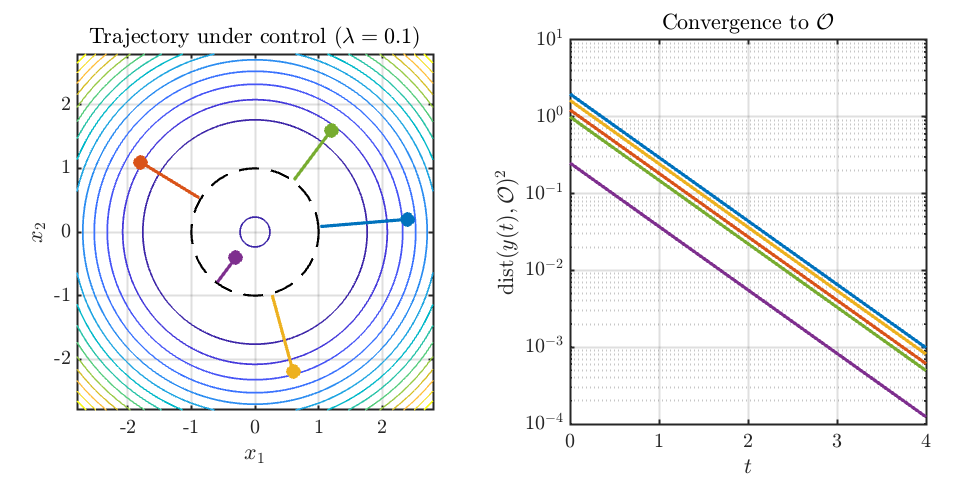}
		\caption{$\mathscr{O}=\{x\in\mathbb{R}^2:\|x\|=1\}$}
		\label{fig:closedform:ring}
	\end{subfigure}\hfill
	\begin{subfigure}[t]{0.49\textwidth}
		\includegraphics[width=1.2\linewidth]{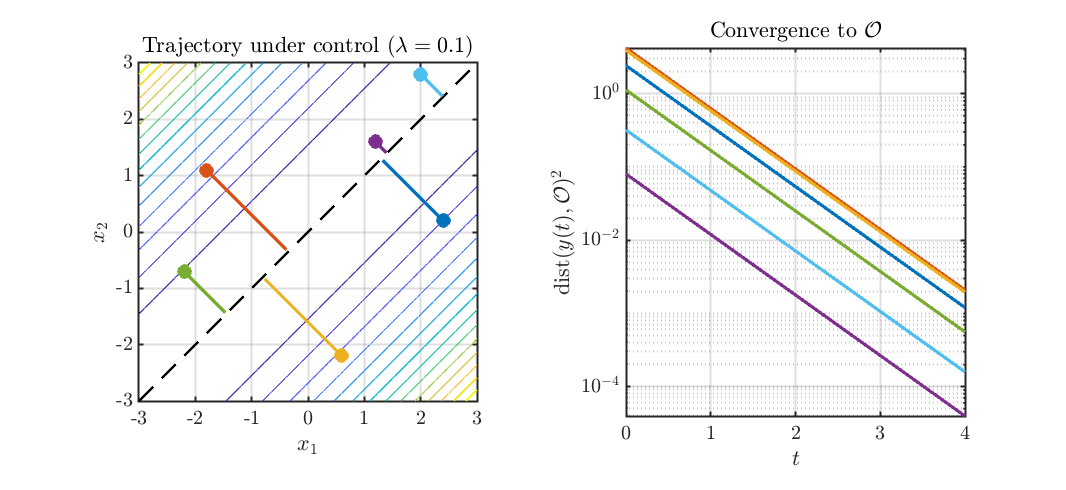}
		\caption{$\mathscr{O}=\{(x_1,x_2)\in\mathbb{R}^2:\,x_1=x_2\}$}
		\label{fig:closedform:diag}
	\end{subfigure}
	\caption{Closed-form quadratic distance-to-set tests (cases A--D): controlled trajectories (left) and exponential decay of $\operatorname{dist}(y(t),\mathscr{O})^2$ (right) for different target sets $\mathscr{O}$.}
	\label{fig:closedform}
\end{figure}

\subsubsection{Himmelblau's function}

We consider global minimization of Himmelblau's function
\begin{equation*}
	f_1(x_1,x_2) = \left(x_1^2+x_2-11\right)^2+\left(x_1+x_2^2-7\right)^2, \quad (x_1,x_2)\in[-4,4]^2,
\end{equation*}
which has four global minimizers of identical value over $\Omega=[-4,4]^2$. The value function is computed with $M=5$. Figure~\ref{fig:Himmelblau}(A) shows closed-loop trajectories from a range of initial conditions, each converging to one of the global minimizers (yellow stars). The red diamond marks a local maximizer; even when initialized at this stationary point, the controlled dynamics drive the state toward a global minimizer, whereas standard gradient flow would remain trapped. Panels~(B)--(C) display the objective landscape together with the computed value function and feedback field. Panel~(D) shows the decay of $\operatorname{dist}(y(t),\mathfrak{M})^2$, consistent with the theoretical predictions: trajectories first reach a neighborhood of $\mathfrak{M}$ inside which the quadratic growth condition~\ref{QG} holds, and subsequently exhibit exponential convergence toward one of the global minimizers.

\begin{figure}[h]
	\centering
	\begin{subfigure}[t]{0.5\textwidth}
		\centering
		\includegraphics[width=\linewidth]{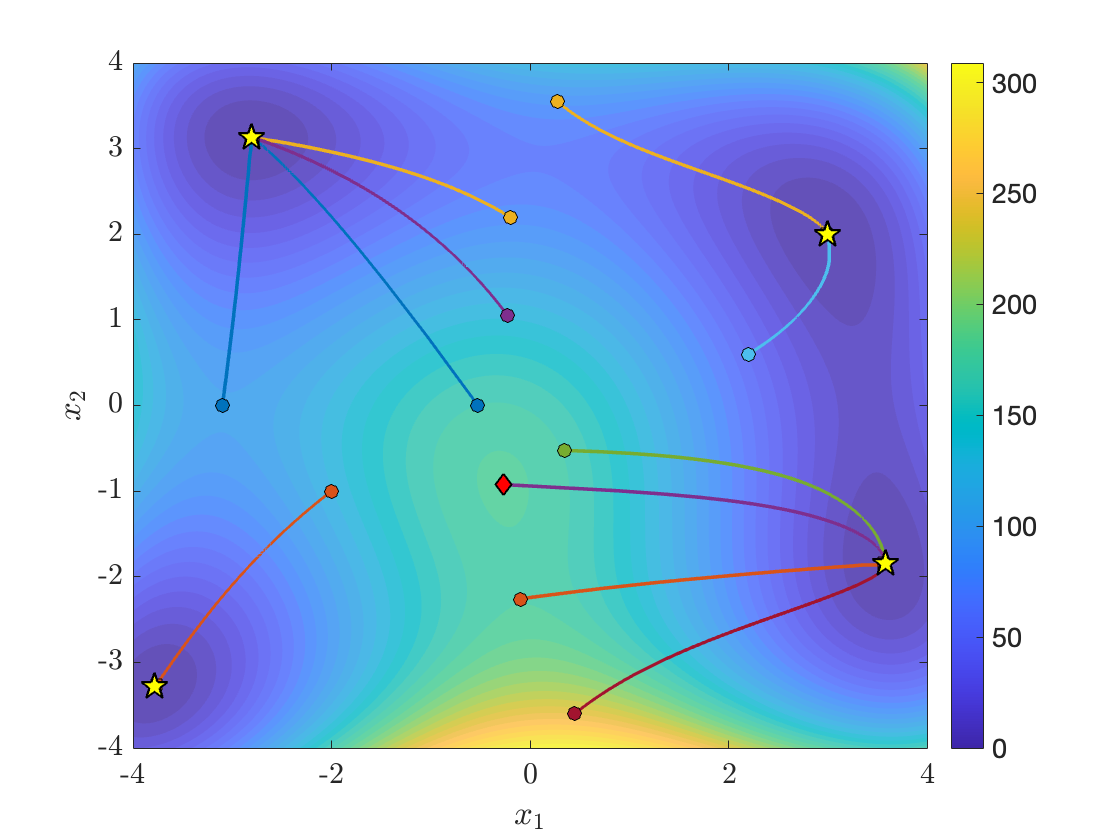}
		\caption{Trajectories under the optimal feedback law. Yellow stars mark the four global minimizers and the red diamond marks a local maximizer.}
	\end{subfigure}\hfill
	\begin{subfigure}[t]{0.5\textwidth}
		\centering
		\includegraphics[width=\linewidth]{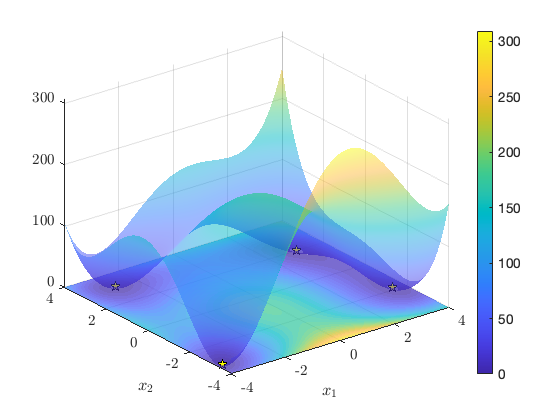}
		\caption{Himmelblau's objective $f_1$ over $[-4,4]^2$.}
	\end{subfigure}\hfill
	\begin{subfigure}[t]{0.5\textwidth}
		\centering
		\includegraphics[width=\linewidth]{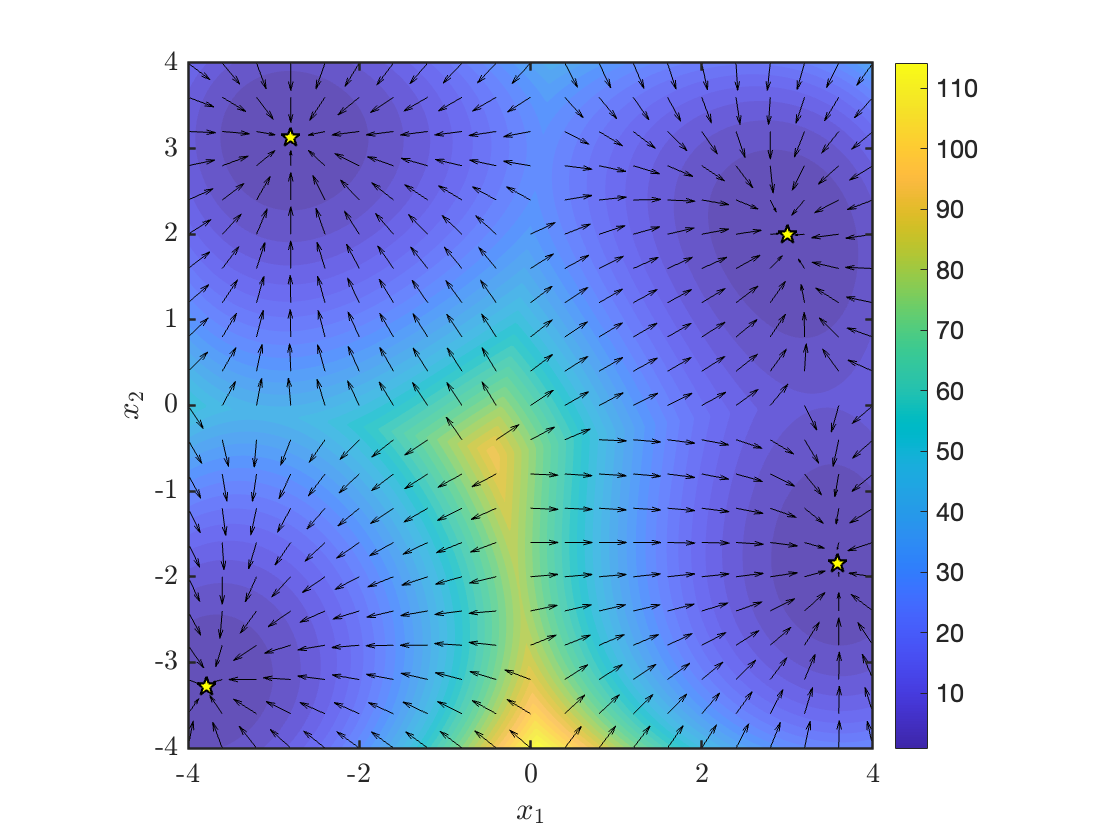}
		\caption{Computed value function and associated optimal feedback field.}
		\label{fig:himmelblau-control}
	\end{subfigure}
	\begin{subfigure}[t]{0.48\textwidth}
		\centering
		\includegraphics[width=\linewidth]{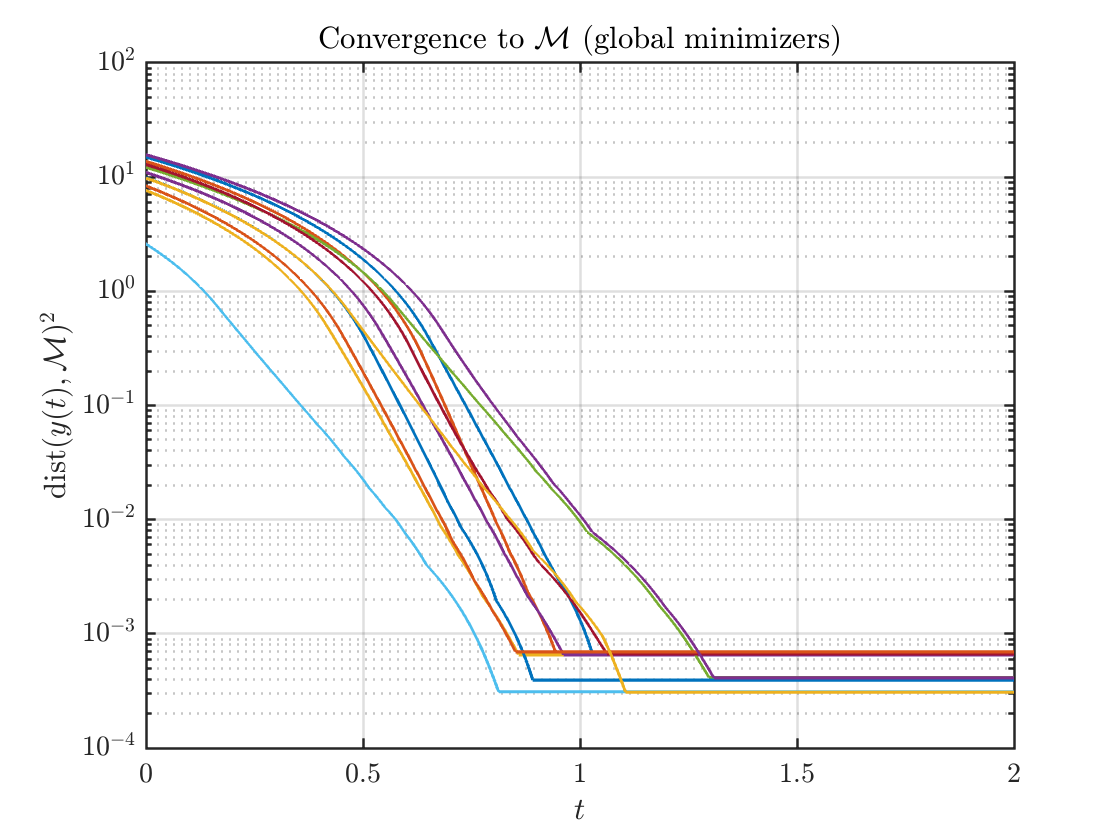}
		\caption{Decay of $\operatorname{dist}(y(t),\mathfrak{M})^2$ along the trajectories in (A).}
	\end{subfigure}
	\caption{Himmelblau test problem over $[-4,4]^2$.}
	\label{fig:Himmelblau}
\end{figure}

\subsubsection{Neural network losses with symmetries}

We test the controlled dynamics on the neural network loss introduced in Example~6 (Section~\ref{sec: ex}),
\begin{equation*}
	f_2(x_1,x_2) = \frac{1}{2}(x_1 x_2-1)^2 \quad \text{on } [-4,4]^2,
\end{equation*}
whose set of global minimizers forms the hyperbola $\mathfrak{M} = \{(x_1,x_2): x_1 x_2 = 1\}$. The value function is computed with $\Omega=[-4,4]^2$ and $M=1$. As shown in Figure~\ref{fig:ex6}, despite the flat valley along $\mathfrak{M}$ and the stationary point at the origin, the controlled dynamics drive all trajectories toward the minimizer set.

\begin{figure}[h]
	\centering
	\begin{subfigure}[t]{0.5\textwidth}
		\centering
		\includegraphics[width=\linewidth]{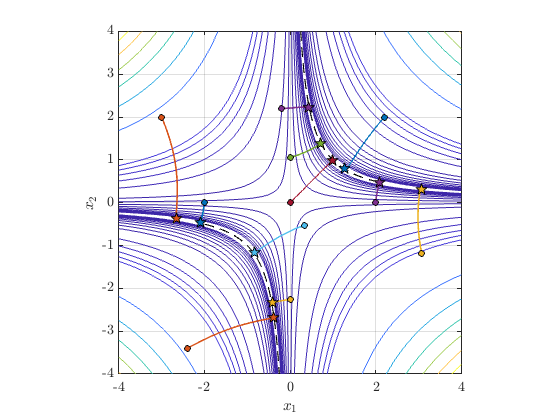}
		\caption{Trajectories under the optimal feedback law, initialized from multiple starting points.}
	\end{subfigure}\hfill
	\begin{subfigure}[t]{0.5\textwidth}
		\centering
		\includegraphics[width=\linewidth]{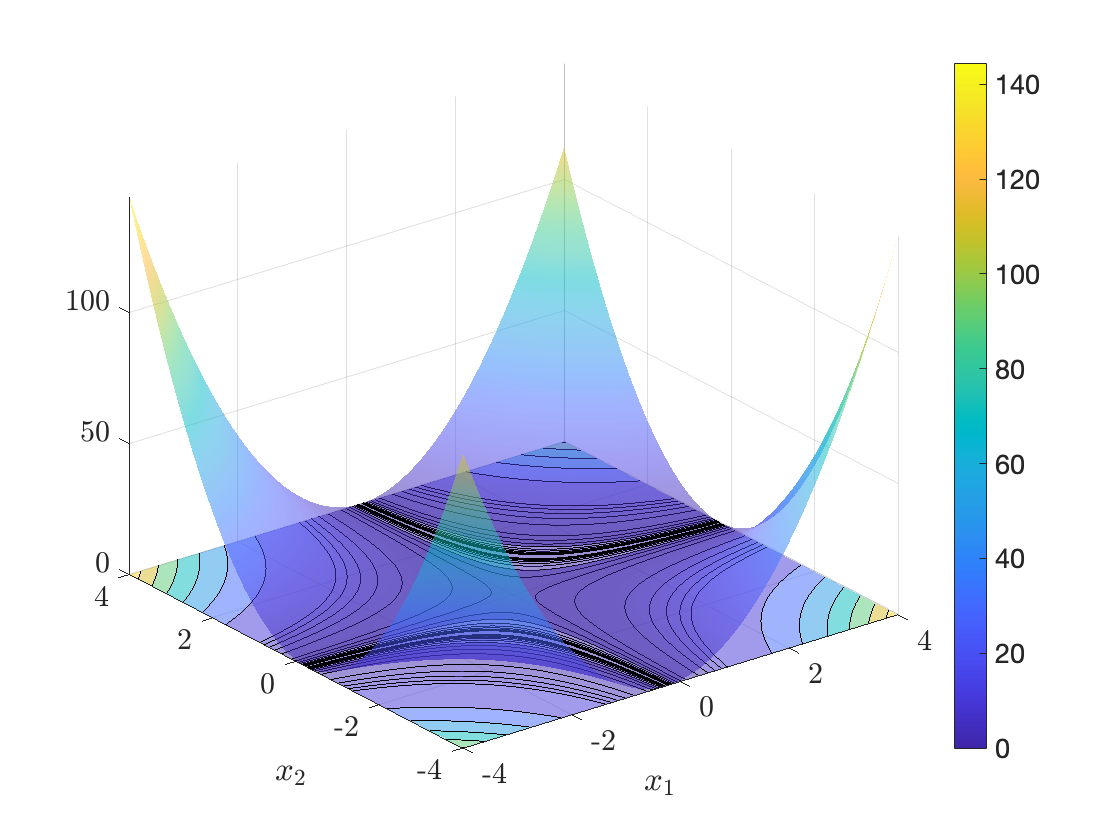}
		\caption{Objective $f_2$ over $[-4,4]^2$.}
	\end{subfigure}\hfill
	\begin{subfigure}[t]{0.5\textwidth}
		\centering
		\includegraphics[width=\linewidth]{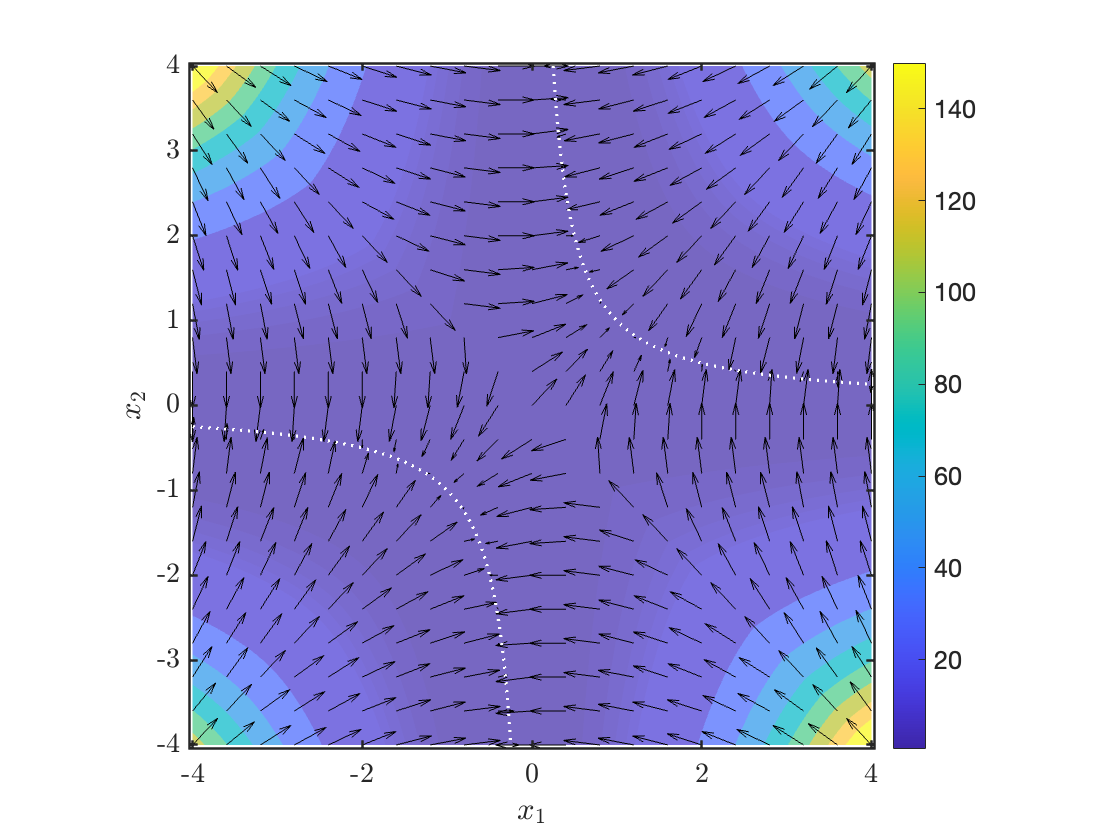}
		\caption{Optimal feedback field; the white dotted curve indicates $\mathfrak{M}$.}
	\end{subfigure}\hfill
	\begin{subfigure}[t]{0.48\textwidth}
		\centering
		\includegraphics[width=\linewidth]{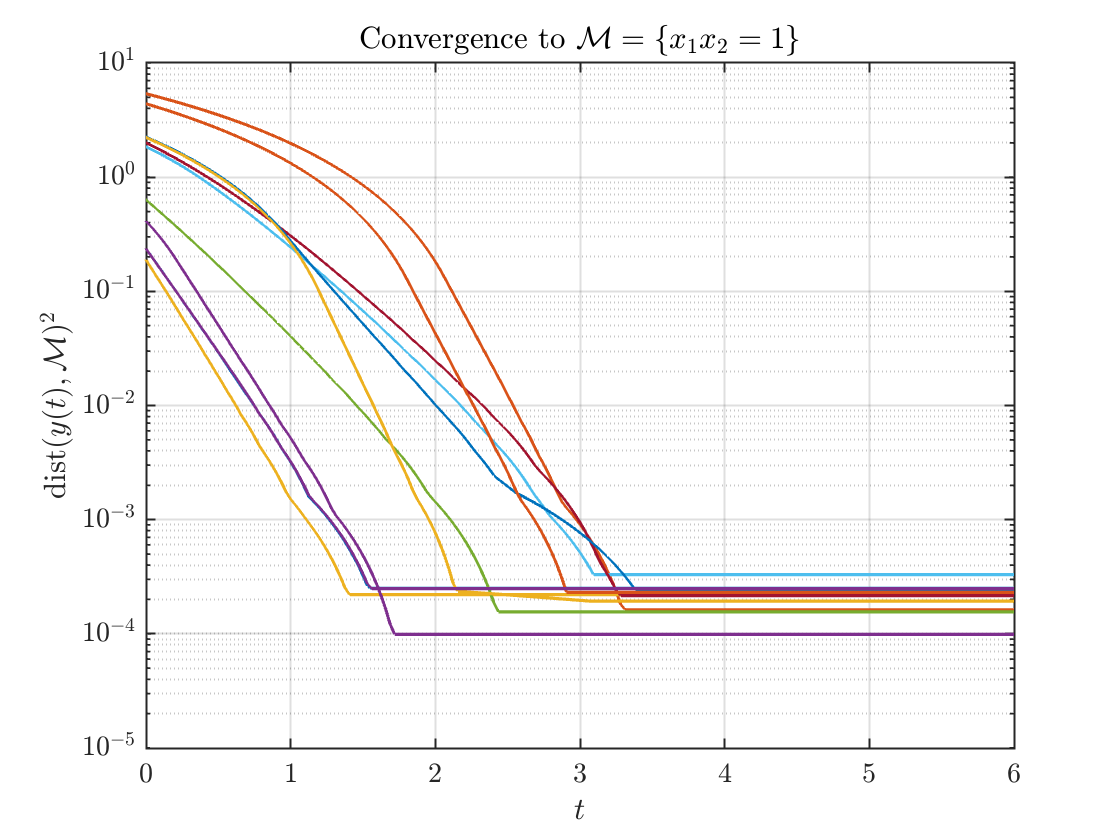}
		\caption{Decay of $\operatorname{dist}(y(t),\mathfrak{M})^2$.}
	\end{subfigure}
	\caption{Non-convex loss $f_2(x_1,x_2) = \frac{1}{2}(x_1 x_2-1)^2$ with a continuum of minimizers $\mathfrak{M}=\{x_1 x_2=1\}$.}
	\label{fig:ex6}
\end{figure}

\subsubsection{Rastrigin function}

The Rastrigin function is a standard benchmark in global optimization, with a unique global minimizer at the origin and a large number of local stationary points. We consider
\begin{equation*}
	f_3(x_1,x_2) = 20+\sum_{i=1}^2\left(x_i^2-10\cos(2\pi x_i)\right) \quad \text{on } [-2,2]^2.
\end{equation*}
The value function is computed with $\Omega=[-2,2]^2$ and $M=10$. Trajectories are initialized both from generic points and from stationary points of $f_3$, including local minimizers and maximizers. As shown in Figure~\ref{fig:ras}, all controlled trajectories escape local optima and converge to the global minimizer, with rapid exponential decay of $\operatorname{dist}(y(t),\mathfrak{M})^2$.

\begin{figure}[h]
	\centering
	\begin{subfigure}[t]{0.45\textwidth}
		\centering
		\includegraphics[width=\linewidth]{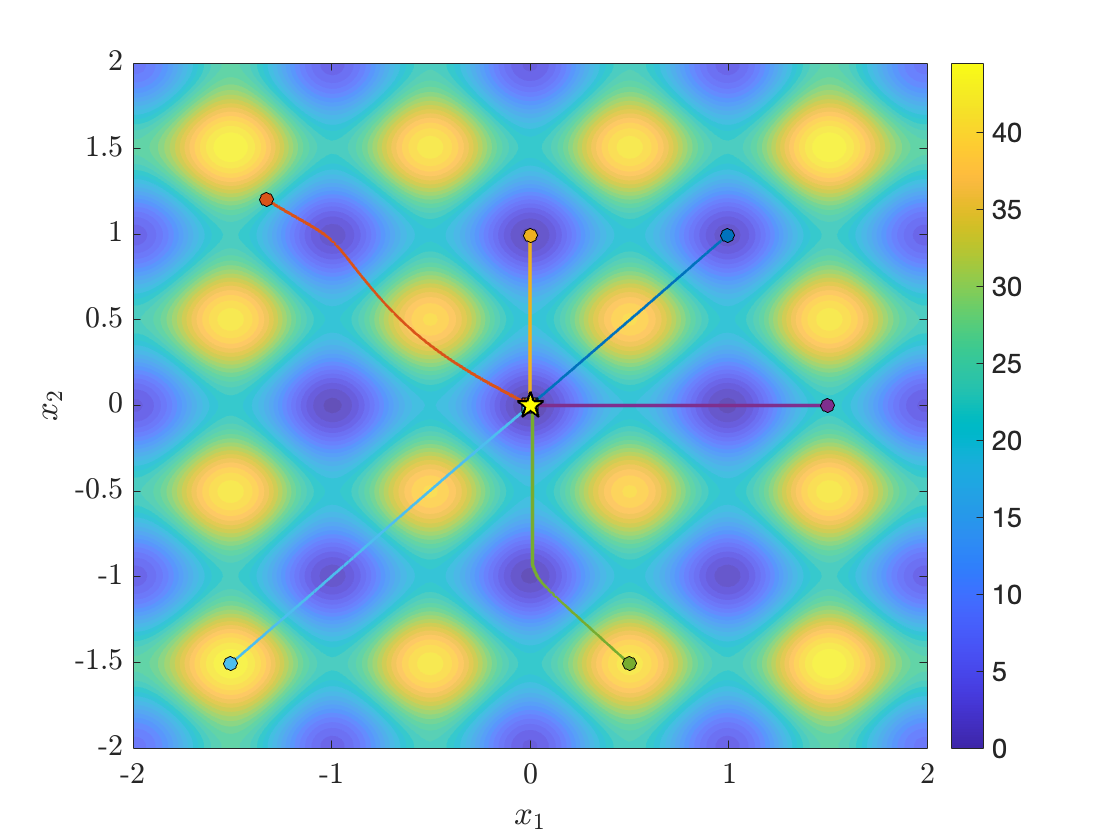}
		\caption{Trajectories under the optimal feedback law, initialized from multiple starting points.}
	\end{subfigure}\hfill
	\begin{subfigure}[t]{0.48\textwidth}
		\centering
		\includegraphics[width=\linewidth]{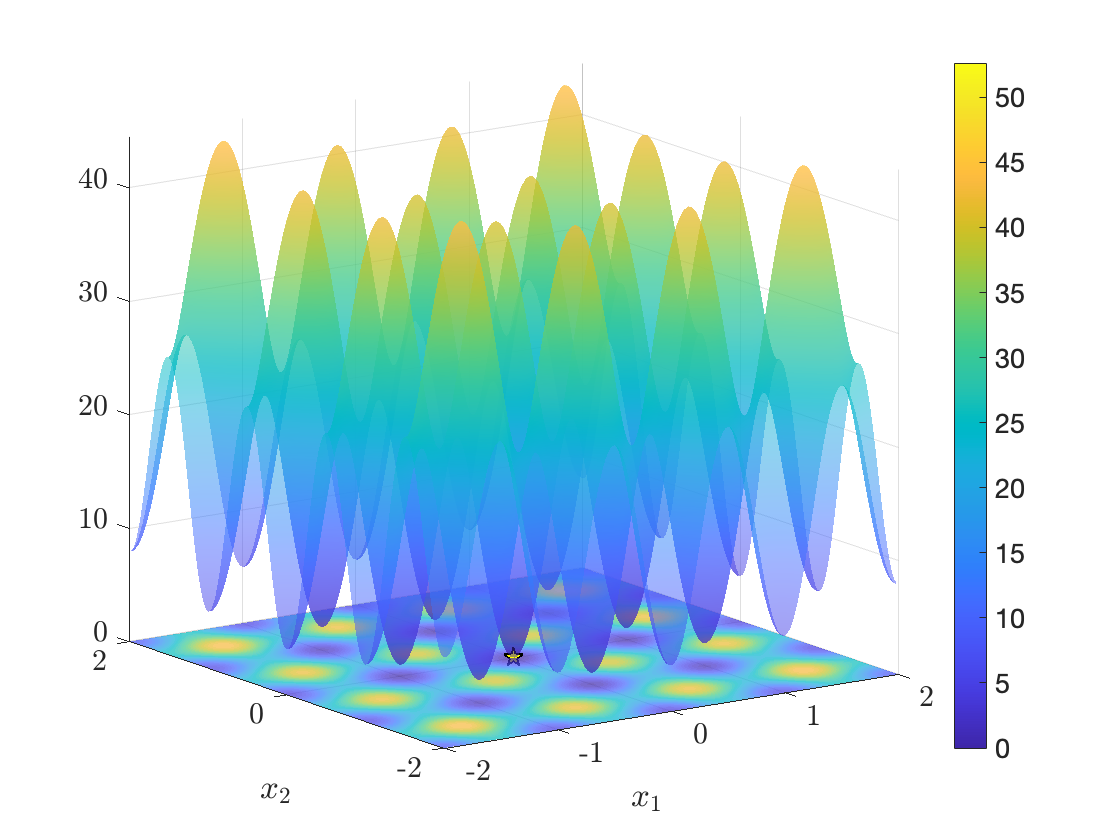}
		\caption{Rastrigin objective $f_3$ over $[-2,2]^2$.}
	\end{subfigure}\hfill
	\begin{subfigure}[t]{0.5\textwidth}
		\centering
		\includegraphics[width=\linewidth]{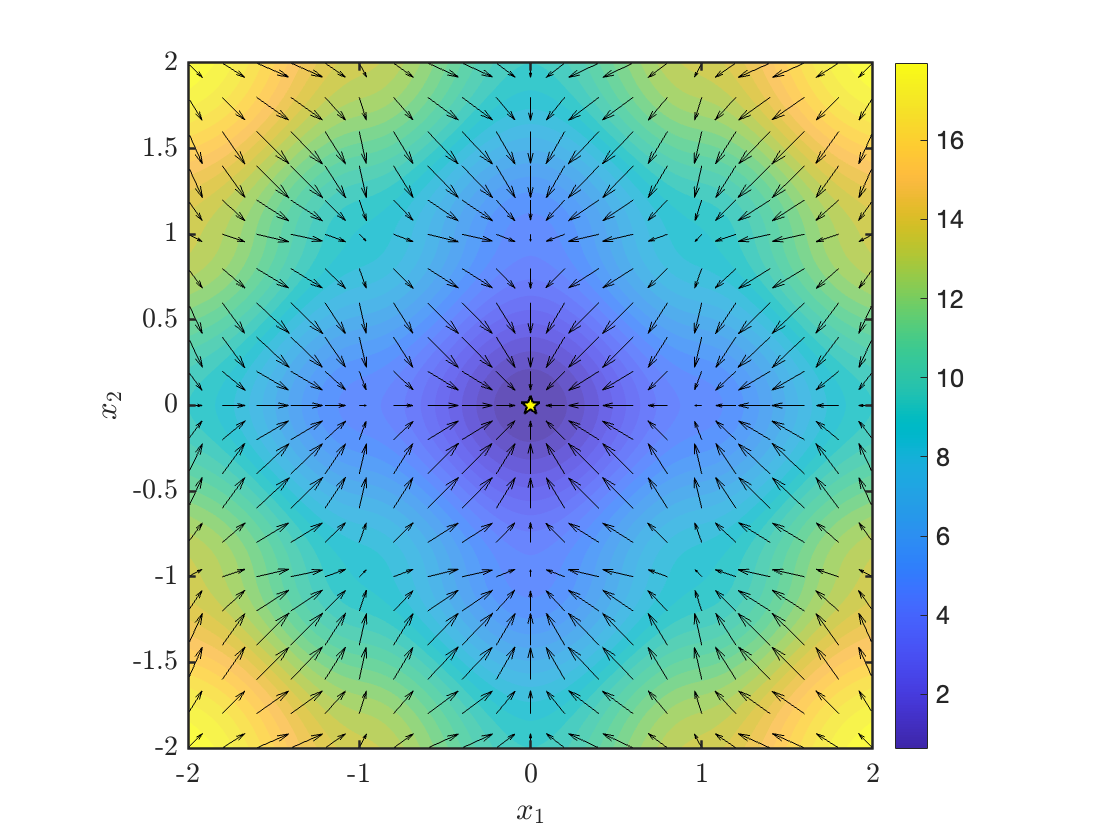}
		\caption{Optimal feedback field.}
	\end{subfigure}\hfill
	\begin{subfigure}[t]{0.5\textwidth}
		\centering
		\includegraphics[width=\linewidth]{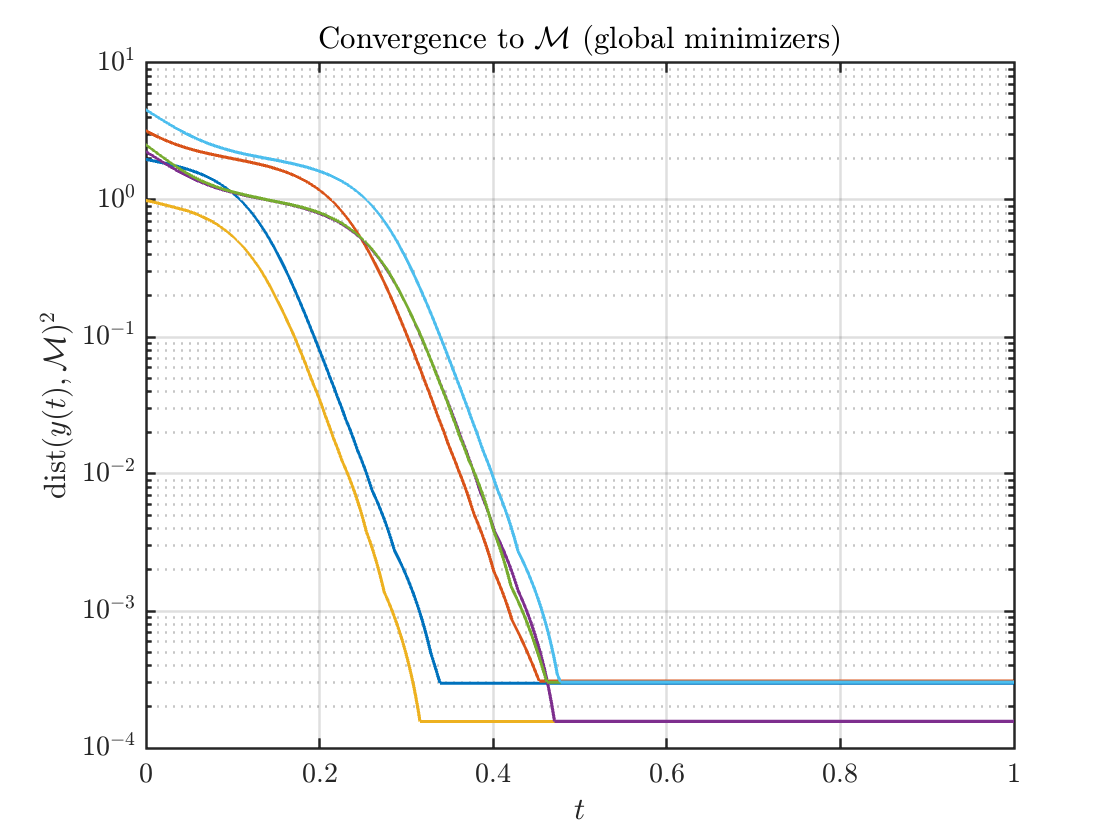}
		\caption{Decay of $\operatorname{dist}(y(t),\mathfrak{M})^2$.}
	\end{subfigure}
	\caption{Rastrigin test problem: the feedback escapes local stationary points and drives convergence to the global minimizer.}
	\label{fig:ras}
\end{figure}
\section{Conclusions}\label{sec: summary}

We have developed a rigorous framework for global non-convex optimization through optimal control theory, establishing explicit exponential convergence rates under minimal structural assumptions. By reformulating the minimization problem as a discounted infinite-horizon control problem, we proved that trajectories following the gradient flow of the value function converge to global minimizers with quantifiable rates, requiring neither convexity, nor differentiability, nor \L{}ojasiewicz-type conditions on the objective function. 

Several key mechanisms underlie these convergence guarantees. The value function serves as a Lyapunov function for the controlled dynamics, decaying exponentially to the global minimum along optimal and quasi-optimal trajectories. The pathwise convergence analysis further reveals Turnpike phenomena and fundamental connections to controllability theory. Crucially, structural assumptions are imposed on the value function rather than directly on the objective, and we showed that mild local properties of the objective (locally linear or quadratic growth near minimizers) automatically induce the desired features in the value function.

This approach offers a deterministic and principled alternative to stochastic or heuristic methods commonly used in non-convex optimization. By lifting the problem to an infinite-dimensional setting through the optimal control framework, the method relies solely on deterministic principles rather than randomness or sampling. This yields both rigorous theoretical guarantees and a clear dynamical interpretation of the optimization process.

Beyond the specific technical contributions, this work illustrates the power of cross-disciplinary approaches in optimization. The control-theoretic perspective transforms an intractable non-convex minimization problem into a stabilization problem with exploitable structure, revealing convergence mechanisms invisible from a purely optimization-theoretic viewpoint  and  offering a fresh perspective on longstanding challenges in non-convex optimization.
\subsection*{Computational challenges and future directions}
The numerical experiments in Section~\ref{sec: numerics} demonstrate the practical feasibility of the approach on challenging non-convex landscapes. However, the primary bottleneck is solving the HJB equation in high dimensions, where standard grid-based methods suffer from the curse of dimensionality. Substantial progress has been made in recent years on high-dimensional HJB solvers, including tensor decomposition methods and polynomial approximation schemes \cite{dolgov2021tensor, Kalise_2018}, which could extend the proposed framework beyond low-dimensional settings without sacrificing its deterministic and analytical character. We note, however, that approaches based on neural network approximation of the value function would be circular in the context of non-convex optimization, as training such networks is itself a non-convex problem. In our view, the most natural application domain for the present approach is highly non-convex problems of moderate dimension, such as structural optimization, where the non-convexity of the landscape is the central challenge and the dimension remains tractable for direct PDE solvers.

A complementary direction, explored in recent work by some of the authors \cite{huang2024fast}, combines coarse approximations of the optimal control with consensus-based optimization (CBO), showing that even rough value function information can significantly improve standard CBO performance. Our theoretical results on quasi-optimality provide rigorous justification for such strategies, since global convergence is preserved even when trajectories are only approximately optimal.

\subsection*{Minimal structural assumptions and extensions}
A central element of our framework is Assumption \ref{main assumption}, which provides the key structural requirement enabling the variational convergence results of Section \ref{sec: conv rate}. This condition ensures exponential decay of the value function along optimal and quasi-optimal trajectories toward the global minimum. We have shown that when the objective 
$f$ exhibits a local-quadratic growth (Assumption \ref{QG}), or a local-linear growth (Assumption \ref{LG}), or it satisfies a controllability condition (Assumption \ref{Cont}),  then Assumption \ref{main assumption} is automatically satisfied. However, we emphasize that these sufficient conditions need not be necessary.  As demonstrated in Section~\ref{sec: PL},  Assumption \ref{main assumption} is a PL-type condition imposed on the value function rather than directly on the objective. When the value function $u_\lambda(\cdot)$ is a classical solution to the HJB equation~\eqref{eq: HJB}, Assumption~\ref{main assumption} is equivalent to the standard PL inequality.  By serving as a surrogate for the objective, the value function may induce the required structure even when $f$ lacks these classical properties, a key strength of the control-theoretic approach. Therefore, identifying \textit{minimal} structural conditions on $f$ that guarantee Assumption \ref{main assumption} remains an important open question. Similarly, it is unclear whether the local quadratic growth condition (Assumption \ref{QG}) required for pathwise convergence can be relaxed while preserving exponential convergence to $\mathfrak{M}$. Resolving these questions would significantly extend the applicability of our framework to broader classes of non-convex optimization problems.

\subsection*{Connections to modern optimization algorithms.} The relationship between our control-based trajectories and classical optimization methods deserves deeper exploration. The damping terms in momentum-based methods (e.g., Nesterov acceleration and  heavy-ball methods) can be interpreted as feedback control of the velocity, suggesting natural connections to our framework. Could existing practical algorithms be reinterpreted through the control-theoretic lens? While momentum methods achieve accelerated rates in convex settings, their behavior in non-convex landscapes remains poorly understood. The control perspective may offer theoretical clarity by revealing the stabilization mechanisms at work and potentially explaining when such methods can escape local minima or when they become trapped.

\appendix

\section{Elements from non-smooth analysis}\label{app: nonsmooth}

We first review several standard definitions and results from non-smooth analysis, after which we present the key result in Lemma \ref{lem:geom}. These definitions will also be used in Appendix \ref{appendix: control} and in Appendix \ref{app: ricc}.

\begin{definition}(See \cite[II.1, page 29]{bardi1997optimal})
Let $u:\Omega\to \mathbb{R}$ be continuous on $\Omega\subseteq \mathbb{R}^{n}$ an open set. We call super-differential of $u(\cdot)$ at the point $x$ the set
\begin{equation*}
    D^{+}u(x) := \left\{\; p\in\mathbb{R}^{n}\,:\, \limsup\limits_{y\to x, y\in\Omega} \, \frac{u(y) - u(x) - p\cdot(y-x)}{|y-x|} \leq 0 \;\right\}.
\end{equation*}
We call sub-differential (or semi-differential) of $u(\cdot)$ at the point $x$ the set
\begin{equation*}
    D^{-}u(x) := \left\{\; p\in\mathbb{R}^{n}\,:\, \liminf\limits_{y\to x, y\in\Omega} \, \frac{u(y) - u(x) - p\cdot(y-x)}{|y-x|} \geq 0 \;\right\}.
\end{equation*}
\end{definition}

\begin{definition}[Dini derivatives] (See \cite[Definition III.2.36, page 125]{bardi1997optimal})\label{def app Dini}
Let $u:\Omega\to \mathbb{R}$, $\Omega\subseteq \mathbb{R}^{n}$ an open set. The lower Dini derivative, or upper contingent derivative, of $u(\cdot)$ at the point $x\in \Omega$ in the direction $q\in\mathbb{R}^{n}$ is
\begin{equation*}
\begin{aligned}
    \partial^{-}u(x;q) & := \liminf_{\substack{t \rightarrow 0^{+}\\p \rightarrow q}} \; \frac{u(x+t\,p) - u(x)}{t}\\
    & = \sup\limits_{\delta>0}\, \inf\left\{\frac{u(x+tq + tz) - u(x)}{t}\,:\, 0<t\leq \delta\,\;|z|\leq \delta\right\}\; \geq -\infty.
\end{aligned}
\end{equation*}
Similarly the upper Dini derivative, or lower contingent  derivative of $u(\cdot)$ at $x$ in the direction $q$ is
\begin{equation*}
\begin{aligned}
    \partial^{+}u(x;q) & := \limsup_{\substack{t \rightarrow 0^{+}\\p \rightarrow q}} \; \frac{u(x+t\,p) - u(x)}{t}\; \leq +\infty.
\end{aligned}
\end{equation*}
\end{definition}

One notes in particular
\begin{equation*}
\begin{aligned}
    \partial^{-}u(x;q) \leq \liminf\limits_{t\to 0^{+}} \frac{u(x+t\,p) - u(x)}{t} \quad \text{ and } \quad 
    \partial^{+}u(x;q) \geq \limsup\limits_{t\to 0^{+}} \frac{u(x +t\,p) - u(x)}{t}
\end{aligned}
\end{equation*}
with equalities if $u(\cdot)$ is Lipschitz. 

\begin{lemma}(See \cite[Lemma III.2.37, page 126]{bardi1997optimal})\label{lem app D partial}
Let $u:\Omega\to \mathbb{R}$ and $\Omega$ an open set. Then
\begin{equation*}
\begin{aligned}
    D^{-}u(x) & = \big\{\, p\,: \; p\cdot q \leq \partial^{-}u(x;q)\quad \forall\,q\in\mathbb{R}^{n}\big\},\\
    D^{+}u(x) & = \big\{\, p\,: \; p\cdot q \geq \partial^{+}u(x;q)\quad \forall\,q\in\mathbb{R}^{n}\big\}.
\end{aligned}
\end{equation*}
\end{lemma}

\begin{definition}\label{def: general direc deriv}[Generalized directional derivative](See \cite[Page 70]{clarke1998nonsmooth})
Let $\phi(\cdot)$ be Lipschitz near a given point $x$. The generalized directional derivative of $\phi(\cdot)$ at $x$ in the direction $v$, denoted by $\phi^{\circ}(x ; v)$, is defined as
\begin{equation*}
    \phi^{\circ}(x ; v):=\limsup _{\substack{y \rightarrow x \\ t \rightarrow 0^{+}}} \; \frac{\phi(y+t \,v) - \phi(y)}{t}.
\end{equation*} 
\end{definition}

\begin{definition}\label{def: general grad}[Generalized gradient](See \cite[Page 27]{clarke1990optimization})
The generalized gradient of $\phi(\cdot)$ at $x$, denoted by $\partial^{C} \phi(x)$, is given by
\begin{equation*}
    \partial^{C} \phi(x) := \left\{\zeta\in\mathbb{R}^n \;:\; \phi^{\circ}(x ; v) \geqslant\langle\zeta, v\rangle \quad \forall\, v \in \mathbb{R}^n\right\} .
\end{equation*}
\end{definition}

\begin{proposition}\label{prop: clarke}(See \cite[Proposition 1.5, page 73]{clarke1998nonsmooth})
Let $\phi(\cdot)$ be Lipschitz near a given point $x\in \mathbb{R}^{n}$. The following hold
\begin{enumerate}[label = (\alph*)]
    \item\label{clarke 1} $\partial^{C} \phi(x)$ is a non-empty, convex and compact subset of $\mathbb{R}^{n}$.
    \item\label{clarke 2} For every $v \in \mathbb{R}^{n}$, one has in particular 
    \begin{equation*}
        \phi^{\circ}(x ; v)=\max \{\;\langle\zeta, v\rangle \;:\; \zeta \in \partial^{C} \phi(x)\; \}.
    \end{equation*}
    \item\label{clarke 3} $\zeta \in \partial^{C}\phi(x)$ \; if and only if \;$ \langle \zeta,v \rangle \leq \phi^{\circ}(x;v) \quad \forall\,v\in \mathbb{R}^{n}$.
    \item\label{clarke 4} $\partial^{C}\phi(x)$ is upper-semicontinuous at $x\in \mathbb{R}^{n}$. In particular, thanks to compactness in \ref{clarke 1}, its graph is closed. 
\end{enumerate}
\end{proposition}

With the preceding results and definitions, we are now ready to state and prove a key result that will enable the proof of Lemma \ref{lem: ricc} in Appendix \ref{app: ricc}. We recall that Lemma \ref{lem: ricc} is used at the end of the proof of Proposition \ref{prop: sandwich u}.

\begin{lemma}[Subdifferential of squared distance]\label{lem:geom}
Let $\mathscr{O}\subset\mathbb{R}^{n}$ be a non-empty and compact set. Define $\; \phi(x):=\frac{1}{2}\,\operatorname{dist}(x,\mathscr{O})^2.\;$ Then
\begin{enumerate}[label = (\roman*)]
\item\label{first in lemma} $\phi(\cdot)$ is locally Lipschitz.
\item\label{second in lemma} For every $x\in\mathbb{R}^{n}$ and every $q\in\partial^C\phi(x)$, one has 
 $|q|\leq \operatorname{dist}(x,\mathscr{O})$.
\item\label{third in lemma} If $y(\cdot):[0,T]\to\mathbb{R}^{n}$ is absolutely continuous, then $t\mapsto\phi(y(t))$ is absolutely continuous and there exists   $q(t)\in\partial^C\phi(y(t))$ such that
\begin{equation*}
    \frac{\dd}{\dd t}\phi(y(t))=\langle q(t),\dot y(t)\rangle\quad\text{for a.a. } t\in[0,T].
\end{equation*}
\end{enumerate}
\end{lemma}

\begin{proof}
First we prove \ref{first in lemma}.\\ 
Using the triangular inequality, the distance map $x\mapsto \operatorname{dist}(x,\mathscr{O})$ is $1$-Lipschitz. Hence for all $x, y\in K$ a bounded subset of $\mathbb{R}^{n}$, one has
\begin{equation*}
    |\phi(x)-\phi(y)|\leq \frac{1}{2}\big(\operatorname{dist}(x, \mathscr{O})+\operatorname{dist}(y, \mathscr{O})\big)|\operatorname{dist}(x, \mathscr{O})-\operatorname{dist}(y, \mathscr{O})| \leq 2L_K |x-y| .
\end{equation*}
with $L_K:=\sup _{x \in K} \operatorname{dist}(x,\mathscr{O})<\infty$.

Next we prove \ref{second in lemma}. \\
For $y\in\mathbb{R}^{n}$, since $\mathscr{O}\subset\mathbb{R}^n$ is non-empty and compact,  the projection set 
\begin{equation*}
    \mathbf{P}_{\mathscr{O}}(y):=\arg \min _{z \in \mathscr{O}}|y-z| \, = \{ z\in \mathscr{O}\,:\, |y-z| = \min\limits_{\omega\in \mathscr{O}}|y-\omega|=\operatorname{dist}(y,\mathscr{O}) \}
\end{equation*}
is non-empty and compact. Fix $z_y\in \mathbf{P}_{\mathscr{O}}(y)$.
Then for any $t>0$ and $h\in\mathbb{R}^{n}$, 
\begin{equation*}
    \phi(y+th) \leq \frac12|y+th-z_y|^2
    = \frac12|y-z_y|^2 + t\langle y-z_y,h\rangle + \frac{t^2}{2}|h|^2,
\end{equation*}
hence
\begin{equation*}
    \frac{\phi(y+th)-\phi(y)}{t}\leq \langle y-z_y,h\rangle + \frac{t}{2}|h|^2 \; \leq \sup_{z\in \mathbf{P}_{\mathscr{O}}(y)}\langle y-z,h\rangle + \frac{t}{2}|h|^2.
\end{equation*}
The $\limsup$ as $y\to x$, $t\rightarrow 0^{+}$ yields the generalized directional derivative (Definition \ref{def: general direc deriv})
\begin{equation}\label{eq:clarke-dir}
    \phi^{\circ}(x;h)\ \le\ \limsup_{y\to x}\ \sup_{z\in \mathbf{P}_{\mathscr{O}}(y)}\langle y-z,h\rangle.
\end{equation}
We claim that
\begin{equation}\label{eq:proj}
    \limsup_{y\to x} \sup_{z\in \mathbf{P}_{\mathscr{O}}(y)}\langle y-z,h\rangle
     \leq \sup_{z\in \mathbf{P}_{\mathscr{O}}(x)}\langle x-z,h\rangle.
\end{equation}
Pick any converging sequence $y_k\to x$. Since $\mathbf{P}_{\mathscr{O}}(y_k)$ is compact, for each $k$, choose an element $z_k\in \mathbf{P}_{\mathscr{O}}(y_k)$ such that
\begin{equation*}
    \langle y_k-z_k,h\rangle = \sup_{z\in \mathbf{P}_{\mathscr{O}}(y_k)}\langle y_k-z,h\rangle.
\end{equation*}
Since $y_k\to x$ and $ |y_k-z_k|=\operatorname{dist}(y_k,\mathscr{O})$ is bounded, the sequence $\{z_k\}_{k}$ is also bounded and admits a convergent subsequence (again denoted by $z_{k}$) such that $z_k\to z^*\in \mathbf{P}_{\mathscr{O}}(x)$. Then,
\begin{equation*}
    \limsup_{k\to\infty} \sup_{z\in \mathbf{P}_{\mathscr{O}}(y_k)}\langle y_k-z,h\rangle
    = \limsup_{k\to\infty} \, \langle y_k-z_k,h\rangle
    = \langle x-z^*,h\rangle
    \leq \sup_{z\in \mathbf{P}_{\mathscr{O}}(x)}\langle x-z,h\rangle,
\end{equation*}
which proves \eqref{eq:proj}. Combining \eqref{eq:clarke-dir}–\eqref{eq:proj} yields
\begin{equation*}
    \phi^{\circ}(x;h)\ \le\ \sup_{z\in \mathbf{P}_{\mathscr{O}}(x)}\langle x-z,h\rangle.
\end{equation*}
Recalling Definition \ref{def: general grad} and statement \ref{clarke 3} in Proposition \ref{prop: clarke}, 
for any $q\in\partial^C\phi(x)$ and any $h\in \mathbb{R}^{n}$, one has
\begin{equation*}
    \langle q,h\rangle \leq \phi^{\circ}(x;h) \leq \sup_{z\in \mathbf{P}_{\mathscr{O}}(x)}\langle x-z,h\rangle.
\end{equation*}
Taking the supremum over $|h|\leq 1$ gives
\begin{equation*}
    |q|=\sup_{|h|\leq 1}\langle q,h\rangle
    \leq \sup_{z\in \mathbf{P}_{\mathscr{O}}(x)}|x-z| = \operatorname{dist}(x,\mathscr{O}).
\end{equation*}

Now we prove \ref{third in lemma}. \\ 
Let $T>0$ be fixed. Let $K:=y([0, T])$ be the image set of $[0,T]$ by $y(\cdot)$, which is compact. From \ref{first in lemma}, we know that $\phi(\cdot)$ is Lipschitz on $K$ with some constant $L_K$, we have $\phi \circ y(\cdot)$ is absolutely continuous over $[0, T]$, thus $(\phi \circ y)^{\prime}(t)$ exists for a.a. $t\in[0, T]$.  Fix such  $t$ at which $y$ is differentiable and $(\phi\circ y)'(t)$ exists.
By Lebourg's mean value theorem  \cite[Chapter 2, Theorem 2.4, p.75]{clarke1998nonsmooth} for small $h$ with $t+h\in[0,T]$, there exist $z_h\in [y(t),y(t+h)]$ and $\zeta_h\in \partial^C\phi(z_h)$ such that
\begin{equation*}
    \phi(y(t+h))-\phi(y(t))=\left\langle\zeta_h, y(t+h)-y(t)\right\rangle.
\end{equation*}
Dividing by $h$ gives
\begin{equation*}
    \frac{\phi(y(t+h))-\phi(y(t))}{h}=\left\langle \zeta_h, \frac{y(t+h)-y(t)}{h}\right\rangle.
\end{equation*}
 From \ref{second in lemma} and by the $1$-Lipschitz property of the distance map $x \mapsto \operatorname{dist}(x, \mathscr{O})$, we have 
\begin{equation*}
    |\zeta_h| \leq \operatorname{dist}\left(z_h, \mathscr{O}\right) \leq \operatorname{dist}(y(t), \mathscr{O})+\left|z_h-y(t)\right|\leq \operatorname{dist}(y(t),\mathscr{O})+|y(t+h)-y(t)|.
\end{equation*}
Since $y(\cdot)$ is continuous, the sequence $\left\{\zeta_h\right\}_{h}$ is bounded and admits a convergent subsequence (again denoted by $\zeta_{h}$) such that $\zeta_h\to \zeta$ as $h \rightarrow 0$. 
We also have $z_h\to y(t)$ as  $h \rightarrow 0$. Since  $x\mapsto\partial^C\phi(x)$ has non-empty compact values and a closed graph in $\mathbb{R}^{n}$ as in Proposition \ref{prop: clarke}, it follows that $\zeta \in \partial^C \phi(y(t))$. Finally, since $y(\cdot)$ is differentiable at $t$, $\frac{y(t+h)-y(t)}{h} \rightarrow \dot{y}(t)$, and we have 
\begin{equation*}
    (\phi\circ y)'(t)
    =\lim_{h\to 0}\frac{\phi(y(t+h))-\phi(y(t))}{h}
    =\lim_{h\to 0}\Big\langle \zeta_h,\; \frac{y(t+h)-y(t)}{h}\Big\rangle
    =\big\langle \zeta,\; \dot y(t)\big\rangle
\end{equation*}
for a.a. $t\in[0,T]$.
\end{proof}

\section{Proofs in Control Theory}\label{appendix: control}

Let us recall the optimal control problem
\begin{equation}\label{OCP appendix}
\begin{aligned}
    u_{\lambda}(x) = \; 
    \inf\limits_{\alpha(\cdot)} \; & \int_{0}^{\infty}\; \left(\frac{1}{2}|\alpha(s)|^{2} + f(y_{x}^{\alpha}(s))\right)\, e^{-\lambda s}\;\text{d}s,\\
    & \text{subject  to  }\; \dot{y}_{x}^{\alpha}(s) = \alpha(s),\quad y_{x}^{\alpha}(0)=x\in\mathbb{R}^n\\
    & \text{and the controls } \alpha(\cdot):[0,\infty) \to B_{M} \text{ are measurable}, 
\end{aligned}
\end{equation}
where $B_{M}$ is the closed ball in $\mathbb{R}^{n}$ of radius $M$.
%===========================

\begin{proof}[\textbf{Proof of Proposition \ref{prop: min_equality}}]
It suffices to note that $\underline{f}\leq \frac{1}{2}|\alpha(s)|^{2} + f(y_{x}^{\alpha}(s))$ for all $\alpha(\cdot)$, $y_{x}^{\alpha}(\cdot)$, and for all $x$, in order to get the first statement of the result. 

To prove the second statement, one writes
\begin{equation*}
    u_{\lambda}(x) - \underline{f}/\lambda = \inf\limits_{\alpha(\cdot)\in \mathcal{A}} \int_{0}^{\infty}\left(\frac{1}{2}|\alpha(s)|^{2} + f(y_{x}^{\alpha}(s)) - \underline{f}\right)\,e^{-\lambda\,s}\,\dd s,
\end{equation*}
thus the integrand is non-negative. Therefore $u_{\lambda}(x) - \underline{f}/\lambda = 0$ if and only if there exists a control $\alpha(\cdot)$ such that $|\alpha(s)|^{2}=0$ and $ f(y_{x}^{\alpha}(s)) - \underline{f}=0$ for almost all $s$. When $\alpha(\cdot)\equiv 0$, the trajectory is constant $y_{x}^{\alpha}(\cdot)\equiv x$. Therefore, if $u_{\lambda}(x) - \underline{f}/\lambda = 0$ then $ f(x) - \underline{f}=0$, i.e. $f(x) = \min f$. Conversely, if $x\in \mathfrak{M}$, i.e. $ f(x) = \underline{f}$, then choosing $\alpha(\cdot)\equiv 0$ yields $u_{\lambda}(x) \leq \underline{f}/\lambda$, and together with the first statement of this proposition, we have $u_{\lambda}(x) =\underline{f}/\lambda$.
\end{proof}

%===========================

\begin{proof}[\textbf{Proof of Proposition \ref{prop: DPP}}]
It is analogous to \cite[Proposition III.2.5, p.102]{bardi1997optimal}. 
Let us recall the equation we wish to prove
\begin{equation}\label{DPP app}
    u_{\lambda}(x) = \inf\limits_{\alpha(\cdot)\in \mathcal{A}} \left\{\, \int_{0}^{t} \left(\frac{1}{2}|\alpha(s)|^{2} + f(y_{x}^{\alpha}(s))\right)\, e^{-\lambda s}\;\dd s \, + \, u_{\lambda}(y_{x}^{\alpha}(t))\,e^{-\lambda t} \,\right\}.
\end{equation}
Denote by $\omega(x)$ the right-hand side of \eqref{DPP app}, let $\ell(\alpha,x) := \frac{1}{2}|\alpha|^{2} + f(x)$. 

First we show that $u_{\lambda}(x) \geq \omega(x)$. 
For all $\alpha(\cdot)\in \mathcal{A}$, we have
\begin{equation*}
\begin{aligned}
    \mathscr{J}(x,\alpha(\cdot)) 
    & = \int_{0}^{t} \ell(\alpha(s), y_{x}^{\alpha}(s))\, e^{-\lambda s}\;\dd s + \int_{t}^{\infty} \ell(\alpha(s), y_{x}^{\alpha}(s))\, e^{-\lambda s}\;\dd s\\
    & = \int_{0}^{t} \ell(\alpha(s), y_{x}^{\alpha}(s))\, e^{-\lambda s}\;\dd s + \int_{0}^{\infty} \ell(\alpha(t+s), y_{x}^{\alpha}(t+s))\, e^{-\lambda (s+t)}\;\dd s\\
    & = \int_{0}^{t} \ell(\alpha(s), y_{x}^{\alpha}(s))\, e^{-\lambda s}\;\dd s + e^{-\lambda\,t} \mathscr{J}(y_{x}^{\alpha}(t), \alpha(t+\cdot))\\
    & \geq  \int_{0}^{t} \ell(\alpha(s), y_{x}^{\alpha}(s))\, e^{-\lambda s}\;\dd s +  u_{\lambda}(y_{x}^{\alpha}(t))\,e^{-\lambda\,t} \geq  \omega(x).
\end{aligned}
\end{equation*}
Taking the infimum over $\alpha(\cdot)\in\mathcal{A}$ yields $u_{\lambda}(x)\geq \omega(x)$.

Now we show that $u_{\lambda}(x) \leq \omega(x)$. 
Let us arbitrarily fix $\alpha(\cdot)\in \mathcal{A}$, and denote by $z := y_{x}^{\alpha}(t)$. Let $\varepsilon>0$ be given, and $\alpha^{\varepsilon}$ be an $\varepsilon$-optimal control for the problem starting from $z$, that is
\begin{equation*}
    \mathscr{J}(z, \alpha^{\varepsilon}(\cdot)) \leq u_{\lambda}(z) + \varepsilon.
\end{equation*}
Define the control $\overline{\alpha}(\cdot) \in \mathcal{A}$ such that
\begin{equation*}
\overline{\alpha}(s) := \left\{ \;
\begin{aligned}
    & \alpha(s) \quad \quad \quad \text{ if } s\leq t\\
    & \alpha^{\varepsilon}(s-t) \quad \text{ if } s>t.
\end{aligned}
\right.
\end{equation*}
Denote by $\overline{y}_{x}(\cdot):= y_{x}^{\overline{\alpha}}(\cdot)$ and $y_{z}^{\varepsilon}(\cdot) := y_{z}^{\alpha^{\varepsilon}}(\cdot)$ the corresponding trajectories. then we have
\begingroup\allowdisplaybreaks
\begin{align*}
    u_{\lambda}(x) 
    & \leq \mathscr{J}(x,\overline{\alpha}) = \int_{0}^{t} \ell(\overline{\alpha}(s), \overline{y}_{x}(s))\, e^{-\lambda s}\;\dd s + \int_{t}^{\infty} \ell(\overline{\alpha}(s), \overline{y}_{x}(s))\, e^{-\lambda s}\;\dd s\\
    & = \int_{0}^{t} \ell(\alpha(s), y_{x}^{\alpha}(s))\, e^{-\lambda s}\;\dd s + \int_{t}^{\infty} \ell(\alpha^{\varepsilon}(s-t), y_{x}^{\alpha^{\varepsilon}}(s))\, e^{-\lambda s}\;\dd s\\
    & =  \int_{0}^{t} \ell(\alpha(s), y^{\alpha}_{x}(s))\, e^{-\lambda s}\;\dd s + \int_{0}^{\infty} \ell(\alpha^{\varepsilon}(s), y_{x}^{\alpha^{\varepsilon}}(s+t))\, e^{-\lambda (t+s)}\;\dd s\\
    & =  \int_{0}^{t} \ell(\alpha(s), y^{\alpha}_{x}(s))\, e^{-\lambda s}\;\dd s + \; e^{-\lambda t}\;\int_{0}^{\infty} \ell(\alpha^{\varepsilon}(s), y_{z}^{\alpha^{\varepsilon}}(s))\, e^{-\lambda s}\;\dd s\\
    & =   \int_{0}^{t} \ell(\alpha(s), y^{\alpha}_{x}(s))\, e^{-\lambda s}\;\dd s +  \; e^{-\lambda t}\;\mathscr{J}(z,\alpha^{\varepsilon}(\cdot))\\
    & \leq  \int_{0}^{t} \ell(\alpha(s), y^{\alpha}_{x}(s))\, e^{-\lambda s}\;\dd s +  \; e^{-\lambda t}\;u_{\lambda}(z) + \varepsilon\\
    & = \int_{0}^{t} \ell(\alpha(s), y^{\alpha}_{x}(s))\, e^{-\lambda s}\;\dd s +  \; e^{-\lambda t}\;u_{\lambda}(y_{x}^{\alpha}(t)) + \varepsilon. 
\end{align*}
\endgroup
Since $\alpha(\cdot)\in \mathcal{A}$ and $\varepsilon$ are arbitrary, we get $u_{\lambda}(x)\leq \omega(x)$.
\end{proof}

%===========================

\begin{proof}[\textbf{Proof of Proposition \ref{Prop: h}}]
Let $x\in \mathbb{R}^{n}$ be fixed. For any $\alpha(\cdot)\in\mathcal{A}$, recall definition of the function $h:[0,\infty)\to \mathbb{R}$ 
\begin{equation*}
    h(t) := \int_{0}^{t} \left(\frac{1}{2}|\alpha(s)|^{2} + f(y_{x}^{\alpha}(s))\right)\, e^{-\lambda s}\;\dd s \, + \, u_{\lambda}(y_{x}^{\alpha}(t))\,e^{-\lambda t}.
\end{equation*}
Let $t_{1}\leq t_{2}$, and $\alpha(\cdot)\in\mathcal{A}$ be given. We show that $h(t_{1})\leq h(t_{2})$. 

Let us again use the notation $\ell(\alpha,x) := \frac{1}{2}|\alpha|^{2} + f(x)$. By definition, we have
\begin{equation*}
\begin{aligned}
    h(t_{2}) 
    & = \int_{0}^{t_{2}} \ell(\alpha(s),y_{x}^{\alpha}(s))\, e^{-\lambda s}\;\dd s \, + \, u_{\lambda}(y_{x}^{\alpha}(t_{2}))\,e^{-\lambda t_{2}}\\
    & = h(t_{1}) + \int_{t_{1}}^{t_{2}}  \ell(\alpha(s),y_{x}^{\alpha}(s))\, e^{-\lambda s}\;\dd s + \, u_{\lambda}(y_{x}^{\alpha}(t_{2}))\,e^{-\lambda t_{2}} - \, u_{\lambda}(y_{x}^{\alpha}(t_{1}))\,e^{-\lambda t_{1}}.
\end{aligned}
\end{equation*}
From \eqref{DPP}, it holds
\begin{equation*}
\begin{aligned}
    u_{\lambda}(y_{x}^{\alpha}(t_{1}))\,e^{-\lambda t_{1}}
    & = \inf\limits_{\alpha(\cdot)\in \mathcal{A}} \left\{ \,\int_{t_{1}}^{t_{2}}  \ell(\alpha(s),y_{x}^{\alpha}(s))\, e^{-\lambda s}\;\dd s + u_{\lambda}(y_{x}^{\alpha}(t_{2}))\,e^{-\lambda t_{2}}\ ,\right\}\\
     & \leq \int_{t_{1}}^{t_{2}}  \ell(\alpha(s),y_{x}^{\alpha}(s))\, e^{-\lambda s}\;\dd s + u_{\lambda}(y_{x}^{\alpha}(t_{2}))\,e^{-\lambda t_{2}}.
\end{aligned}
\end{equation*}
Hence $h(t_{2})\geq h(t_{1})$. 

For the second statement, it suffices to note that $\alpha(\cdot)$ is optimal if and only if the infimum is achieved in $\alpha(\cdot)$, that is
\begin{equation*}
\begin{aligned}
    u_{\lambda}(y_{x}^{\alpha}(t_{1}))\,e^{-\lambda t_{1}}
    & = \inf\limits_{\alpha(\cdot)\in \mathcal{A}} \left\{ \,\int_{t_{1}}^{t_{2}}  \ell(\alpha(s),y_{x}^{\alpha}(s))\, e^{-\lambda s}\;\dd s + u_{\lambda}(y_{x}^{\alpha}(t_{2}))\,e^{-\lambda t_{2}}\ ,\right\}\\
     & = \int_{t_{1}}^{t_{2}}  \ell(\alpha(s),y_{x}^{\alpha}(s))\, e^{-\lambda s}\;\dd s + u_{\lambda}(y_{x}^{\alpha}(t_{2}))\,e^{-\lambda t_{2}},
\end{aligned}
\end{equation*}
which ultimately yields $h(t_{2})=h(t_{1})$. The conclusion follows from the arbitrariness of $t_{1},t_{2}$.
\end{proof}

%===========================

\begin{proof}[\textbf{Proof of Theorem \ref{thm: opt cont}}]
Let $\alpha^{*}(\cdot)\in \mathcal{A}$ be an optimal control whose existence is shown in \cite[Theorem 2.5]{huang2025control}, and denote by $y_{x}^{*}(\cdot)$ the corresponding optimal trajectory, that is $\dot{y}^{*}_{x}(s) = \alpha^{*}(s)$ for almost every $s\geq 0$. Let us denote again by $h(\cdot)$ the function
\begin{equation*}
    h(t) := \int_{0}^{t} \ell(\alpha^{*}(s), y_{x}^{*}(s))\, e^{-\lambda s}\;\dd s \, + \, u_{\lambda}(y_{x}^{*}(t))\,e^{-\lambda t},
\end{equation*}
where  $\ell(\alpha,x) := \frac{1}{2}|\alpha|^{2} + f(x)$. 
We also know (from \cite[Lemma 2.1 (iii)]{huang2025control}) that $u_{\lambda}(\cdot)$ is Lipschitz, hence by Rademacher theorem it is almost everywhere differentiable, and so is the curve $y^{*}_{x}(\cdot)$. Therefore, Dini derivatives\footnote{See Definition \ref{def app Dini}. We shall use the short notation $\partial^{\pm}$ when we refer to both $\partial^{+}$ and $\partial^{-}$ indistinctively.} exist and are finite. 
Moreover Dini derivatives are compatible with composition along differential curves $u_{\lambda}\circ y^{*}_{x}(\cdot):= u_{\lambda}(y^{*}_{x}(\cdot))$, that is, the derivative along the curve $\partial^{\pm}u_{\lambda}\circ y^{*}_{x}(s;1)$ equals the directional derivative in the velocity direction $\partial^{\pm}u_{\lambda}(y^{*}_{x}(s);\dot{y}^{*}_{x}(s))$, and we have
\begin{equation*}
    \partial^{\pm}h(t;1) = e^{-\lambda t}\left(\, \partial^{\pm}u(y^{*}_{x}(t);\dot{y}^{*}_{x}(t)) + \ell(\alpha^{*}(t),y^{*}_{x}(t)) - \lambda\,u_{\lambda}(y^{*}_{x}(t))\,\right).
\end{equation*}
The second statement of Proposition \ref{Prop: h} ensures that $h(\cdot)$ is constant for all $t\geq 0$, and thanks to Lipschitz regularity of $u_{\lambda}(y^{*}_{x}(\cdot))$, we have 
\begin{itemize}
    \item $h'(t) = \partial^{+}h(t;1) = \partial^{-}h(t;1) = 0$ for almost every $t\geq 0$, hence
    \begin{equation}\label{eq: h const}
        \partial^{-}u_{\lambda}(y^{*}_{x}(t);\dot{y}_{x}(t)) = \lambda\,u_{\lambda}(y^{*}_{x}(t)) - \ell(\dot{y}^{*}_{x}(t),y^{*}_{x}(t))   \quad \text{ a.a. } t\geq0.
    \end{equation}
    \item $u_{\lambda}(y^{*}_{x}(\cdot))$ is almost everywhere differentiable (by Rademacher Theorem) which means $D^{+}u_{\lambda}(y^{*}_{x}(t)) = D^{-}u_{\lambda}(y^{*}_{x}(t)) = \{Du_{\lambda}(y^{*}_{x}(t))\}$ for almost every $t\geq 0$.
\end{itemize}
Denote $p = Du_{\lambda}(y^{*}_{x}(t))$. Since $p\in D^{+}u_{\lambda}(y^{*}_{x}(t))$ and $u_{\lambda}(\cdot)$ is a viscosity solution of the HJB equation $\;\lambda\, u_{\lambda}(x) + \frac{1}{2}|Du_{\lambda}(x)|^{2} = f(x),\;$ 
it is in particular a viscosity subsolution that is
\begin{equation}\label{eq: subsol}
    \lambda\, u_{\lambda}(y^{*}_{x}(t)) + \frac{1}{2}|p|^{2} \leq f(y^{*}_{x}(t)) \quad \text{ when } p\in D^{+}u_{\lambda}(y^{*}_{x}(t)).
\end{equation}
But $p$ is also in $D^{-}u_{\lambda}(y^{*}_{x}(t))$, which yields (see Lemma \ref{lem app D partial})
\begin{equation}\label{eq: ineq p}
    p \cdot \dot{y}^{*}_{x}(t) \leq \partial^{-}u_{\lambda}(y^{*}_{x}(t); \dot{y}^{*}_{x}(t)).
\end{equation}
Combining \eqref{eq: h const}, \eqref{eq: subsol}, and \eqref{eq: ineq p}, one obtains
\begin{equation*}
\begin{aligned}
    \eqref{eq: subsol} \Rightarrow\; \lambda\, u_{\lambda}(y^{*}_{x}(t)) + \frac{1}{2}|p|^{2} & \leq f(y^{*}_{x}(t))\\
    \Rightarrow\; \lambda\, u_{\lambda}(y^{*}_{x}(t)) + \frac{1}{2}|p|^{2} + \frac{1}{2}|\dot{y}^{*}_{x}(t)|^{2} &\leq \ell(\dot{y}^{*}_{x}(t),y^{*}_{x}(t))\\
    \Rightarrow\; \lambda\, u_{\lambda}(y^{*}_{x}(t)) - \ell(\dot{y}^{*}_{x}(t),y^{*}_{x}(t)) & \leq -\left(\frac{1}{2}|p|^{2} + \frac{1}{2}|\dot{y}^{*}_{x}(t)|^{2}\right)\\
    \eqref{eq: h const} \Rightarrow\; \partial^{-}u_{\lambda}(y^{*}_{x}(t);\dot{y}^{*}_{x}(t)) & \leq  -\left(\frac{1}{2}|p|^{2} + \frac{1}{2}|\dot{y}^{*}_{x}(t)|^{2}\right)\\
    \eqref{eq: ineq p}\Rightarrow\; p\cdot \dot{y}^{*}_{x}(t) & \leq  -\left(\frac{1}{2}|p|^{2} + \frac{1}{2}|\dot{y}^{*}_{x}(t)|^{2}\right)
\end{aligned}
\end{equation*}
hence $ \;\dot{y}^{*}_{x}(t) = -p \;$, which means $\alpha^{*}(t) = -Du_{\lambda}(y^{*}_{x}(t))$ for almost every $t\geq 0$.
\end{proof}

%===========================

\section{Proof of Lemma \ref{lem: ricc}}\label{app: ricc}

Let us recall the control problem 
\begin{equation}\label{OCP riccati app}
\begin{aligned}
    R_{\mathfrak{c}}(x,\mathscr{O}) := \; 
    \inf\limits_{\alpha(\cdot)} \; & \int_{0}^{\infty}\; \left(\frac{1}{2}|\alpha(s)|^{2} + \frac{\mathfrak{c}}{2}\,\operatorname{dist}(y_{x}^{\alpha}(s),\mathscr{O})^{2}\right)\, e^{-\lambda s}\;\dd s,\\
    & \; \dot{y}_{x}^{\alpha}(s) = \alpha(s),\quad y_{x}^{\alpha}(0)=x\in\mathbb{R}^n\\
    & \text{and the controls } \alpha(\cdot):[0,\infty) \to B_{M} \text{ are measurable}. 
\end{aligned}
\end{equation}

We shall now prove Lemma \ref{lem: ricc} which we recall is used to prove Proposition \ref{prop: sandwich u}.

\begin{proof}
Let $x\in \mathbb{R}^{n}$ be arbitrarily fixed. We proceed in two steps.

\textit{Step 1. (We show that $R_{\mathfrak{c}}(x,\mathscr{O}) \leq C\,\operatorname{dist}(x,\mathscr{O})^{2}$.) } \\
Let $z\in \mathscr{O}$ be such that $\operatorname{dist}(x,\mathscr{O}) = |x-z|$.  Consider the following dynamics
\begin{equation*}
    \dot{y}_{x}^{\alpha}(t)=\alpha(t):=-p(x-z)e^{-p\,t},\quad y_{x}(0)=x,
\end{equation*}
where $p$ is a constant that we will later make precise. This control is admissible, since the constant $M$ in the definition \eqref{admissible control set} of $\mathcal{A}$ can be made arbitrarily large as discussed at the end of Remark \ref{rem: bounded controls}. Therefore we have
\begin{equation*}
    y_{x}^{\alpha}(t) = x - (x-z)\big(1-e^{-p\,t}\big), \quad \text{ hence }\quad 
    y_{x}^{\alpha}(t) - z = (x-z)\,e^{ -p\, t}. 
\end{equation*}
A direct computation, with the admissible control $\alpha(t) = -p(x-z)e^{-p\,t}$ and noting that $\operatorname{dist}(y_{x}^{\alpha}(t),\mathscr{O})\leq |y_{x}(t) - z|$, gives
\begin{equation*}
\begin{aligned} 
    R_{\mathfrak{c}}(x,\mathscr{O})
    & \leq \int_0^{\infty} \left(\frac{1}{2}|\alpha(t)|^2+\frac{\mathfrak{c}}{2}\operatorname{dist}(y_{x}^{\alpha}(t),\mathscr{O})^2\right)\,e^{-\lambda t}\, \dd t \\
    & \leq \int_0^\infty \left(\frac{1}{2} p^{2}|x-z|^{2}e^{-2p\,t}+\frac{\mathfrak{c}}{2}|x - z|^{2}e^{-2p\,t}\right) \,e^{-\lambda t}\, \dd t\\
    & \leq \frac{1}{2}\left(p^2+\mathfrak{c}\right)|x - z|^2 \,\int_0^\infty e^{-(\lambda+2p) t}\, \dd t\\
    & =\frac{1}{2} \,\frac{p^2+\mathfrak{c}}{\lambda+2 p} \, |x-z|^2.
\end{aligned}
\end{equation*}
Choose now $p$ as the positive root of $p^{2} + \lambda\,p - \mathfrak{c}=0$, that is $p = (-\lambda + \sqrt{\lambda^{2} + 4\mathfrak{c}}\,)/2$. Therefore 
\begin{equation*}
    \frac{p^2+\mathfrak{c}}{\lambda+2p}=\frac{p^2+\left(p^2+\lambda p\right)}{\lambda+2p}=p.
\end{equation*}
Thus
\begin{equation*}
    R_{\mathfrak{c}}(x,\mathscr{O}) \leq \int_0^\infty \left(\frac12|\alpha(t)|^2+\frac{\mathfrak{c}}{2}\operatorname{dist}(y(t),\mathfrak{M})^2\right)\,e^{-\lambda t}\,\dd t
    \leq\frac{p}{2}|x-z|^2
    = \frac{p}{2} \operatorname{dist}(x,\mathscr{O})^2.
\end{equation*}

\textit{Step 2. (We show that $R_{\mathfrak{c}}(x,\mathscr{O}) \geq C\,\operatorname{dist}(x,\mathscr{O})^{2}$.) } \\
We shall rely on some results collected in Lemma \ref{lem:geom} in Appendix \ref{app: nonsmooth}. 

Define $\phi(\cdot):=\frac{1}{2}\operatorname{dist}(\,\cdot\,,\mathscr{O})^{2}$. Let $x\in \mathbb{R}^{n}$ and $\alpha(\cdot)\in \mathcal{A}$ be given, and $y_{x}^{\alpha}(\cdot)$ is the corresponding trajectory, i.e. $\dot{y}_{x}^{\alpha}(\cdot) = \alpha(\cdot)$ and $y_{x}(0)=x$. It is in particular absolutely continuous, and we can invoke statements \ref{second in lemma}-\ref{third in lemma} of Lemma \ref{lem:geom}, that is, there exists $q(t)\in \partial^{C}\phi(y_{x}^{\alpha}(t))$ such that for a.a. $t\geq0$
\begin{equation}\label{useful in proof of ricc}
    \frac{\dd}{\dd t}\phi(y_{x}^{\alpha}(t)) = \langle q(t),\dot{y}_{x}^{\alpha}(t) \rangle = \langle q(t), \alpha(t) \rangle \quad \text{ and } \quad \frac{1}{2}|q(t)|^{2}\leq \phi(y_{x}^{\alpha}(t)).
\end{equation}
Let $p\in \mathbb{R}^{n}$ be as in the previous step, i.e. the positive root of $p^{2} + \lambda\,p - \mathfrak{c}=0$. 
Therefore, it holds
\begin{equation*}
\begin{aligned}
    0 & \leq \frac{1}{2}|\alpha(t) + pq(t)|^{2} = \frac{1}{2}|\alpha(t)|^{2} + \frac{p^{2}}{2}|q(t)|^{2} + p\,\langle q(t), \alpha(t)\rangle \\
    & \leq \frac{1}{2}|\alpha(t)|^{2} + p^{2}\,\phi(y_{x}^{\alpha}(t)) + p\,\frac{\dd}{\dd t}\phi(y_{x}^{\alpha}(t))\\
    \Rightarrow\, -p\,\frac{\dd}{\dd t}\phi(y_{x}^{\alpha}(t)) & \leq \frac{1}{2}|\alpha(t)|^{2} + p^{2}\,\phi(y_{x}^{\alpha}(t)).
\end{aligned}
\end{equation*}
Adding $\lambda\,p\,\phi(y_{x}^{\alpha}(t))$ to both sides of this last inequality, and recalling the definition of $p$, one gets 
\begin{equation}\label{ineq in step 2 of proof Ricc}
    p\,\left(\lambda\,\phi(y_{x}^{\alpha}(t)) -\,\frac{\dd}{\dd t}\phi(y_{x}^{\alpha}(t)) \right) \leq \frac{1}{2}|\alpha(t)|^{2} + \mathfrak{c}\,\phi(y_{x}^{\alpha}(t)).
\end{equation}
Observe that 
\begin{equation*}
    \int_{0}^{T} e^{-\lambda \, t} \frac{\dd}{\dd t}\phi(y_{x}^{\alpha}(t))\,\dd t = e^{-\lambda\,T}\phi(y_{x}^{\alpha}(T)) - \phi(x) + \lambda \int_{0}^{T} e^{-\lambda \, t}\phi(y_{x}^{\alpha}(t))\,\dd t.
\end{equation*}
After rearranging the terms, one gets
\begin{equation*}
    \int_{0}^{T} \left( \lambda \,\phi(y_{x}^{\alpha}(t))\,\dd t -\frac{\dd}{\dd t}\phi(y_{x}^{\alpha}(t))\right)\,e^{-\lambda \, t}\,\dd t =  \phi(x) - e^{-\lambda\,T}\phi(y_{x}^{\alpha}(T)). 
\end{equation*}
Hence integrating \eqref{ineq in step 2 of proof Ricc} after multiplying both sides by $e^{-\lambda \,t}$ yields
\begin{equation}\label{nice ineq in proof Ricc}
    p\,\left(\phi(x) - e^{-\lambda\,T}\phi(y_{x}^{\alpha}(T))\right) \leq \int_{0}^{T} \left(\frac{1}{2}|\alpha(t)|^{2} + \mathfrak{c}\,\phi(y_{x}^{\alpha}(t))\right)\,e^{-\lambda\,t}\,\dd t.
\end{equation}
Let us define $\psi(t) := e^{-\lambda\,t}\phi(y_{x}^{\alpha}(t))$. Recalling \eqref{useful in proof of ricc}, we have
\begin{equation*}
\begin{aligned}
    |\psi'(t)| 
    & = e^{-\lambda\,t}\,\left| -\lambda \phi(y_{x}^{\alpha}(t))+ \frac{\dd }{\dd t}\phi(y_{x}^{\alpha}(t))\right|\\
    & = e^{-\lambda\,t}\,\left| -\lambda \phi(y_{x}^{\alpha}(t))+ \langle q(t), \alpha(t) \rangle \right|\\
    & \leq e^{-\lambda\,t}\,\left(\;\lambda \phi(y_{x}^{\alpha}(t)) + \frac{1}{2}|\alpha(t)|^{2} + \frac{1}{2} |q(t)|^{2}\;\right)\\
    & \leq e^{-\lambda\,t}\,\left(\;(1+\lambda) \phi(y_{x}^{\alpha}(t)) + \frac{1}{2}|\alpha(t)|^{2} \;\right).
\end{aligned}
\end{equation*}
Let $z\in \mathscr{O}$ be arbitrarily fixed. We have $\phi(y_{x}^{\alpha}(t)) \leq \frac{1}{2}|y_{x}^{\alpha}(t) - z|^{2} \leq |z|^{2} + |y_{x}^{\alpha}(t)|^{2}$. Recalling the dynamics of $y_{x}^{\alpha}(\cdot)$ and $\alpha(\cdot)$ being an element of $\mathcal{A}$ as defined in \eqref{admissible control set}, we have $|y_{x}^{\alpha}(t)|^{2} \leq 2|x|^{2} + 2M^{2}\,t^{2}$. Hence we have
\begin{equation*}
\begin{aligned}
    |\psi'(t)| 
    & \leq e^{-\lambda\,t}\left(C_{1} + C_{2}\,t^{2}\right)\quad \text{where } C_{1}:= (1+\lambda)\big(|z|^{2} + 2|x|^{2}\big)+\frac{1}{2}M^{2},\; C_{2}:=2(1+\lambda)M^{2}.
\end{aligned}
\end{equation*}
Therefore $\int_{0}^{\infty} |\psi'(t)|\dd t \leq \frac{C_{1}}{\lambda} + 2\frac{C_{2}}{\lambda^{3}}<\infty$, which guarantees that $\lim\limits_{t\to +\infty}\psi(t)=:\psi_{\infty}<\infty$. Similarly, we have
\begin{equation*}
\begin{aligned}
    |\psi(t)|
    & = e^{-\lambda\,t} |\phi(y_{x}^{\alpha}(t))| \; \leq \frac{1}{2}e^{-\lambda\,t} |y_{x}^{\alpha}(t) - z|^{2} \\
    & \leq (C_{3} + C_{4}\,t^{2})e^{-\lambda\,t} \quad \text{ where } C_{3} := \frac{1}{2}|z|^{2} + |x|^{2}, \; C_{4}:= M^{2} ,
\end{aligned}
\end{equation*}
and then $\int_{0}^{\infty} |\psi(t)|\dd t \leq \frac{C_{3}}{\lambda} + 2\frac{C_{4}}{\lambda^{3}}<\infty$. Together with the previous argument, this ensures that $\psi_{\infty}= 0$. Going back to \eqref{nice ineq in proof Ricc}, after sending $T\to +\infty$, the term $e^{-\lambda\,T}\phi(y_{x}^{\alpha}(T)) =\psi(T)\to 0$, and one gets
\begin{equation*}
    p\,\phi(x)  \leq \int_{0}^{\infty} \left(\frac{1}{2}|\alpha(t)|^{2} + \mathfrak{c}\,\phi(y_{x}^{\alpha}(t))\right)\,e^{-\lambda\,t}\,\dd t.
\end{equation*}
This last inequality being true for any $\alpha(\cdot)\in \mathcal{A}$, in particular when taking the infimum, one recovers the desired result
\begin{equation*}
    C\,\operatorname{dist}(x,\mathscr{O})^{2} = p\,\phi(x) \leq R_{\mathfrak{c}}(x,\mathscr{O}) ,
\end{equation*}
where $C=p/2$ and $p = (-\lambda + \sqrt{\lambda^{2} + 4\mathfrak{c}}\,)/2$.
\end{proof}

\section{Algorithms in practice}\label{app:algo}

We recall the main building blocks for the implementation of the API algorithm from \cite{alla2015efficient}.

\begin{algorithm}[H]\label{VI}
\SetAlgoLined
\KwData{Mesh $G$, $\Delta t$, initial guess $V^0$, tolerance $\varepsilon$.}
\While{$|V^{k+1}-V^k|\geq \varepsilon$}
{
    \ForAll{$x_i\in G$}
    {
    \begin{equation*}
        \text{Solve: } \;V_i^{k+1} =\min_{\alpha \in B_M} \left\{ e^{-\lambda \Delta t} \mathds{I}\left[V^k\right]\left(x_i+\alpha\,\Delta t\right)+ \frac{\Delta t}{2}|\alpha|^{2} \right\} + f(x_i)
    \end{equation*}
    }
    $k=k+1$
}
\caption{Value Iteration for infinite horizon optimal control \textbf{(VI)}}
\end{algorithm}

\begin{algorithm}[H]\label{PI}
\SetAlgoLined
\KwData{Mesh $G$, $\Delta t$, initial guess $V^0$, tolerance $\varepsilon$.}
\While{$|V^{k+1}-V^k|\geq \varepsilon$}
{
    \ForAll{$x_i\in G$}
    {
    \begin{equation*}
         V_i^k = \frac{\Delta t}{2}|\alpha_{i}^{k}|^{2} + f(x_{i}) + e^{-\lambda \Delta t} \,\mathds{I}\left[V^k\right]\left(x_i+\alpha_{i}^{k}\,\Delta t \right)
    \quad\text{\it (Policy evaluation)}\end{equation*}
    }
    \ForAll{$x_i\in G$}
    {
    \begin{equation*}
        \alpha_i^{k+1} =  \argmin\limits_{\alpha\in B_M} \left\{   e^{-\lambda \Delta t}\, \mathds{I}\left[V^k\right](x_i+\alpha\,\Delta t) + \frac{\Delta t}{2}|\alpha|^{2}\right\}\quad\text{\it (Policy improvement)}
    \end{equation*}
    }
    $k=k+1$
}
\caption{Policy Iteration for infinite horizon optimal control \textbf{(PI)}}
\end{algorithm}

\begin{algorithm}[H]
\SetAlgoLined
\KwData{Coarse mesh $G_c$ and $\Delta t_c$ , fine mesh $G_f$ and $\Delta t_f$, initial coarse guess $V_{c}^{0}$, coarse-mesh tolerance $\varepsilon_c$, fine-mesh tolerance $\varepsilon_f$.}
\Begin
{
    Coarse-mesh value iteration step: perform Algorithm \ref{VI}\\
    \KwIn{$G_c$, $\Delta t_c$, $V_{c}^{0}$, $\varepsilon_c$}
    \KwOut{$V_{c}^{*}$}
    \ForAll{$x_i\in G_f$}
    {
        \begin{equation*}
            V^{0}_{f}(x_i) =\mathds{I}_1[V^*_c](x_i), \quad A^{0}_{f}(x_i)  =\argmin\limits_{\alpha\in B_M}\;\left\{e^{-\lambda \Delta t}\, \mathds{I}_1 \left[V^{0}_{f}\right](x_i+\alpha)+ \frac{\Delta t}{2}|\alpha|^{2} \right\}.
        \end{equation*}
    }
    Fine-mesh policy iteration step: perform Algorithm \ref{PI}\\
    \KwIn{$G_f$, $\Delta t_f$, $V^0_f$, $A^0_f$, $\varepsilon_f$}
    \KwOut{$V^*_f$}
}
\caption{Accelerated Policy Iteration \textbf{(API)}}
\end{algorithm}

\bibliographystyle{amsplain}
\bibliography{bibliography}

\end{document}